\documentclass[a4paper,10pt]{amsart}

\usepackage{amsmath}
\usepackage{amsthm}
\usepackage{epsfig}
\usepackage{amssymb}
\usepackage{amsmath}
\usepackage{amssymb}
\usepackage{amsfonts}
\usepackage{amsxtra}

\usepackage[utf8]{inputenc}
\usepackage{bm}
\usepackage{bbm}
\usepackage{esint}

\usepackage{tikz-cd}
\usetikzlibrary{decorations.pathreplacing}
\usepackage{euscript}
\usepackage{mathrsfs}
\usepackage{color}

\allowdisplaybreaks
\usepackage{tikz-cd}
\usetikzlibrary{decorations.pathreplacing}

\usepackage{enumitem}
\setenumerate{label={\rm (\alph{*})}}

\usepackage{xcolor}
\colorlet{ColorPink}{red!30}
\usepackage{graphicx}

\usepackage[colorlinks=true, linktocpage=true, linkcolor=red!70!black,
 citecolor=green!50!black]{hyperref}

\newcommand{\R}{\mathbb R}
\newcommand{\N}{\mathbb N}

\providecommand{\degPol}{\deg_{\Pol}}

\newcommand{\dista}{\operatorname{dist}}

\DeclareMathOperator{\dist}{dist}

\newcommand{\dif}{\operatorname{d}\!}
\newcommand{\lebe}{\operatorname{L}}
\newcommand{\sobo}{\operatorname{W}}

\newcommand{\imag}{\operatorname{i}}
\newcommand{\locc}{\operatorname{loc}}

\newcommand{\hold}{\operatorname{C}}

\newcommand\norm[1]{\left\lVert#1\right\rVert}
\newcommand{\bv}{\operatorname{BV}}
\newcommand{\ball}{B}

\newcommand{\A}{\mathbb{A}}

\newcommand{\besov}{\operatorname{B}}
\newcommand{\Rn}{\mathbb{R}^{n}}

\newcommand{\diam}{\operatorname{diam}}

\newcommand{\spt}{\operatorname{spt}}

\newcommand{\trace}{\operatorname{Tr}}

\newcommand{\osc}{\operatorname{osc}}

\newcommand{\ch}{\operatorname{ch}}

\providecommand{\osc}{\operatorname{osc}}

\numberwithin{equation}{section}

\theoremstyle{theorem}

\newtheorem{theorem}{Theorem}[section]
\newtheorem{lemma}[theorem]{Lemma}
\newtheorem{proposition}[theorem]{Proposition}
\newtheorem{corollary}[theorem]{Corollary}

\newtheorem{definition}[theorem]{Definition}

\theoremstyle{remark}

\newtheorem{remark}[theorem]{Remark}
\newtheorem{example}[theorem]{Example}

\usepackage[notheorem,all]{paper_diening}

\providecommand{\Pol}{\mathscr{P}}

\providecommand{\opA}{\A}

\providecommand{\opT}{\mathbb{T}}

\providecommand{\T}{\mathbb{T}}

\providecommand{\W}{\mathbb{W}}
\def\cigar{{\rm cig}\,}

\renewcommand{\dashint}{\fint}

\begin{document}

\renewcommand{\baselinestretch}{1.05}

\title[Sharp trace and Korn inequalities]{Sharp Trace and Korn inequalities \\ for differential operators}
\author[L.~Diening]{Lars Diening}
\address[L.~Diening]{Faculty of Mathematics, University of Bielefeld, Universit\"{a}tsstra\ss e 25, 33615 Bielefeld, Germany}
\email{lars.diening@uni-bielefeld.de}
\author[F.~Gmeineder]{Franz Gmeineder}
\address[F.~Gmeineder]{Mathematical Institute, University of Bonn, Endenicher Allee 60, 53115 Bonn, Germany (\emph{expiring address})\\ \newline Department of Mathematics and Statistics, University of Konstanz, Universit\"{a}tsstrasse 10, 78457 Konstanz, Germany}
\email{fgmeined@math.uni-bonn.de, franz.gmeineder@uni-konstanz.de}
\thanks{\emph{Acknowledgments.} The authors are grateful to Herbert Koch for general discussions on the theme of the paper, and  to David Cruz-Uribe for pointing out the reference \cite{CruMarPer11} towards extrapolation in Lorentz spaces. F.G. acknowledges financial support by the Hausdorff Centre for Mathematics, Bonn, and the University of Bielefeld. The research of L.D. was supported by the DFG through the CRC 1283.}

\maketitle
\begin{abstract}
We establish sharp trace- and Korn-type inequalities that involve vectorial differential operators, the focus being on situations where global singular integral estimates are not available. Starting from a novel approach to sharp Besov boundary traces by Riesz potentials and oscillations that equally applies to $p=1$, a case difficult to be handled by  harmonic analysis techniques, we then classify  boundary trace- and Korn-type inequalities. For $p=1$ and so despite the failure of the Calder\'{o}n-Zygmund theory, we prove that sharp trace estimates can be systematically reduced to full $k$-th order gradient estimates. Moreover, for $1<p<\infty$, where sharp trace- yield Korn-type inequalities on smooth domains, we show for the basically optimal class of John domains that Korn-type inequalities persist -- even though the reduction to global Calder\'{o}n-Zygmund estimates by extension operators might not be possible.
\end{abstract}
\setcounter{tocdepth}{1}
\tableofcontents
\section{Introduction}
\subsection{Calder\'{o}n-Zygmund and $\lebe^{1}$-estimates}
A variety of partial differential equations or problems from the calculus of variations require to bound the $\lebe^{p}$-norms of full higher order gradients by those of given differential operators. To be more precise, let $k\in\mathbb{N}$, $k\geq 1$  and suppose that $\A$ is a homogeneous, $k$-th order linear, constant-coefficient differential operator on $\R^{n}$ with $n\geq 2$ between the two finite dimensional real vector spaces $V$ and $W$. Then $\A$ has a representation 
\begin{align}\label{eq:form}
\A=\sum_{\substack{\alpha\in\mathbb{N}_{0}^{n}\\|\alpha|=k}}\A_{\alpha}\partial^{\alpha}, 
\end{align}
where for each $\alpha\in\mathbb{N}_{0}^{n}$ with $|\alpha|=k$, $\A_{\alpha}\colon V\to W$ is a fixed linear map. Following the foundational work of \textsc{Calder\'{o}n \& Zygmund} \cite{CZ56}, if $1<p<\infty$, there exists a constant $c=c(\A,p)>0$ such that there holds
\begin{align}\label{eq:CZ}
\|D^{k}u\|_{\lebe^{p}(\R^{n})}\leq c\|\A u\|_{\lebe^{p}(\R^{n})}\qquad\text{for all $u\in\hold_{c}^{\infty}(\R^{n};V)$}
\end{align}
if and only if $\A$ is elliptic; \eqref{eq:CZ} is also referred to as \emph{Korn-type inequality} \cite{Korn09}. Here, in the most general sense, we call $\A$ \emph{elliptic} provided the Fourier symbol 
\begin{align}\label{eq:ellipticIntro}
\A[\xi] = \sum_{|\alpha|=k}\xi^{\alpha}\A_{\alpha}\colon V \to W\qquad\text{is injective for any $\xi\in\R^{n}\setminus\{0\}$}, 
\end{align}
cf. \textsc{H\"{o}rmander} \cite{Hoermander1} and \textsc{Spencer} \cite{Spencer}.  Due to a celebrated counterexample of \textsc{Ornstein} \cite{Ornstein} -- colloquially termed \emph{Ornstein's Non-Inequality} -- inequality \eqref{eq:CZ} does not persist for $p=1$ in general. Despite the failure of \eqref{eq:CZ} for $p=1$, one might still hope to find conditions on $\A$ such that lower order quantities depending on $u$ can be controlled in terms of $\|\A u\|_{\lebe^{1}}$. In the present paper we shall give a resolution of this matter for the boundary traces in the case of $k$-th order differential operators and \emph{all} $1\leq p<\infty$, cf. Theorem~\ref{thm:main} for $p=1$ (and Theorem \ref{thm:mainLp} for $1<p<\infty$). We then classify this result within situations where \emph{global} Calder\'{o}n-Zygmund estimates are not available, might it be because of $p=1$ (Theorem~\ref{thm:main1}) or the irregularity of the underlying domains for $1<p<\infty$ (Theorem~\ref{thm:Korn}).

\subsection{Limiting $\lebe^{1}$-trace estimates via Riesz potentials}\label{sec:IntroLimitingRiesz}
The general objective of limiting $\lebe^{1}$-estimates is to classify those operators $\A$ such that well-known inequalities for the full $k$-th order gradients hold with $D^{k}$ replaced by $\A$, too. Starting with \textsc{Bourgain \& Brezis} \cite{BouBre04,BouBre07}, the problem of bounding lower order norms against $\lebe^{1}$-norms of differential expressions $\A u$ has been intesively studied from various perspectives in recent years, cf. \textsc{Van Schaftingen} et al. \cite{VS04,Van13,VS14a,VS14b,GRVS} for a systematic treatment of Sobolev-type inequalities; also see \cite{Raita19}. By \textsc{Ornstein}'s Non-Inequality  \cite{Ornstein,KiKr16,KSW17}, none of these results can be directly inferred from the full $k$-th order gradient estimates. This equally applies to trace estimates, dealt with in the first order case by \textsc{Breit} and the authors \cite{BreDieGme20},  and in turn gives rise to Sobolev embeddings on domains, cf. \textsc{Raita} and the second author \cite{GR19}. 

The instrumental condition making boundary trace inequalities work is that of $\setC$\emph{-ellipticity}. Following \cite{Smith70,BreDieGme20,GR19}, we say that $\A$ of the form \eqref{eq:form} is \emph{$\mathbb{C}$-elliptic} provided 
\begin{align}\label{eq:Celliptic}
\text{$\A[\xi]\colon V+\imag\!V\to W+\imag\!W$ is injective for each}\;\xi\in\mathbb{C}^{n}\setminus\{0\}.
\end{align}
$\setC$-ellipticity is a stronger condition than \textsc{Van
  Schaftingen}'s cancellation condition \cite{Van13}.
So for instance the generalisation of \textsc{Gagliardo}'s trace inequality, 
\begin{align}\label{eq:firstordertraces}
  \|u\|_{\lebe^{1}(\partial\Omega)}\leq c\,(\|u\|_{\lebe^{1}(\Omega)}+\|\A u\|_{\lebe^{1}(\Omega)})\qquad\text{for all}\;u\in\hold^{\infty}(\overline{\Omega};V)
\end{align}
for first order $\setC$-elliptic operators holds if and only if $\A$ is $\mathbb{C}$-elliptic, see~\cite{BreDieGme20}; here $\Omega\subset\R^{n}$ is a bounded smooth domain. Since the
trace operator
$\trace_{\partial\Omega}\colon\sobo^{1,1}(\Omega;V)\to\lebe^{1}(\Omega;V)$
is surjective, \eqref{eq:firstordertraces} is optimal. However, if $k\geq 2$, by a result of \textsc{Uspenski\u{\i}} \cite{Uspenskii} there exists $c=c(n,k,V)>0$ such that for any open halfspace $\bbH\subset\R^{n}$  there holds 
\begin{align}\label{eq:Uspenskiimain}
\|u\|_{{\dot\besov}{_{1,1}^{k-1}}(\partial\bbH)}\leq c\|D^{k}u\|_{\lebe^{1}(\bbH)}\qquad\text{for all}\;u\in\hold_{c}^{\infty}(\overline{\mathbb{H}};V)
\end{align}
with the homogeneous Besov space ${\dot\besov}{_{1,1}^{k-1}}(\partial\bbH;V)(\subsetneq {\dot\sobo}{^{1,1}}(\partial\mathbb{H};V))$ (cf. Section~\ref{sec:Besov} for the definition), resulting in a \emph{surjective trace operator}. One may thus wonder whether $\mathbb{C}$-ellipticity is equally necessary and sufficient to yield the corresponding variant of the trace estimate \eqref{eq:Uspenskiimain}, also see \cite[Open Problem 4.8]{GR19}. As our first result, we answer this by
\begin{theorem}[Higher order traces I]\label{thm:main}
Let $k\in\mathbb{N}$, $k\geq 2$. Then the following are equivalent for a $k$-th order differential operator $\A$ of the form \eqref{eq:form}: 
\begin{enumerate}
\item\label{item:mainA} $\A$ is $\mathbb{C}$-elliptic in the sense of \eqref{eq:Celliptic}. 
\item\label{item:mainB} There exists $c=c(\A)>0$ such that for any open halfspace $\mathbb{H}$ there holds 
\begin{align}\label{eq:tracemaininequality0}
\|u\|_{{\dot\besov}{_{1,1}^{k-1}}(\partial\bbH)}\leq c\|\A u\|_{\lebe^{1}(\bbH)}\qquad\text{for all}\;u\in\hold_{c}^{\infty}(\overline{\bbH};V).
\end{align}
\end{enumerate}
\end{theorem}
The previous theorem particularly implies that, for open and bounded domains $\Omega$ with smooth boundary $\partial\Omega$ and $\mathbb{C}$-elliptic $\A$, there exists a surjective boundary trace operator $\trace_{\partial\Omega}\colon {\bv}{^{\A}}(\Omega)\to{\besov}{_{1,1}^{k-1}}(\partial\Omega;V)$ for \emph{functions of bounded $\A$-variation} (cf.~Theorem~\ref{rem:tracestrict}). 

Even though a lowering of the smoothness of the boundary is conceivable (and hereafter requires to work with higher order smoothness spaces on less  regular sets), our main focus in the previous theorem is not a possibly involved geometry or low regularity of boundaries, but the estimate \eqref{eq:tracemaininequality0} based on Riesz potential estimates; more general domains are addressed in Section~\ref{sec:introreduction} and~\ref{sec:introKorn}. Note that a  weaker estimate, replacing the ${\dot\besov}{_{1,1}^{k-1}}$-norm in \eqref{eq:tracemaininequality0} by the ${\dot\sobo}{^{k-1,1}}$-norm has been obtained by \textsc{Raita, Van Schaftingen} and the second author in
\cite{GRVS}; again note that ${\dot\besov}{_{1,1}^{k-1}}\subsetneq{\Dot\sobo}{^{k-1,1}}$ (cf.~\textsc{Brezis \& Ponce} \cite[Rem.~ A.1, p.~1238]{BrezisPonce}).

 As pointed out by \textsc{Leoni \& Tice} \cite{LeoniTice}, the only available approaches to the classical \textsc{Uspenski\u{\i}} trace estimate \eqref{eq:Uspenskiimain} for $p=1$ rely on the finite difference characterisation of Besov spaces; also see \textsc{Mironescu \& Russ} \cite{MirRus15}. By \textsc{Ornstein}'s Non-Inequality, controlling the requisite finite differences in terms of $\A u$ and \emph{not} $D^{k}u$ is far from clear, and so \eqref{eq:tracemaininequality0} must be approached differently. Toward Theorem~\ref{thm:main}, our line of action instead is to combine an oscillation characterisation of Besov spaces (cf. Lemma~\ref{lem:Besovosc}) and a novel sharp Riesz potential representation formula for $\mathbb{C}$-elliptic operators (cf. Proposition~\ref{prop:decomposition}).  This approach equally seems to be new in the case $1<p<\infty$ and even for $\A=D^{k}$, see Section~\ref{sec:BVAsection}.
 
By  \textsc{Aronszajn} \cite{Aro54} and \textsc{Smith} \cite{Smith70} (also see  \textsc{\Kalamajska{}} \cite{Kal93,Kal94}), it is well-known that for $\mathbb{C}$-elliptic operators $\A$ and $1<p<\infty$ one has the norm equivalence
\begin{align}\label{eq:KornIntroEquiv}
\|u\|_{\sobo^{k,p}(\Omega)}\simeq \|u\|_{\sobo^{k-1,p}(\Omega)}+\|\A u\|_{\lebe^{p}(\Omega)},\qquad u\in\sobo^{k,p}(\Omega;V)
\end{align}
for smoothly bounded $\Omega\subset\R^{n}$. 
As asserted in Theorem~\ref{thm:mainLp}, for such domains a corresponding sharp variant of the trace estimate \eqref{eq:tracemaininequality0} for $1<p<\infty$ does not only follow from but \emph{is equivalent} to \eqref{eq:KornIntroEquiv}, thereby providing a self-contained proof of \eqref{eq:KornIntroEquiv}. Here, the respective trace estimates yield the existence of a suitable extension operator and so allow to reduce \eqref{eq:KornIntroEquiv} to \emph{global} singular integral estimates on $\R^{n}$. In particular, the sharp trace estimates obtained by our Riesz potential approach imply \eqref{eq:KornIntroEquiv} and then vice versa by the sharp trace theorem for $\sobo^{k,p}$.

In view of the failure of Calder\'{o}n-Zygmund estimates on $\lebe^{1}$, inequality \eqref{eq:KornIntroEquiv} cannot be obtained for $p=1$ in general and, if $1<p<\infty$, the above approach only allows to conclude \eqref{eq:KornIntroEquiv} for suitable extension domains. We thus proceed by classifying the underlying mechanisms that allow to arrive at Theorem~\ref{thm:main} and the Korn-type estimate \eqref{eq:KornIntroEquiv} in situations where the above approaches are not necessarily available. 
 
\subsection{Reduction to $\sobo^{k,1}$-estimates despite failure of $\lebe^{1}$-CZ theory}\label{sec:introreduction}
Working from Theorem~\ref{thm:main}, $\mathbb{C}$-ellipticity yields the same trace estimates as known from the full $k$-th order gradient case. Thus one might wonder whether -- \emph{despite the failure of the Calder\'{o}n-Zygmund estimates for $p=1$} -- $\mathbb{C}$-ellipticity directly allows to
\begin{align}\label{eq:reduxMP}
\fbox{\text{reduce limiting $\lebe^{1}$-trace estimates to those available for $\sobo^{k,1}$.}}
\end{align}
As our second main result, Theorem~\ref{thm:main1} below and, more importantly, its proof, this is the case indeed. It is singled out as its proof seems to be \emph{the first instance in the literature which uses the reduction metaprinciple \eqref{eq:reduxMP}} to arrive at limiting $\lebe^{1}$-estimates for differential operators in the face of \textsc{Ornstein}'s Non-Inequality. The proof is naturally set up for $\mathrm{NTA}_{n-1}$-domains (cf. Section~\ref{sec:domains} for this terminology), leading to the following
\begin{theorem}[Higher order traces II]\label{thm:main1}
Let $k\in\mathbb{N}_{\geq 1}$. Then the following are equivalent for an operator $\A$ of the form \eqref{eq:form}: 
\begin{enumerate}
\item\label{item:mainA1} $\A$ is $\mathbb{C}$-elliptic in the sense of \eqref{eq:Celliptic}. 
\item\label{item:mainB1} For every open, bounded, $\mathrm{NTA}_{n-1}$-domain $\Omega\subset\R^{n}$, the space $\sobo^{\A,1}(\Omega):=\{u\in\lebe^{1}(\Omega;V)\colon\;\|u\|_{\lebe^{1}(\Omega)}+\|\A u\|_{\lebe^{1}(\Omega)}<\infty\}$ \emph{has the same trace space on $\partial\Omega$ as $\sobo^{k,1}(\Omega;V)$}. 
\end{enumerate}
\end{theorem}
Here, given a Banach space
$(\mathscr{X}(\partial\Omega),\|\cdot\|_{\mathscr{X}(\partial\Omega)})$ with
$\mathscr{X}(\partial\Omega)\subset\lebe_{\locc}^{1}(\partial\Omega;V)$,
we say that $\sobo^{k,1}(\Omega;V)$ \emph{has trace space
  $\mathscr{X}(\partial\Omega;V)$} provided there exists a linear and
bounded surjective trace operator
$\trace_{\partial\Omega}\colon\sobo^{k,1}(\Omega;V)\to\mathscr{X}(\partial\Omega;V)$
(and so satisfies $\trace_{\partial\Omega}(u)=u$
$\mathscr{H}^{n-1}$-a.e. on $\partial\Omega$ for all $u\in
\hold(\overline{\Omega};V)\cap\sobo^{k,1}(\Omega;V)$). Complementing
Theorem~\ref{thm:main}, Theorem~\ref{thm:main1} does not specify the
trace space of $\sobo^{\A,1}(\Omega)$ but rather asserts that it
\emph{equals that of $\sobo^{k,1}(\Omega;V)$}. Also note that we could replace $\sobo^{\A,1}$ by $\bv^{\A}$ in~\ref{item:mainB1}. 

Theorem~\ref{thm:main1} is approached by an advancement and strenghtening of a method employed in the first order case by the authors and \textsc{Breit} \cite{BreDieGme20}. The underlying key
novelty of the proof of Theorem~\ref{thm:main1} is that for a
$\mathbb{C}$-elliptic operator $\A$, the nullspace of $\A$ consists of
polynomials of a fixed degree and so, by the equivalence of all norms on
finite dimensional spaces, \textsc{Ornstein}'s Non-Inequality becomes
invisible on $\ker (\A)$. Section~\ref{sec:trac-reduct-class} is
devoted to the implementation of this strategy.

\subsection{Korn's inequality without global CZ-estimates}\label{sec:introKorn}
Returning to the Korn-type estimate~\eqref{eq:KornIntroEquiv} for $1<p<\infty$, the approach sketched in Section~\ref{sec:IntroLimitingRiesz} works by extensions and applying global Calder\'{o}n-Zygmund estimates on $\R^{n}$. Following the discussions in \cite{DieRuzSch10,Jiang17} for the symmetric gradient case $\A u=\frac{1}{2}(Du+Du^{\top})$, the natural geometric setup for such inequalities is given by \emph{John domains} (see Section~\ref{sec:domains} for this terminology). However, John domains need not even be extension domains for $\sobo^{k,p}$. Hence, in this situation, estimate \eqref{eq:KornIntroEquiv} cannot be established by means of global Calder\'{o}n-Zygmund estimates on $\R^{n}$. Interestingly, \eqref{eq:KornIntroEquiv} still persists for John domains and $\mathbb{C}$-elliptic operators:
\begin{theorem}[Korn for John]\label{thm:Korn}
Let $1<p<\infty$ and $\A$ be a differential operator of the form \eqref{eq:form} with $k\in\mathbb{N}_{\geq 1}$. Then the following are equivalent: 
\begin{enumerate}
\item\label{item:KornA} $\A$ is $\mathbb{C}$-elliptic. 
\item\label{item:KornB} For all open and bounded John domains $\Omega\subset\R^{n}$ we have 
\begin{align*}
\|u\|_{\sobo^{k,p}(\Omega)}\simeq \|u\|_{\lebe^{p}(\Omega)}+\|\A u\|_{\lebe^{p}(\Omega)},\qquad u\in\sobo^{k,p}(\Omega;V).
\end{align*}
\end{enumerate}
\end{theorem}
Theorem~\ref{thm:Korn} is established by a generalisation of a decomposition method introduced by  \textsc{\Ruzicka{}, Schumacher} and the first author \cite{DieRuzSch10} to the sharp class of operators for which the above Korn-type inequalities can hold at all. Theorem~\ref{thm:Korn} is in the spirit of Theorem~\ref{thm:main1}, however, note that John domains need not allow for a boundary trace operator; but if they do, Theorem~\ref{thm:Korn} immediately implies the equality of the trace space of $\sobo^{k,p}(\Omega;V)$ and that of the $\A$-Sobolev space $\sobo^{\A,p}(\Omega):=\{u\in\lebe^{p}(\Omega;V)\colon\;\|u\|_{\lebe^{p}(\Omega)}+\|\A u\|_{\lebe^{p}(\Omega)}<\infty\}$. Motivated by recent interest in Korn-type inequalities on more general space scales (cf.~\cite{BreDie12,Cianchi14,DieRuzSch10}), Theorem~\ref{thm:Korn} is obtained in Section~\ref{sec:Korn} as a special case of an $A_{p}$-weighted version, cf.~Theorem~\ref{thm:KornMain}. By \textsc{Rubio de Francia} extrapolation, the latter   implies a variety of Korn-type inequalities, so e.g. on Orlicz- or Lorentz spaces; see Section~\ref{sec:Kornextend}.

\subsection{Organisation of the paper}  
Section~\ref{sec:domains} collects the different notions of domains considered in the main part of the paper. As a technical novelty, Section~\ref{sec:impr-repr} provides an improved representation formula a l\'{a} \textsc{Smith} and \textsc{\Kalamajska{}}. This results in a family of Poincar\'{e}-type inequalities for John domains and Riesz potential-type inequalities which should be of independent interest but also display a crucial ingradient for the sequel. In Section~\ref{sec:traces-estimates-via}, after gathering background facts on Besov spaces, we establish Theorem~\ref{thm:main}. Sections~\ref{sec:trac-reduct-class} and~\ref{sec:Korn} are devoted to the proofs of Theorems~\ref{thm:main1} and~\ref{thm:Korn}, respectively. The appendix, Section~\ref{sec:appendix} gathers the proofs of auxiliary results. 

%\setcounter{tocdepth}{1}
%\tableofcontents
% We conclude with a discussion of certain boundary regularity results
% for elliptic systems with $\lebe^{1}$-data in
% Section~\ref{sec:Dirichlet}.
\vspace{0.35cm}
\begin{center}
\textsc{General Notation}
\end{center}
\vspace{0.15cm}
By $\Omega\subset\R^{n}$ we understand an open set throughout, and the open ball of radius $r$ centered at $x_{0}$ is denoted $\ball(x_{0},r):=\{y\in\R^{n}\colon\;|x_{0}-y|<r\}$. For $\ball=\ball(x_{0},r)$ and $\sigma>0$, we set $\sigma\ball :=\ball(x_{0},\sigma r)$. The $n$-dimensional Lebesgue and $(n-1)$-dimensional Hausdorff measures are denoted $\mathscr{L}^{n}$ and $\mathscr{H}^{n-1}$, respectively; for brevity, we sometimes use $|A|:=\mathscr{L}^{n}(A)$ and $\dif^{n-1}:=\dif\mathscr{H}^{n-1}$. Given a
finite dimensional real vector space $E$, we denote the finite positive Radon measures on $\Omega$ by $\mathscr{M}(\Omega)$, the finite $E$-valued Radon measures on $\Omega$ by $\mathscr{M}(\Omega;E)$ and, for $\mu\in\mathscr{M}(\Omega;E)$, $|\mu|(\Omega)$ its total variation. Given $\mu\in\mathscr{M}(\Omega)$ and a Borel subset $U$ of $\Omega$ with $\mu(U)>0$, we put for a $\mu$-integrable map $f\colon\Omega\to E$
\begin{align*}
\dashint_{U}f\dif\mu := \frac{1}{\mu(U)}\int_{U}f\dif \mu, 
\end{align*}
and the choice of $\mu$ will be clear from the context. If $\mu=\mathscr{L}^{n}$, we also use the shorthand $(f)_{U}:=\dashint_{U}f\dif\mathscr{L}^{n}$. The space of $E$-valued polynomials on $\R^{n}$ of degree at most $m\in\mathbb{N}_{0}$ is denoted $\mathscr{P}_{m}(\R^{n};E)$, and the homogeneous $E$-valued polynomials of degree $m$ by $\mathscr{P}_{m}^{h}(\R^{n};E)$; we also set $\mathscr{P}_{-m}(\R^{n};E)=\{0\}$ for $m\in\mathbb{N}$ and $\mathscr{P}(\R^{n};E):=\bigcup_{m\in\mathbb{N}_{0}}\mathscr{P}_{m}(\R^{n};E)$. The symbol $\odot^{m}(\R^{n};E)$ denotes the symmetric $m$-multilinear maps from $\R^{n}$ to $E$. By $c,C>0$ we denote generic constants which might change from line to line and shall only be specified provided their precise value is required. As such, we write $a\simeq b$ if there exist $c,C>0$ such that $ca\leq b \leq Ca$ and both $c,C$ do not depend in any essential way on $a$ and $b$. We also use $X\simeq Y$ for normed spaces to indicate that $X=Y$ with equivalence of norms, but no ambiguities will arise from this.

\section{Domains}\label{sec:domains}
In this section we collect the various geometric assumptions on domains $\Omega$ that arise throughout the main part of the paper. 

Let $\Omega\subset\R^{n}$ be given. We say that $\Omega$ satisfies the \emph{interior corkscrew condition} provided there exist $R>0$ and $M>1$ such that for all $x\in\partial\Omega$ and all $0<r<R$ there exists $y\in\Omega$ such that 
\begin{align*}
|x-y|<r\;\;\;\text{and}\;\;\;\ball\Big(y,\frac{r}{M}\Big)\subset\Omega. 
\end{align*}
Likewise, $\Omega$ satisfies the \emph{exterior corkscrew condition} provided $\R^{n}\setminus\Omega$ satisfies the interior corkskrew condition.

Next, for $x_{1},x_{2}\in\Omega$, we say that a sequence of balls $\ball_{1},...,\ball_{K}\subset\Omega$ is a \emph{Harnack chain from $x_{1}$ to $x_{2}$ of length $K\in\mathbb{N}$} provided $x_{1}\in\ball_{1}$, $x_{2}\in\ball_{K}$ and for all $1\leq j \leq K-1$ there holds $\ball_{j}\cap\ball_{j+1}\neq\emptyset$ and their radii satisfy $M^{-1}r(\ball_{j})<\dista(\ball_{j},\partial\Omega)<Mr(\ball_{j})$ for all $1\leq j \leq K$. The set $\Omega$ is then said to satisfy the \emph{interior Harnack chain condition} if, whenever $\varepsilon>0$ and $x_{1},x_{2}\in\Omega\cap\ball(\xi,\frac{r}{4})$ (for some $\xi\in\partial\Omega$ and $0<r<R$) satisfy $\dista(x_{j},\partial\Omega)>\varepsilon$, $j\in\{1,2\}$, together with $|x_{1}-x_{2}|<2^{l}\varepsilon$ for some $l\in\mathbb{N}$, then there exists a Harnack chain $\ball_{1},...,\ball_{K}$ from $x_{1}$ to $x_{2}$ of length $K\leq Ml$ such that 
$\mathrm{diam}(\ball_{j})\geq\frac{1}{M}\min\{\dista(x_{1},\partial\Omega),\dista(x_{2},\partial\Omega)\}$ for all $j=1,...,K$. 
\begin{definition}[NTA and $\mathrm{NTA}_{n-1}$-domains]\label{def:NTA}
An open and bounded set $\Omega\subset\R^{n}$ is \emph{non-tangentially accessible} (or \emph{NTA} for brevity) provided $\Omega$ satisfies \emph{the interior, the exterior corkscrew and the interior Harnack chain condition}, we refer to $R>0$ and $M>1$ from above as the \emph{NTA-parameters} of $\Omega$. Moreover, $\Omega$ is said to have \emph{$(n-1)$-Ahlfors regular boundary $\partial\Omega$} if there exist $R>0$ and $L>0$ such that for all $x\in\partial\Omega$ and $0<r<R$ there holds 
\begin{align*}
\frac{1}{L}r^{n-1}\leq \mathscr{H}^{n-1}(\partial\Omega\cap\ball(x,r))\leq Lr^{n-1}. 
\end{align*}
If $\Omega\subset\R^{n}$ is NTA and has $(n-1)$-Ahlfors regular boundary, we say that $\Omega$ is $\mathrm{NTA}_{n-1}$.
\end{definition}
See \cite{Aik12,HofMitTay10,JerisonKenig} for more information on NTA domains. Now let $\gamma \subset \Rn$ be a rectifiable path with endpoints $a$ and
$b$ and length $\abs{\gamma}$. Assuming that $\gamma\,:[0,\abs{\gamma}]
\to \Rn$ is parametrised by arclength, we define the {\em $\alpha$-cigar} with core~$\gamma$ and
parameter~$\alpha>0$ by
\begin{align*}
  \cigar(\gamma,\alpha) &:= \bigcup_{t\in [0, \abs{\gamma}]} 
    B \Big(\gamma(t), \frac{1}{\alpha} \min \set{ t, \abs{\gamma}-t}\Big)
   .
\end{align*}
\begin{definition}[John domains]
  \label{def:john}
  An open and bounded set $\Omega \subset \Rn$ is called an \emph{$\alpha$-John
    domain}, $\alpha>0$, if every pair of distinct points
  $a,b \in \Omega$ can be joined by a rectifiable path~$\gamma$ such
  that $\cigar(\gamma,\alpha) \subset \Omega$.  If the
  constant~$\alpha$ is not important, we just say that $\Omega$ is a
  \emph{John domain}.
\end{definition}
Introduced by \textsc{John} \cite{John61} and named after him by \textsc{Martio \& Sarvas} \cite{MarSar78}, John domains include sets with fractal boundary such as the Koch snowflake or slit domains. Moreover, they can be decomposed into a suitable set of balls or cubes
that satisfy a certain chain condition. In particular, by~\cite[Thm.~3.8]{DieRuzSch10} every bounded John
domain also satisfies the \emph{emanating chain condition} in the
sense of the following definition.
\begin{definition}[Emanating chain condition]
  \label{def:boman_chain}
  Let $\Omega \subset \Rn$ be an open and bounded set and let $\sigma_1 >1$, $\sigma_2 \geq 1$. Then we say that~$\Omega$ satisfies the {\em
    emanating chain condition} with constants~$\sigma_1$
  and~$\sigma_2$ if there exists a covering~$\mathcal{W} =
  \set{W_i\,:\, i \in \setN_0}$ of~$\Omega$ consisting of open balls 
  (or cubes) such that:
  \begin{enumerate}[label=(C\arabic{*})]
  \item \label{itm:C1} We have $\sigma_1 W \subset \Omega$ for all $W
    \in \mathcal{W}$ and $\sum_{W \in \mathcal{W}} \mathbbm{1}_{\sigma_1 W}
    \leq \sigma_2\,\mathbbm{1}_\Omega$ on $\Rn$.
  \item \label{itm:C2} For every $W_i \in \mathcal{W}$ there exists a
    {\em chain} of $W_{i,0}, W_{i,1}, \dots, W_{i,m_i}$ (pairwise
    different) from $\mathcal{W}$ such that $W_{i,0} = W_i$,
    $W_{i,m_i} = W_0$, and $W_{i,l_1} \subset \sigma_2 W_{i,l_2}$ for
    $0 \leq l_1 \leq l_2 \leq m_i$.  Moreover,
    $W_{i,l} \cap W_{i,l+1}$, $0 \leq l < m_i$, contains a ball
    $B_{i,l}$ such that
    $W_{i,l} \cup W_{i,l+1} \subset \sigma_2 B_{i,l}$.  The chain
    $W_{i,0}, \dots, W_{i,m_i}$ is called \emph{chain emanating from
      $W_i$}. The number $m_i\in \setN_0 $ is called the {\em length}
    of this chain.
  \item \label{itm:C3}
    The set $\set{i \in \setN_0 \,:\, W_i \cap K \neq \emptyset}$ is
    finite for every compact subset $K \subset \Omega$.
  \end{enumerate}
  The family~$\mathcal{W}$ is called the {\em chain-covering}
  of~$\Omega$.  The ball (or cube)~$W_0$ is called the {\em central ball (or
    cube)}, since every chain ends in~$W_0$.
\end{definition}
Clearly, all sets satisfying the emanating chain condition are connected. Moreover, if $\Omega$ satisfies the emanating chain condition with $\sigma_{1}>1$ and $\sigma_{2}\geq 1$, we have 
\begin{align}\label{eq:mutualradiusECCbound}
\mathrm{diam}(\Omega)\leq \sigma_{2}\mathrm{diam}(W_{0}),
\end{align}
since every $x\in\Omega$ is contained in some $W\in\mathcal{W}$ that can be connected with $W_{0}$ by a chain satisfying \ref{itm:C2} and so $W\subset\sigma_{2}W_{0}$. Moreover, note that whenever an open and bounded set satisfies the emanating chain condition with $\sigma_{1},\sigma_{2}\geq 1$, then by \cite[Thm.~3.8]{DieRuzSch10} it automatically satisfies the emanating chain condition with $\sigma_{1}>1$ and $\sigma_{2}\geq 1$ (for some possibly different chain covering); this is why we directly restrict ourselves to $\sigma_{1}>1$ in Definition~\ref{def:boman_chain}.
\begin{remark}[Choice of overlap balls]\label{rem:chooseBALLS}
By \cite[Rem.~3.15]{DieRuzSch10}, if $\Omega\subset\R^{n}$ satisfies the requirements of Definition~\ref{def:boman_chain}, the balls $\ball_{i,l}$ as in \ref{itm:C2} can be chosen to belong to a family $\mathscr{B}$ of balls such that $\sum_{\ball\in\mathscr{B}}\mathbbm{1}_{\ball}\leq \sigma_{2}\mathbbm{1}_{\Omega}$ as an estimate on $\R^{n}$. 
\end{remark}
\begin{figure} 
\begin{tikzpicture}[rotate=270,scale=0.08] 
\draw[fill=black!5!white,opacity=0.8](0, 0) .. controls (5.18756, -27) and (60.36, -18.40)
   .. (60, 40) .. controls (59.87, 59.88) and (57.33, 81.64203)
   .. (40, 90) .. controls (22.39, 98.48) and (4.72, 84.46368)
   .. (10, 70) .. controls (13.38, 60.71) and (26.35, 59.1351)
   .. (30, 50) .. controls (39.19409, 26.95) and (-4.10, 21.23)
   .. (0, 0);    
\fill[fill=black!10!red, fill opacity=0.25] (45,40) circle [radius=10]; %firstball
\fill[fill=black!10!red, fill opacity=0.25] (48.125,49.25) circle [radius=7.5]; %secondball
\fill[fill=black!10!red, fill opacity=0.25] (48.75,58) circle [radius=6]; %thirdball
\fill[fill=black!10!red, fill opacity=0.25] (47.5,64) circle [radius=4.5]; %fourthball
\fill[fill=black!10!red, fill opacity=0.25] (45.75,68.75) circle [radius=3.5]; %fifthball 
\fill[fill=black!10!red, fill opacity=0.25] (44.6,72.5) circle [radius=2.75]; %sixthball$
\fill[fill=black!10!red, fill opacity=0.25] (44.6,75.5) circle [radius=2.25];
\fill[fill=black!10!red, fill opacity=0.25] (45.1,78) circle [radius=2]; 
\fill[fill=black!10!red, fill opacity=0.25] (45.8,80) circle [radius=1.75]; 
\fill[fill=black!10!red, fill opacity=0.25] (45.8,81.9) circle [radius=1.5]; 
\fill[fill=black!10!red, fill opacity=0.25] (45.5,83.5) circle [radius=1.25];
\fill[fill=black!10!red, fill opacity=0.25] (44.8,84.85) circle [radius=1];
\fill[fill=black!10!red, fill opacity=0.25] (44.2,85.85) circle [radius=0.8];
\fill[fill=black!10!red, fill opacity=0.25] (43.7,86.5) circle [radius=0.6];
\fill[fill=black!10!red, fill opacity=0.25] (43.1,87) circle [radius=0.5];
\fill[fill=black!10!red, fill opacity=0.25] (41.1,49.25) circle [radius=7]; %secondballshifted
\fill[fill=black!10!red, fill opacity=0.25] (36.5,57) circle [radius=6]; %thirdballshifted
\fill[fill=black!10!red, fill opacity=0.25] (31,63) circle [radius=5.15]; %fourthballshifted
\fill[fill=black!10!red, fill opacity=0.25] (28,69) circle [radius=4];
\fill[fill=black!10!red, fill opacity=0.25] (27,74) circle [radius=3.5]; 
\fill[fill=black!10!red] (27,74) circle [radius=0.4];
\draw[black!10!red,dashed] (27,74) circle [radius=19];
\fill[fill=black!10!red, fill opacity=0.25] (28,78.25) circle [radius=3];
\fill[fill=black!10!red, fill opacity=0.25] (29.5,81.5) circle [radius=2.35]; 
\fill[fill=black!10!red, fill opacity=0.25] (31,83.75) circle [radius=2]; 
\fill[fill=black!10!red, fill opacity=0.25] (32.65,85.45) circle [radius=1.75]; 
\fill[fill=black!10!red, fill opacity=0.25] (34.1,86.45) circle [radius=1.28];
\fill[fill=black!10!red, fill opacity=0.25] (35.35,87.2) circle [radius=1.1];
\fill[fill=black!10!red, fill opacity=0.25] (36.35,87.8) circle [radius=0.85];
\fill[fill=black!10!red, fill opacity=0.25] (37.25,88.1) circle [radius=0.65];
\fill[fill=black!10!red, fill opacity=0.25] (37.85,88.25) circle [radius=0.5]; 
\draw[red] (37.85,88.25) -- (45,95); 
\node[red] at (48,104) {\large $W_{i}=W_{i,0}$};
%nowthethirdrowstarts
\fill[fill=black!10!red, fill opacity=0.25] (41,29.25) circle [radius=8.5]; %secondballshifted
\fill[fill=black!10!red, fill opacity=0.25] (36,21.25) circle [radius=8]; %secondballshifted 
\fill[fill=black!10!red, fill opacity=0.25] (29,15) circle [radius=7.25]; %secondballshifted 
\fill[fill=black!10!red, fill opacity=0.25] (22,10) circle [radius=6]; %secondballshifted 
\fill[fill=black!10!red, fill opacity=0.25] (16,6) circle [radius=5]; %secondballshifted 
\fill[fill=black!10!red, fill opacity=0.25] (11.35,4.15) circle [radius=3.7];
\fill[fill=black!10!red, fill opacity=0.25] (7.5,3.6) circle [radius=2.9];
\fill[fill=black!10!red, fill opacity=0.25] (4.8,3.5) circle [radius=2.1];
\fill[fill=black!10!red, fill opacity=0.25] (2.8,3.55) circle [radius=1.3];
\fill[fill=black!10!red, fill opacity=0.25] (1.6,3.65) circle [radius=0.9];%fourthrow
\fill[fill=black!10!red, fill opacity=0.25] (37,12.5) circle [radius=7.25]; %secondballshifted 
\fill[fill=black!10!red, fill opacity=0.25] (37,5) circle [radius=6]; %secondballshifted 
\fill[fill=black!10!red, fill opacity=0.25] (33,0) circle [radius=5]; %secondballshifted 
\fill[fill=black!10!red, fill opacity=0.25] (27.75,-3.75) circle [radius=3.7];
\fill[fill=black!10!red, fill opacity=0.25] (23.75,-5.55) circle [radius=2.8];
\fill[fill=black!10!red, fill opacity=0.25] (20.5,-7) circle [radius=2.1];
\fill[fill=black!10!red, fill opacity=0.25] (18.5,-8) circle [radius=1.4];
\fill[fill=black!10!red, fill opacity=0.25] (17,-8.75) circle [radius=1.1];
\fill[fill=black!10!red, fill opacity=0.25] (16,-9.4) circle [radius=1];
\fill[fill=black!10!red, fill opacity=0.25] (15.1,-10) circle [radius=0.9];
\fill[fill=black!10!red, fill opacity=0.25] (14.4,-10.6) circle [radius=0.8];
\draw (40,90) .. controls (45,80) and (35,75) .. (35,70);
\draw (6.125,-10) .. controls (6.75,-6.5) and (17.5,-6) .. (20,-1);
\node[black!45!white] at (17.5,78.5) {\LARGE $\Omega$};
\node[black] at (1.5,18.5) {\LARGE $\partial\Omega$};
\node[red] at (51,38) {\large $W_{0}$};
\end{tikzpicture}
\caption{The emanating chain condition for an exemplary non-extension domain $\Omega\subset\R^{2}$ with slits, illustrating the key property in \ref{itm:C2}: In any chain $W_{i,0}=W_{i},..., W_{i,m_{i}}=W_{0}$, any $W_{i,l}$ might be blown up by $\sigma_{2}$ such that $\sigma_{2}W_{i,l}$ contains \emph{all} $W_{i,j}$, $0\leq j\leq l$. }
\end{figure}
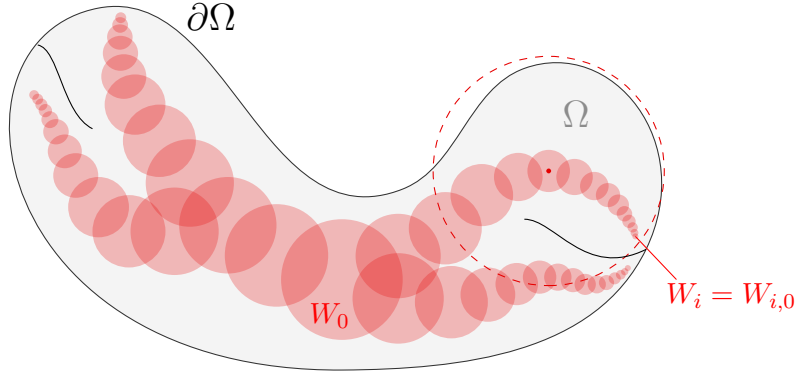
Following  \textsc{Boman} \cite{Boman82} and \textsc{Iwaniec \& Nolder} \cite{IwaNol85}, domains merely satisfying \ref{itm:C1} and \ref{itm:C2} are referred to as \emph{Boman (chain) domains}; also see \textsc{Hurri} \cite{Hurri88}. By a result due to \textsc{Buckley, Koskela \& Lu} \cite{BucKosLu96}, an open and bounded domain is a Boman chain domain if and only if it is John; moreover, by \cite{DieRuzSch10}, any open and bounded John domain satisfies the emanating chain condition. In total, for open and bounded sets $\Omega\subset\R^{n}$ these notions of sets are connected as follows (see, e.g., \cite[Sec.~3.7]{Aik12} and \cite[Thm.~3.8]{DieRuzSch10}): 
\begin{align}
\fbox{\text{$\text{NTA}_{n-1}\Longrightarrow \text{NTA} \Longrightarrow \text{John} =\text{Emanating Chain Condition}=\text{Boman}$}}
\end{align}

\section{The sharp representation theorem and Poincar\'{e}-type inequalities}
\label{sec:impr-repr}

In this section we derive a new representation formula in terms of
$\bbA u$. From this we conclude \Poincare{}-type estimates for balls
and John domains which should be of independent interest but also play a pivotal role in the proof of Theorems~\ref{thm:main}--\ref{thm:Korn}.

\subsection{The representation formula}
\label{sec:repr-form}
The natural framework for the results of this and the subsequent sections is given by the following variants of Sobolev- and BV-spaces adapted to differential operators $\A$ of the form \eqref{eq:form}: Given $1\leq p < \infty$ and an open  set~$\Omega \subset \Rn$, we define 
\begin{align}\label{eq:defBVA}
  \begin{split}
    &\sobo^{\A,p}(\Omega):=\{u\in\lebe^{p}(\Omega;V)\colon\;\|u\|_{\sobo^{\A,p}(\Omega)}:=\|u\|_{\lebe^{p}(\Omega)}+\|\A u\|_{\lebe^{p}(\Omega)}<\infty\},\\
    &\bv^{\A}(\Omega):=\{u\in\lebe^{1}(\Omega;V)\colon\;\|u\|_{\bv^{\A}(\Omega)}:=\|u\|_{\lebe^{1}(\Omega)}+|\A u|(\Omega)<\infty\},
  \end{split}
\end{align}
where $\|\A u\|_{\lebe^{p}(\Omega)}<\infty$ or $|\A u|(\Omega)<\infty$
have to be interpreted in the sense that the distributional
differential $\A u$ can be represented by an $\lebe^{p}(\Omega;W)$-map
or a finite $W$-valued Radon measure, respectively.  Moreover, we set ${\widetilde{\sobo}}{^{\A,p}}(\Omega):={\sobo}{^{\A,p}}(\Omega)\cap\sobo^{k-1,p}(\Omega;V)$ (in \cite{GR19} also denoted $\mathrm{V}^{\A,p}$) to be endowed with norm $\|u\|_{\widetilde{\sobo}{^{\A,p}}(\Omega)}:=\|u\|_{\sobo^{k-1,p}(\Omega)}+\|\A u\|_{\lebe^{p}(\Omega)}$. All of these spaces are Banach, and it is routine to check that $\hold^{\infty}(\Omega;V)\cap{\widetilde{\sobo}}{^{\A,p}}(\Omega)$
is dense in ${\widetilde{\sobo}}{^{\A,p}}(\Omega)$ and, if $\partial\Omega$ is sufficiently regular, we even have density of $\hold^{\infty}(\overline{\Omega};V)\cap{\widetilde{\sobo}}{^{\A,p}}(\Omega)$. For future reference, note that such density properties cannot expected for the norm on $\bv^{\A}$ (cf.~\textsc{Ambrosio} \cite[Chpt.~3]{AmbFusPal00} et al. for $\A=D$), and the correct substitute here is the $\A$-strict metric defined by
\begin{align}\label{eq:Astrictmetric}
  d_{\A}(u,v):=\|u-v\|_{\lebe^{1}(\Omega)}+\bigabs{|\A u|(\Omega)-|\A
  v|(\Omega)},\qquad u,v\in{\bv}{^{\A}}(\Omega). 
\end{align}
(Non-)examples of $\mathbb{C}$-elliptic operators are as follows; other relevant instances are gathered in \cite[Sec.~6.4]{Van13}: 
\begin{example}\label{ex:exfirst}
For each $k,n\in\mathbb{N}$ and $N\in\mathbb{N}$, the $k$-th order gradients $D^{k}$ acting on $u\colon\R^{n}\to\R^{N}$ are $\mathbb{C}$-elliptic. For $n\in\mathbb{N}$ and maps $u\colon\R^{n}\to\R^{n}$, the \emph{symmetric gradient operator} $\varepsilon(u):=\frac{1}{2}(Du+Du^{\top})$ is $\mathbb{C}$-elliptic for all $n\geq 1$, see \cite[Sec.~2.2]{BreDieGme20}.  For $n\geq 3$, the  \emph{trace-free symmetric gradient} $\varepsilon^{D}u:=\varepsilon(u)-\frac{1}{n}\mathrm{div}(u)E_{n\times n}$ is equally $\mathbb{C}$-elliptic (again, see \cite[Sec.~2.2]{BreDieGme20}). It is easy to see that the composition of two $\mathbb{C}$-elliptic operators is again $\mathbb{C}$-elliptic; an example that occurs frequently (so, e.g., in the regularity theory for elasticity and plasticity, is $\A = D\varepsilon^{D}$, see \textsc{Fuchs \& Seregin} \cite{FuchsSeregin}).
\end{example}
\begin{example}\label{ex:FuchsSeregin}
The trace-free symmetric gradient is not $\mathbb{C}$-elliptic for $n=2$. This can be seen by the fact (following from Proposition~\ref{prop:decomposition} below) that $\mathbb{C}$-elliptic operators have finite dimensional nullspace on connected sets. Elements of $\ker(\varepsilon^{D})$ in $n=2$ have a one-to-one-correspondence with holomorphic functions on $\mathbb{C}$, hence $\varepsilon^{D}$ is not $\mathbb{C}$-elliptic, see \cite[Sec.~2.2]{BreDieGme20}. On the other hand, if $n\geq 2$ and $\A$ is an elliptic operator on $\R^{n}$ from $V$ to $W$, then the $\A$-\emph{Laplacean} $\Delta_{\A}:=\A^{*}\A$ is \emph{not $\mathbb{C}$-elliptic}. Indeed, in this situation $\Delta_{\A}$ is elliptic as an operator on $\R^{n}$ from $V$ to $V$, and hence does not even satisfy \textsc{Van Schaftingen}'s cancellation condition:
\begin{align*}
\bigcap_{\xi\in\R^{n}\setminus\{0\}}\Delta_{\A}[\xi](V)=V\neq\{0\}. 
\end{align*}
As $\mathbb{C}$-ellipticity implies ellipticity and cancellation (cf. \cite[Lem.~3.2]{GR19}, \cite[Prop.~3.1]{GRVS}) for $n\geq 2$, the $\A$-Laplacean $\Delta_{\A}$ is not $\mathbb{C}$-elliptic. 
\end{example}
Towards Proposition~\ref{prop:decomposition}, we need to revise a decomposition as established in the pioneering works of \textsc{Smith}~\cite{Smith70} and
\textsc{\Kalamajska{}} \cite{Kal94}, cf. \eqref{eq:decomposition-Kal} below. Let $\Omega \subset \Rn$ be a domain that is star-shaped with respect to an open ball~$B$ and satisfies~$\diameter(\Omega) \leq c_0\, \diameter(B)$. The averaged Taylor polynomial $\bbT_m^B$ of order~$m$ then  is defined as 
in~\cite[Chpt.~1.1.10]{Maz85} and \cite[Chpt.~4]{BreSco08}:
Choose~$\omega \in \hold^\infty_{c}(B)$ with $\int_B \omega(y)\dif y=1$
and then define for $u\in\lebe_{\locc}^{1}(\Omega;V)$:
\begin{align}\label{eq:Taylor}
  \begin{split}
    \opT_m^{\ball}u(x) &:= \int_{B} u(y) \sum_{\abs{\alpha} \leq m} \frac{(-1)^{|\alpha|}}{\alpha!}
    \partial_{y}^\alpha \big(\omega(y) (x-y)^\alpha\big) \dif y. 
  \end{split}
\end{align}
The averaged Taylor polynomial has good approximation
properties and commutes in some sense with derivatives. In
particular, we have
\begin{align}\label{eq:TaylorSwift}
  \A\mathbb{T}_{m-1}^{B}u = \mathbb{T}_{m-k-1}^{B}\A u.
\end{align}
Moreover, it follows
from~\cite[Chpt.~1.1.10,~Thm.~1]{Maz85} that for all $x\in\Omega$
there holds
\begin{align}
\begin{split}
  \label{eq:Mazrep}
  u(x) & = \mathbb{T}_{m-1}^{\ball}u(x) +
         \sum_{|\alpha|=m}\int_{\Omega}K_{\alpha,B}(x,y)\partial^{\alpha}u(y)\dif
         y,\;\;\;\text{where}
  \\
 &\;\;\;\;\;\;\;\;\;\;\;K_{\alpha,B}(x,y)=\frac{(-1)^{m}m}{\alpha!}\frac{(y-x)^{\alpha}}{|x-y|^{n}}
  \int_{|x-y|}^{\infty}\omega\Big( x+t\frac{y-x}{|y-x|}\Big)t^{n-1}\dif t. 
  \end{split}
\end{align}
As established in \cite{Kal94}, the $\mathbb{C}$-ellipticity in combination with the Hilbert Nullstellensatz \cite[Lem.~4]{Kal94}
implies the existence of $m \in \setN$ and a
$\odot^{m}(\R^{n};V)$-valued differential operator $\mathbb{L}$ of order
$(m-k)$ such that $D^{m}=\mathbb{L}\A$. The constant~$m$ is not given
explicitly, but at least $\ker(\bbA) \subset \Pol_{m-1}(\Rn;V)$, so the \emph{polynomial degree} $\degPol(\A)$ of $\A$ satisfies 
\begin{align}\label{eq:deg-pol}
m \geq \degPol(\bbA) &:= \min \set{j \in \setN_0\,:\, \ker(\bbA) \subset \Pol_{j-1}(\Rn;V)}\geq k.
\end{align}
For future reference, we note that $\degPol(\A)<\infty$ not only follows but is in fact equivalent to $\mathbb{C}$-ellipticity, cf.~\cite[Prop.~3.1]{GR19} and \cite{Smith70}. Now, $D^m = \bbL \bbA$ and integrating by parts $(m-k)$-times yields for $u\in\hold^{\infty}(\overline{\Omega};V)$
\begin{align}
  \label{eq:decomposition-Kal}
  u(x) &= \mathbb{T}_{m-1}^Bu(x) + \int_{\Omega}\widetilde{\frK}^B_{\bbA}(x,y)\A u(y)\dif
         y\qquad\text{for all}\;x\in\Omega,
\end{align}
cf.~\cite[Theorem~4 (v)]{Kal94}, where the kernel 
$\widetilde{\frK}_\A^B\colon
(\Omega\times\Omega)\setminus\Delta \to\mathscr{L}(W;V)$ is smooth off
the diagonal $\Delta$, satisfies
$\widetilde{\frK}_\A^B(x,\cdot)=0$ near $\partial\Omega$
for each $x\in\Omega$ and, with $C=C(c_{0},\A)>0$
\begin{align}
  \label{eq:kernel-Kal}
  \abs{\partial_{x}^\alpha \partial_{y}^\beta
  \widetilde{\frK}_\A^B(x,y)} &\leq C\,\abs{x-y}^{k-n-                                                                              \abs{\alpha}-\abs{\beta}}           \qquad\text{for                                                                              all}\;x,y\in\Omega,\;x\neq y, 
\end{align}
for all $\alpha,\beta\in\mathbb{N}_{0}^{n}$ with
$|\alpha|+|\beta|\leq k$.

In general, the polynomial $\bbT_{m-1}^B u$ in  \eqref{eq:decomposition-Kal} does not belong to the nullspace of~$\bbA$. As a simple, yet effective improvement, we now modify this representation formula in such a way that we can replace $\bbT_{m-1}^B$ by a projection $\Pi_\A^B$ to $\ker(\bbA)$, thereby yielding a quick argument for Riesz- and Poincar\'{e}-type estimates in Sections~\ref{sec:poincare} and \ref{sec:traces-estimates-via}. 
\begin{proposition}[Representation formula]\label{prop:decomposition}
  Let $\Omega\subset\R^{n}$ be a bounded domain which is star-shaped
  with respect to an open ball $\ball\subset\Omega$ with
  $\diameter \Omega \leq c_0\, \diameter B$. Let $\A$ be a $k$-th order
  $\mathbb{C}$-elliptic differential operator of the form
  \eqref{eq:form}.
  % Let $m=m_\opA \in \setN_0$ be the smallest number
  % such that 
  % $\ker \opA \subset \Pol_{m-1}$.
  Then there exists an integral kernel $\frK_\A^B\,:\, (\Omega
  \times \Omega )\setminus\Delta\to \mathscr{L}(W;V)$
  and a linear projection~$\Pi_\A^B\,:\,\lebe^1(\Omega;V) \to \ker(\opA)$ such that
  the following hold:
  \begin{enumerate}[label={(K\arabic{*})},start=1]
  \item\label{item:decomp-kernel}
    $\frK_\A^B$ is smooth off the diagonal, and for $x,y
    \in \Omega$, $x\neq y$ and $\abs{\alpha}+\abs{\beta} \leq k$ we have
    \begin{align*}
      |\partial_{x}^{\alpha}\partial_{y}^{\beta}\frK_\A^B(x,y)|\leq
      C\,|x-y|^{k-n-|\alpha|-|\beta|}
    \end{align*}
    with a constant $C= C(\opA, c_0)>0$.
  \item\label{item:kernel-zero} $\frK_\A^B(x,\cdot)$ is zero
    near~$\partial \Omega$ for all $x \in \Omega$, i.e.,
    $\spt(\frK_\A^B(x,\cdot)) \compactsubset \Omega$.
  \item\label{item:decomp-repr} For any
    $u\in\hold^{\infty}(\Omega;V)$
    % there exists a
    % polynomial $\Pi_\A^B u\in\ker \A$ of some fixed maximal
    % degree $m=m(\A)$ such that
    and  $x\in\Omega$ there holds 
    \begin{align}\label{eq:decompositionmain}
      u(x) = \Pi_\A^Bu(x) + \int_{\Omega}\frK_\A^B(x,y)\A u(y)\dif y.
    \end{align}
  \item \label{item:decomp-projection}
    $\Pi_\A^B\colon\Pol_{\ell}(\R^{n};V)\to\Pol_{\ell}(\R^{n};V)\cap\ker(\A)$ for
    every~$\ell$ with $0 \leq \ell < \degPol(\bbA)$.
  \end{enumerate}
\end{proposition}
\begin{proof}
  Using translation and dilation we can assume without loss of
  generality that $B$ is the unit ball~$B(0,1)$.

  Let $m \in \setN_0$  and $\widetilde{\frK}_\bbA^B$ be such
  that~\eqref{eq:decomposition-Kal} holds. In particular, we have $m \geq
  \degPol (\bbA)$ and
  $\ker(\bbA) \subset \Pol_{m-1}(\R^{n};V)$.  Define
  $\mathcal{Z}_{\ell}:=\ker(\A)\cap\mathscr{P}_{\ell}^{h}(\R^{n};V)$ for $\ell=0,\dots, m-1$. We claim that
  \begin{align}
    \label{eq:kerA-decomp}
    \ker(\bbA) = \bigoplus_{\ell=0}^{m-1} \mathcal{Z}_\ell
  \end{align}
  as a direct sum. For this, let $q \in \ker(\A)$. Then $q=
  \sum_{\ell=0}^{m-1} q_\ell$ with $q_\ell \in
  \Pol_{\ell}^h(\Rn;V)$. Since $0 = \opA q = \sum_{\ell=0}^{m-1} \opA
  q_\ell$ and $\opA q_\ell \in \Pol_{\ell-k}^h(\Rn;W)$ by the homogeneity of $\A$, we obtain
  $q_{\ell} \in \ker(\opA)$ for $\ell=0,\dots,m-1$. This proves the decomposition~\eqref{eq:kerA-decomp}.

  For $\ell=0,\dots,m-1$, we then choose a linear subspace
  $\mathcal{W}_{\ell}\subset\mathscr{P}_{\ell}^{h}(\R^{n};V)$ such
  that
  $\mathscr{P}_{\ell}^{h}(\R^{n};V)=\mathcal{Z}_{\ell}\oplus\mathcal{W}_{\ell}$. Then
  we have
  \begin{align}\label{eq:directsumdecomp}
    \mathscr{P}_{m-1}(\R^{n};V)=\bigoplus_{\ell=0}^{m-1}\mathcal{Z}_{\ell}
    \oplus\bigoplus_{\ell=0}^{m-1}\mathcal{W}_{\ell} =: \ker(\opA) + \mathcal{W}.
  \end{align}
  Let $\psi_1,\dots, \psi_M$ denote a basis of~$\Pol_{m-1}(\Rn;V)$
  such that for each $\ell=0,...,m-1$, $\psi_{j_\ell},\dots, \psi_{j_{\ell+1}-1}$ is a basis of
  $\mathcal{Z}_{\ell}$ and $\psi_{j_m},\dots,\psi_M$ is a basis
  of~$\mathcal{W}$; here, $j_{0}=1$. Now, let $\psi_1^{*},...,\psi_{M}^{*}$ be a dual basis with
  respect to the $\lebe^2(B;V)$-inner product $\skp{\cdot}{\cdot}_{\lebe^{2}(\ball)}$,
  i.e., $\skp{\psi_l}{\psi_j^*}_{\lebe^{2}(B)} = \delta_{l,j}$ for all $j,l=1,...,M$.

  Based on \eqref{eq:directsumdecomp}, we define the projection $\widetilde{\Pi}_\A\colon
  \mathscr{P}_{m-1}(\R^{n};V)\to \ker(\A)$ by 
  \begin{align*}
    \widetilde{\Pi}_{\A}(q) &:=     \sum_{j=1}^{j_m-1} \skp{q}{\psi_j^*}_{\lebe^{2}(\ball)}\psi_j.
  \end{align*}
  We define $\Pi_\A^B := \widetilde{\Pi}_\bbA \bbT_{m-1}^B$; since $\bbT_{m-1}^B$ is a projection to $\Pol_{m-1}(\Rn;V)$, $\Pi_\A^B$ is a projection to~$\ker(\bbA)$. Moreover, by
  construction, $\widetilde{\Pi}_\bbA$ maps $\Pol_{\ell}^h(\Rn;V)$ onto
  $\Pol_{\ell}^h(\Rn;V)$ for $\ell=0,\dots,m-1$, thus $\Pi_\A^B$ also
  maps $\Pol_{\ell}^h(\Rn;V)$ onto $\Pol_{\ell}^h(\Rn;V)$ for
  $\ell=0,\dots,m-1$.  This proves~\ref{item:decomp-projection}. Moreover, for all
  $u\in\hold^{\infty}(\overline{\Omega};V)$ we have
  \begin{align}\label{eq:rewritedifference}
    \begin{split}
      u - \Pi_\A^B u &= u-\widetilde{\Pi}_{\A}\T_{m-1}^{\ball}u
      \\
      & = (u-\T_{m-1}^{\ball}u) + (\identity -
      \widetilde{\Pi}_\bbA)\T_{m-1}^{\ball}u.
    \end{split}
  \end{align}
  Note that $\langle \opA \cdot, \opA \cdot\rangle_{\lebe^{2}(\ball)}$ is an inner product on
  $\mathcal{W}$ as 
  $\Pol_{m-1}(\Rn;V) = \ker(\bbA) \oplus \mathcal{W}$. Thus, for each
  $\psi_\ell^*$ with $j_m \leq \ell \leq M$ we find $\xi_\ell \in
  \mathcal{W}$ such that
  \begin{align*}
    \skp{w}{\psi_\ell^*}_{\lebe^{2}(B)} &= \skp{\opA w}{\opA \xi_\ell}_{\lebe^{2}(B)} \qquad \text{for all $w \in
                         \mathcal{W}$ and $j_m \leq \ell \leq M$}.
  \end{align*}
  If $z \in \ker(\opA)$, then
    $\skp{z}{\psi_\ell^*}_{\lebe^{2}(B)} = 0 = \skp{\opA z}{ \opA \xi_\ell}_{\lebe^{2}(B)}$ for all  $j_m \leq \ell \leq M$. Thus,
  % \begin{align*}
  %   \skp{z}{\psi_\ell}_B &= 0 = \skp{\opA z}{ \opA \psi_\ell}_B\qquad \text{for all $w \in
  %                        \mathcal{Z}$ and $j_m \leq \ell \leq M$}.
  % \end{align*}
  % This proves
  \begin{align}
    \label{eq:innerprodct}
    \skp{q}{\psi_\ell^*}_{\lebe^{2}(B)} &= \skp{\opA q}{\opA \xi_\ell}_{\lebe^{2}(B)}\;\;\text{for all $q \in
                         \Pol_{m-1}(\Rn,V)$ and $j_m \leq \ell \leq M$}.
  \end{align}
  In conclusion,
  \begin{align}
    \begin{aligned}
      (\identity &- \widetilde{\Pi}_\bbA )\T_{m-1}^{\ball}u = \sum_{j=j_m}^M
      \skp{\bbT_{m-1}^B u}{\psi_j^*}_{\lebe^{2}(\ball)} \psi_j
      \\
      &= \sum_{j=j_m}^M \skp{\bbA \bbT_{m-1}^B u}{\bbA \xi_j}_{\lebe^{2}(\ball)} \psi_j \stackrel{\eqref{eq:TaylorSwift}}{=} \sum_{j=j_m}^M
      \skp{\bbT_{m-1-k}^B \bbA u}{\bbA \xi_j}_{\lebe^{2}(\ball)} \psi_j .
    \end{aligned}
  \end{align}
  Hence, by the definition~\eqref{eq:Taylor} of $\bbT_{m-1}^B$, we obtain for all
  $x\in\Omega$
  \begin{align}\label{eq:rewriteremainder}
    \begin{aligned}
      \lefteqn{\big((\identity - \widetilde{\Pi}_\bbA)\T_{m-1}^{\ball}u\big)(x)} \qquad&
      \\
      &= \int_B   \int_B \sum_{\substack{\abs{\alpha} \leq m-k-1\\ j_m 
          \leq j \leq M }} \!\!\!\!\!\! \frac{(-1)^{|\alpha|}}{\alpha!}
      \partial_{z}^\alpha \big(\omega(z)
      (y-z)^\alpha\big) \psi_{j}(x)(\A\xi_{j}(y))^{\top} \dif y \A
      u(z)\dif z
      \\
      &=: \int_{B}\overline{\frK}_\A^B(x,z)\A u(z)\dif z,
    \end{aligned}
  \end{align}
  where $\overline{\frK}_\A^B$ is defined in the obvious
  manner. Recalling
  \eqref{eq:decomposition-Kal}, put
  $\frK_\A^B:=\overline{\frK}_\A^B+\widetilde{\frK}_\A^B$.
  Then by \eqref{eq:decomposition-Kal}, \eqref{eq:rewritedifference}
  and \eqref{eq:rewriteremainder}, ~\eqref{eq:decompositionmain} follows.  Note that
  $\overline{\frK}_\bbA^B(x,\cdot)$ is supported on
  $\spt(\omega) \compactsubset B \subset \Omega$ and is therefore
  zero near~$\partial\Omega$. On the other hand, by ~\eqref{eq:decomposition-Kal}ff. we obtain that 
  $\widetilde{\frK}_\bbA^B(x,\cdot)$ is zero near~$\partial
  \Omega$. This proves that $\frK_\bbA^B(x,\cdot)$ is zero
  near~$\partial \Omega$ and so~\ref{item:kernel-zero} follows. It remains to establish the estimates for $\frK_\opA^B$
  in~\ref{item:decomp-kernel}; however, the estimates
  for~$\widetilde{\frK}_\bbA^B$ are much less singular and follow
  easily using $\diam(\Omega) \leq c_0 \diameter(B)$ and our
  simplifying assumption $0 \in B$.  The proof is complete.
\end{proof}
\begin{remark}
  It is possible to extend~\ref{item:decomp-projection} of
  Proposition~\ref{prop:decomposition} to all $0 \leq \ell < m$ for
  any fixed $m \geq \degPol (\bbA)$. For this one has to start
  with~\eqref{eq:decomposition-Kal} with a possible enlarged~$m$. Moreover, if $B$ in Proposition~\ref{prop:decomposition} is centered at zero
  or $\Omega$ at least contains zero, then it is possible to
  construct~$\Pi_\A^B$ such
  $\Pi_\A^B\colon\Pol_{\ell}^{h}(\R^{n};V)\to\Pol_{\ell}^{h}(\R^{n};V)\cap\ker
  (\A)$ for every~$\ell$ with $0 \leq \ell < \degPol \bbA$, where
  $\Pol_\ell^h$ is the linear space of homogeneous polynomials of
order~$\ell$. Note that not increasing the degree of polynomials when projecting \emph{cannot} be improved to the commuting-type relation $\partial^{\alpha}\mathrm{proj}_{\ker(\A)}u=\mathrm{proj}_{\partial^{\alpha}\ker(\A)}(\partial^{\alpha}u)$ for $u\in\mathscr{P}_{m-1}(\R^{n};V)$ for suitable projections onto $\ker(\A)$ or $\partial^{\alpha}\ker(\A)$, respectively.

This can be seen by means of the symmetric gradient $\A=\varepsilon$ (cf.~Example~\ref{ex:exfirst}) in $n=2$ dimensions. In this case,  the nullspace of $\A$ is given by the three-dimensional space of \emph{rigid deformations}
\begin{align*}
\ker(\A)=\left\{(x,y)\mapsto \Big(\begin{matrix} -ay+b \\ ax+c\end{matrix}\Big)\colon\;a,b,c\in\R  \right\}, 
\end{align*}
cf.~\textsc{Reshetnyak} \cite{Resh70}. Denote $p(x,y):=(0,x)^{\top}$. Should the commuting relation $\partial^{\alpha}\mathrm{proj}_{\ker(\A)}p=\mathrm{proj}_{\partial^{\alpha}\ker(\A)}(\partial^{\alpha}p)$ hold, then we have $\partial_{x}\mathrm{proj}_{\ker(\A)}p(x,y)=(0,1)^{\top}$ and $\partial_{y}\mathrm{proj}_{\ker(\A)}p(x,y)=(0,0)^{\top}$. Thus, $\mathrm{proj}_{\ker(\A)}p(x,y)=(b,x+c)^{\top}$ for certain $b,c\in\R$, but since $(x,y)\mapsto (b,x+c)^{\top}$ does not belong to $\ker(\A)$, such a projection $\mathrm{proj}_{\ker(\A)}$ onto $\ker(\A)$ cannot exist.
\end{remark}
\begin{lemma}
  \label{lem:PiB-estimate}
  Under the assumptions of Proposition~\ref{prop:decomposition}, 
  $\Pi_\A^B$ extends to a bounded linear operator
  $\Pi_\A^B\colon\lebe^{1}(\Omega;V)\to\ker(\A)$. Moreover, for all $1
  \leq p <\infty$ and all $u\in\lebe^{1}(\Omega;V)$ we have, with $r(B)$ denoting the radius of $\ball$,
  \begin{align}\label{eq:inverseFull}
  \begin{split}
    &\bigg(\dashint_\Omega \abs{\Pi_\A^B u}^p \dif x\bigg)^{\frac 1p} \leq
    \norm{\Pi_\A^B u}_{\lebe^\infty(\Omega)}
    \leq 
      C\, \dashint_B \abs{u}\!\dif x \leq 
      C\, \bigg(\dashint_B \abs{u}^p\,\dif x\bigg)^{\frac 1p},\\ 
      & \frac{1}{C}\Big(\sum_{|\alpha|\leq k}r(B)^{|\alpha|p}\dashint_{\ball}|\partial^{\alpha}\Pi_{\A}^{B}u|^{p}\dif x \Big)^{\frac{1}{p}}\leq  \dashint_{\ball}|\Pi_{\A}^{\ball}u|\dif x \leq C \dashint_{\ball}|u|\dif x,
  \end{split}
  \end{align}
  with a constant $C=C(\bbA,c_0)>0$.
\end{lemma}
\begin{proof}
 Since $\Pi_\A^B = \widetilde{\Pi}_{\bbA}\mathbb{T}_{m-1}^B$ is a projection to a finite dimensional vector space, this follows immediately by inverse estimates for polynomials using $\diameter(\Omega) \leq c_0 \diameter(B)$ and scaling.
\end{proof}
By the support and growth properties of the integral kernel in Proposition~\ref{prop:decomposition}, convolving $u\in\sobo^{\A,p}(\Omega)$ or $u\in\bv^{\A}(\Omega)$ with smooth bumps and passing to the limit yields
\begin{corollary}
  \label{cor:repr-estimate}
  Under the assumptions of Proposition~\ref{prop:decomposition}, let
  $u \in \sobo^{\bbA,p}(\Omega)$ for some $1\leq p \leq \infty$. Then
  there holds for all $\abs{\alpha} < k$ 
  \begin{align}
    \label{eq:repr-estimate-1}
    \partial^\alpha u = \partial^\alpha \Pi_\A^Bu +
    \int_{\Omega}(\partial^\alpha_x \frK_\A^B)(\cdot,y)\A
    u(y)\dif y\qquad \text{$\mathscr{L}^{n}$-almost everywhere}.
  \end{align}
  Moreover, for all $\ell$ with $0 \leq \ell < k$ we have 
  \begin{align}
    \label{eq:repr-estimate-20}
    \abs{D^\ell (u - \Pi_\A^Bu)} &\leq
                                                                C\, 
    \int_{\Omega} \abs{\cdot-y}^{-n+k-\ell}\abs{\A
    u(y)}\dif y\qquad \text{$\mathscr{L}^{n}$-almost everywhere}.
  \end{align}
  For $u \in \setBV^{\opA}(\Omega)$, \eqref{eq:repr-estimate-1} and \eqref{eq:repr-estimate-20} remain valid upon replacing $\bbA
  u(y)\dif y$ by $\dif \bbA u(y)$ and $\abs{\bbA
  u(y)}\dif y$ by $\dif \abs{\bbA u}(y)$, respectively.
\end{corollary}
%\begin{proof}
%  The estimate~\eqref{eq:repr-estimate-2} follows easily for all $u
%  \in \hold^\infty(\Omega;V)$ by Proposition~\ref{prop:decomposition}. Now, smooth approximation by convolution with a smooth bump yields \eqref{eq:repr-estimate-2} for $u \in \sobo^{\bbA,p}(\Omega)$ or $u
%  \in \setBV^{\bbA}(\Omega)$. At this step we can use
%  standard convolution, since $\frK_\bbA^B(x,\cdot)$ is zero
%  near~$\partial \Omega$. Note also, that
%  $\abs{\cdot-y}^{-n+k-\ell} \in \lebe^1(\Omega)$ for
%  $\ell < k$.
%\end{proof}

\subsection{\Poincare{}-type inequalities}
\label{sec:poincare}
Throughout this section, let $\A$ be a $k$-th order $\mathbb{C}$-elliptic differential operator of the form \eqref{eq:form}. 
The present section is devoted to the proof of the following Poincar\'{e}-type inequality for the vast class of John domains, announced in \cite[Rem.~2.4]{DieGme20}, which shall turn out a crucial tool for the following sections.
\begin{theorem}[\Poincare{} for John]
  \label{thm:poincare-john}
  Let $\Omega\subset\R^{n}$ be an open and bounded $\alpha$-John domain or an open and bounded domain satisfying the emanating chain condition with
  constants~$\sigma_1$ and $\sigma_2$ and central ball~$B$, respectively. Then for all $u \in \sobo^{\bbA,p}(\Omega)$ with
  $1 \leq p \leq \infty$ and $\ell\in\{0,\dots,k-1\}$ there holds
  \begin{align}
    \label{eq:poincare-star}
    \bigg( \dashint_\Omega \abs{D^\ell (u-
    \Pi_\A^Bu)}^p\,\dif x\bigg)^{\frac 1p} &\leq C\,
                                         \diameter(\Omega)^{k-\ell}
                                         \bigg(\dashint_{\Omega}\abs{\A
                                        u}^p\dif x \bigg)^{\frac 1p}
    \\
    \intertext{and}
    \label{eq:poincare-star2}
    \diameter(\Omega)^{\ell} \bigg( \dashint_\Omega
    \abs{D^\ell u}^p\,\dif x\bigg)^{\frac 1p} &\leq C\,
                                                     \bigg(\dashint_{\Omega}\abs{u}^p\dif x \bigg)^{\frac 1p}
    \\
    \notag
                                      &\quad+ C\,
                                        \diameter(\Omega)^{k} \bigg(\dashint_{\Omega}\abs{\A u}^p\dif x
                                        \bigg)^{\frac 1p},
  \end{align}
  where $C=C(p,\bbA,\sigma_{1},\sigma_{2})$. For $p=\infty$, one has to exchange
  $(\dashint_\Omega \abs{\cdot}^p\dif x)^{\frac 1p}$ by $\esssup_\Omega$.
  For $p=1$ and
  $u \in \setBV^{\opA}(\Omega)$, one replaces
  $\abs{\bbA u(y)}\dif y$ by $\dif \abs{\bbA u}(y)$.
\end{theorem}

In view of Corollary ~\ref{cor:repr-estimate}, Theorem~\ref{thm:poincare-john} is already available for domains which are star-shaped with respect to a ball; in particular, if $\Omega$ is a ball, we may choose $B=\Omega$, giving us back the Poincar\'{e} inequalities from \cite{BreDieGme20,GR19}. Moreover, the quantity on the left-hand side of \eqref{eq:poincare-star} enjoys a best approximation property, a fact that we shall return to in slightly higher generality in Section~\ref{sec:bestapprox}. 

Theorem~\ref{thm:poincare-john} is a direct consequence of the following Riesz-type estimates, improving Corollary~\ref{cor:repr-estimate} to John domains. The proof is slightly inspired by~\cite[Section~8.2]{DieHarHasRuz11}.
\begin{proposition}
  \label{prop:john-riesz-est}
In the situation of Theorem~\ref{thm:poincare-john},  for 
  all $\ell\in\{0,...,k-1\}$ and all $u \in
  \sobo^{\bbA,1}(\Omega)$ there holds with a constant $C=C(\A,\sigma_{1},\sigma_{2})>0$
  \begin{align}
    \label{eq:repr-estimate-2}
    \abs{D^\ell (u - \Pi_\A^Bu)} &\leq
                                                                C\, 
    \int_{\Omega} \abs{\cdot-y}^{-n+k-\ell}\abs{\A
    u(y)}\dif y\qquad \text{$\mathscr{L}^{n}$-almost everywhere}.
  \end{align}
  For $u \in \setBV^{\opA}(\Omega)$, \eqref{eq:repr-estimate-2} remains valid upon replacing $\abs{\bbA
  u(y)}\dif y$ by $\dif \abs{\bbA u}(y)$.
\end{proposition}
\begin{proof}
Let $x \in \Omega$. Slightly abusing notation, let $m\in\mathbb{N}$ and let $(W_j)_{j=0}^m$ be an emanating chain of balls connecting~$x$ with the central ball $B$ such that $W_0=B$ and $x\in W_m$ in the sense of Definition~\ref{def:boman_chain}, and we remark that each $\sigma_{1}W_{j}$ is still contained in $\Omega$. For the following argument, we may assume that $x\notin \sigma_{1}W_{j}$ for all $j\in\{0,...,m-1\}$ as otherwise we may take the minimal $j_{0}$ with $x\in \sigma_{1}W_{j_{0}}$, $x\notin \sigma_{1}W_{i}$ for $0\leq i<j_{0}$ and redefine $m$ to be $j_{0}$; we will only use that $x\in\sigma_{1}W_{m}$ (and not that $x\in W_{m}$) in the sequel. 

Now consider the chain $\mathbb{W}_{0}:=W_{0},...,\mathbb{W}_{m-1}:=W_{m-1}$ and $\mathbb{W}_{m}:=\sigma_{1}W_{m}$. It is then easy to see that this chain satisfies the following modification of~\ref{itm:C2}: For any $0\leq j \leq j' \leq m$ we have $\mathbb{W}_{j'}\subset\sigma_{1}\sigma_{2}\mathbb{W}_{j}$ and each $W_{j}\cap W_{j+1}$, $j\in\{0,...,m-1\}$ contains an open ball $\mathbb{B}_{j}$ with $\mathbb{W}_{j}\cup\mathbb{W}_{j+1}\subset\sigma_{1}\sigma_{2}\mathbb{B}_{j}$. As a main consequence of this slightly modified construction, every $y\in\mathbb{W}_{j}$, $j\in\{0,...,m-1\}$ is at least $(\sigma_{1}-1)r(\mathbb{W}_{j})$ distant apart from $x$, which follows from $x\in\sigma_{1}W_{m}$ but $x\notin\sigma_{1}W_{j}$ for any $0\leq j\leq m-1$ by the above minimality of $m$. On the other hand, since $x\in \mathbb{W}_{m}\subset \sigma_{1}\sigma_{2}\mathbb{W}_{j}$ for any $j\in\{0,...,m\}$, we have $|x-y|\leq 2\sigma_{1}\sigma_{2}r(W_{j})$ for all $y\in W_{j}$. In conclusion, we have  
\begin{align}\label{eq:mutualradiibounds}
(\sigma_{1}-1)r(\mathbb{W}_{j})\leq |x-y|\leq 2\sigma_{1}\sigma_{2}r(\mathbb{W}_{j})\;\;\;\text{for all}\;1\leq j \leq m-1\;\text{and}\;y\in\mathbb{W}_{j}.
\end{align}
Define $\Pi_\A^{\mathbb{W}_j}$ and
  $\Pi_\A^{\mathbb{B}_j}$ as indicated after Theorem~\ref{thm:poincare-john}.  Then, since $\Pi_{\A}^{\mathbb{W}_{j+1}}\Pi_{\A}^{\mathbb{B}_{i}}=\Pi_{\A}^{\mathbb{B}_{i}}$, $i\in\{j+1,j\}$,
  \begin{align*}
    \abs{D^\ell (u - \Pi_\A^{\mathbb{W}_{0}} u)(x)}
    &\leq
      \abs{D^\ell (u - \Pi_\A^{\mathbb{B}_{m-1}}u)(x)} + \abs{D^\ell (\Pi_\A^{\mathbb{W}_{0}} - \Pi_\A^{\mathbb{B}_{0}}u)(x)}\\ &
      +\sum_{j=0}^{m-2} \big(
      \abs{D^{\ell}(\Pi_{\A}^{\mathbb{B}_{j+1}} u  - \Pi_\A^{\mathbb{W}_{j+1}}u)(x)} +
      \abs{D^{\ell}(\Pi_{\A}^{\mathbb{B}_j} u - \Pi_\A^{\W_{j+1}}u)(x)} \big)\\
      &\leq
      \abs{D^\ell (u - \Pi_\A^{\mathbb{B}_{m-1}}u)(x)} + \abs{D^\ell (\Pi_\A^{\mathbb{W}_{0}} - \Pi_\A^{\mathbb{B}_{0}}u)(x)}\\ &
      +\sum_{j=0}^{m-2} \big(
      \abs{D^{\ell}\Pi_{\A}^{\mathbb{W}_{j+1}}(\Pi_{\A}^{\mathbb{B}_{j+1}} u  - u)(x)} +
      \abs{D^{\ell}\Pi_{\A}^{\mathbb{W}_{j+1}}( u - \Pi_\A^{\mathbb{B}_j}u)(x)} \big)\\ 
      & =: \mathrm{I} + \mathrm{II} + \mathrm{III}. 
  \end{align*}
  Then, since $x\in \mathbb{W}_{m}\subset\sigma_{1}\sigma_{2}\mathbb{W}_{j+1}$ for all $0\leq j \leq m-2$, we obtain by inverse estimates
\begin{align*}
\mathrm{III} & \leq \sum_{j=0}^{m-2}\sum_{i\in\{j,j+1\}} \|D^{\ell}\Pi_{\A}^{\mathbb{W}_{j+1}}(\Pi_{\A}^{\mathbb{B}_{i}} u  - u)\|_{\lebe^{\infty}(\sigma_{1}\sigma_{2}\mathbb{W}_{j+1})}\\ 
      & \leq c\sum_{j=0}^{m-2}\sum_{i\in\{j,j+1\}} r(\sigma_{1}\sigma_{2}\mathbb{W}_{j+1})^{-\ell}\times \\ & \;\;\;\;\;\;\;\;\;\;\;\;\;\;\;\;\;\;\;\;\;\;\;\;\;\;\;\;\times\Big( r(\sigma_{1}\sigma_{2}\mathbb{W}_{j+1})^{\ell} 
      \dashint_{\sigma_{1}\sigma_{2}\mathbb{W}_{j+1}}|D^{\ell}\Pi_{\A}^{\mathbb{W}_{j+1}}(\Pi_{\A}^{\mathbb{B}_{i}} u  - u)|\dif y\Big) \\ 
     & \!\!\!\!\stackrel{\eqref{eq:inverseFull}_{2}}{\leq}c\sum_{j=0}^{m-2}\sum_{i\in\{j,j+1\}} r(\mathbb{W}_{j+1})^{-\ell}\big(
      \dashint_{\sigma_{1}\sigma_{2}\mathbb{W}_{j+1}}|\Pi_{\A}^{\mathbb{W}_{j+1}}(\Pi_{\A}^{\mathbb{B}_{i}} u  - u)|\dif y \Big) \\ 
    & \leq c\sum_{j=0}^{m-2}\sum_{i\in\{j,j+1\}} r(\mathbb{W}_{j+1})^{-\ell}\big(
      \dashint_{\mathbb{W}_{j+1}}|\Pi_{\A}^{\mathbb{W}_{j+1}}(\Pi_{\A}^{\mathbb{B}_{i}} u  - u)|\dif y \Big) \\ 
       & \leq c\sum_{j=0}^{m-2}\sum_{i\in\{j,j+1\}} r(\mathbb{W}_{j+1})^{-\ell}\big(
      \dashint_{\mathbb{W}_{j+1}}|u-\Pi_{\A}^{\mathbb{B}_{i}} u|\dif y \Big)  =: c\sum_{j=0}^{m-2}\mathrm{III}_{j}, 
\end{align*}
where $c=c(\sigma_{1},\sigma_{2},\A)>0$ is a constant. Since each $\mathbb{W}_{j+1}$, $j\in\{0,...,m-2\}$, is starshaped with respect to both $\mathbb{B}_{j+1}$ and $\mathbb{B}_{j}$, we obtain by use of Proposition~\ref{prop:decomposition} with a constant $c=c(\sigma_{1},\sigma_{2},\A)>0$ and Fubini's theorem for $i\in\{j,j+1\}$ 
\begin{align*}
\mathrm{III}_{j} & \leq c\sum_{i\in\{j,j+1\}}r(\mathbb{W}_{j+1})^{k-\ell-n}\int_{\mathbb{W}_{j+1}}|\A u(z)|\dif z\\ 
& \leq cr(\mathbb{W}_{j+1})^{k-n-\ell}\int_{\mathbb{W}_{j+1}}|\A u(z)|\dif z \stackrel{\eqref{eq:mutualradiibounds}}{\leq} c\int_{\mathbb{W}_{j+1}}|x-z|^{k-n-\ell}|\A u(z)|\dif z =: \mathrm{IV}_{j}.
\end{align*}
On the other hand, the term $\mathrm{I}$ can clearly be bounded against $\mathrm{IV}_{j}$ with $j=m-1$, whereas the term $\mathrm{II}$ is estimated along the same lines as $\mathrm{III}$, now using that $\mathbb{W}_{0}$ is starshaped with respect to $\mathbb{B}_{0}$. In consequence, by the uniformly finite overlap of the balls $W_{0},...,W_{m}$ implied by \ref{itm:C2}, we conclude \eqref{eq:repr-estimate-2}. The proof is hereby complete. 
\end{proof}

\section{Proof of Theorem~\ref{thm:main}: The trace estimate via Riesz potentials}\label{sec:traces-estimates-via} 
In this section we establish Theorem~\ref{thm:main}. To this end, we recall in Section~\ref{sec:Besov} the scale of Besov spaces and prove Theorem~\ref{thm:main} in Sections~\ref{sec:geometricsetup}--~\ref{sec:BVAsection}. 
\subsection{Besov spaces}\label{sec:Besov}
For Theorem~\ref{thm:main}, we briefly recall the scale of Besov spaces and provide an instrumental characterisation by means of oscillations. Besov spaces can be introduced by different means, and we refer to \textsc{Triebel} \cite{Trie1,Trie2} for comprehensive overviews. In what follows, let $s>0$, $1\leq p < \infty$, $1\leq q < \infty$. Let $\varphi\in\hold_{c}^{\infty}(\ball(0,2))$ satisfy $\mathbbm{1}_{\ball(0,1)}\leq\varphi\leq\mathbbm{1}_{\ball(0,2)}$ and put, for $j\in\mathbb{Z}$, $\varphi_{j}(\xi):=\varphi(2^{-j}\xi)-\varphi(2^{-j+1}\xi)$. We recall that the \emph{Besov space} $\besov_{p,q}^{s}(\R^{n};V)$ consists of all $f\in\mathscr{S}'(\R^{n};V)$ such that 
\begin{align*}
\|f\|_{\besov_{p,q}^{s}(\R^{n})}:=\Big(\|(\varphi \widehat{f})^{\vee}\|_{\lebe^{p}(\R^{n})}^{q}+\sum_{j=1}^{\infty}2^{jsq}\|(\varphi_{j}\widehat{f})^{\vee}\|_{\lebe^{p}(\R^{n})}^{q}\Big)^{\frac{1}{q}}<\infty.
\end{align*}
The \emph{homogeneous Besov space} ${\dot\besov}{_{p,q}^{s}}(\R^{n};V)$ consists of all $f\in(\mathscr{S}'/\mathscr{P})(\R^{n};V)$ with
\begin{align}\label{eq:homogBesovnorm}
\|f\|_{{\dot\besov}{_{p,q}^{s}}(\R^{n})}:=\Big(\sum_{j=-\infty}^{\infty}2^{jsq}\|(\varphi_{j}\widehat{f})^{\vee}\|_{\lebe^{p}(\R^{n})}^{q}\Big)^{\frac{1}{q}}<\infty.
\end{align} 
Note that, whereas \eqref{eq:homogBesovnorm} does not give rise to a norm on $\mathscr{S}'(\R^{n};V)$ but only on the quotient $(\mathscr{S}'/\mathscr{P})(\R^{n};V)$, it does define a norm on $\mathscr{S}(\R^{n};V)$. By \textsc{Bergh \& L\"{o}fstr\"{o}m} \cite[p.~148]{MR58:2349} we have for all $s>0$ and $1\leq p,q<\infty$
\begin{align}\label{eq:BergLoef}
{\besov}{_{p,q}^{s}}(\R^{n};V)\simeq (\lebe^{p}\cap{\dot\besov}{_{p,q}^{s}})(\R^{n};V).
\end{align}
Next, for $1\leq p < \infty$ and $M\in\N_{0}$, we define the $(p,M)$-\emph{oscillation} of a measurable map $f\colon\R^{n}\to V$ at $x_{0}\in\R^{n}$ by 
\begin{align*}
\osc_{p}^{M}f(x_{0},r):=\inf\Big\{\Big(\dashint_{\ball(x_{0},r)}|f(x)-\pi(x)|^{p}\dif x\Big)^{\frac{1}{p}}\colon\;\pi\in\mathscr{P}_{M}(\R^{n};V) \Big\}. 
\end{align*}
The next lemma is certainly clear to the experts, but we have been unable to find a precise reference; see, however, \textsc{Dorronsoro} \cite{Dor85} for a related result. In our arguments below, we will merely require inequality '$\lesssim$' in \eqref{eq:Besovoscillation}, and a quick argument for this is sketched in the Appendix, Section~\ref{sec:oscchar}.
\begin{lemma}\label{lem:Besovosc}
  Let $1\leq p,q< \infty$ and $s>0$. If
  $M\in\mathbb{N}$ satisfies $M > \lfloor s\rfloor$, then for all $1\leq u \leq p$ we have the equivalence of norms 
  \begin{align}\label{eq:Besovoscillation}
    \begin{split}
\|f\|_{{\dot\besov}{_{p,q}^{s}}(\R^{n})}\simeq \Big(\int_{0}^{\infty}\|\osc_{u}^{M}f(\cdot,t)\|_{\lebe^{p}(\R^{n})}^{q}\frac{\dif t}{t^{1+sq}}\Big)^{\frac{1}{q}},\qquad f\in\mathscr{S}(\R^{n};V), 
    \end{split}
  \end{align}
  where the constants implicit in '$\simeq$' only depend on $n,M,s,p,q,u$ and $V$. 
\end{lemma}
Let $\Omega\subset\R^{n}$ be open and bounded with $\hold^{\infty}$-boundary $\partial\Omega$ and outer unit normal $\nu=(\nu_{1},...,\nu_{n})$. By definition, there exists $N\in\mathbb{N}$ and open balls $W_{1},...,W_{N}$ with $\partial\Omega\subset\bigcup_{j=1}^{N}W_{j}$ such that the following holds for all $j=1,...,N$: We have $\partial\Omega\cap W_{j}\neq\emptyset$, and there exists $F^{(j)}\in\hold^{\infty}(\overline{W}_{j};\R^{n})$ with $F^{(j)}\colon W_{j}\to\R^{n}$ injective and having open, bounded image, $F^{(j)}(\partial\Omega\cap W_{j})\subset\{x\in\R^{n}\colon\;x_{n}=0\}$ and $F^{(j)}(\Omega\cap W_{j})\subset\R^{n-1}\times (0,\infty)$ is open, simply connected and satisfies $\det(\nabla F^{(j)})\neq 0$ on $\overline{W}_{j}$. Let  $W_{0}\Subset\Omega$ be open and bounded with $\hold^{\infty}$-boundary $\partial W_{0}$ such that $\overline{\Omega}\subset \bigcup_{j=0}^{N}W_{j}$ and $(\psi_{j})_{j=0}^{N}$ be a partition of unity subject to $(W_{j})_{j=0}^{N}$. Following \textsc{Triebel} \cite[Chpt.~3.2.2, Def.~2]{Trie1}, for $s>0$, $1\leq p,q<\infty$, we define the Besov norm 
\begin{align}\label{eq:BesovBoundary}
\|f\|_{{\besov}{_{p,q}^{s}}(\partial\Omega)}:=\sum_{j=1}^{N}\|(\psi_{j}f)({F^{(j)}}^{-1}(\cdot,0))\|_{{\besov}{_{p,q}^{s}}(\R^{n-1})},\qquad f\in\mathscr{D}'(\partial\Omega;V).
\end{align}
In view of a convenient characterisation of trace spaces for smooth domains, we follow the slightly more general approach of \textsc{Maz'ya, Mitrea \& Shaposhnikova} \cite{MazMitSha05} and define for $m\in\mathbb{N}$, $1<p<\infty$ and $0<s<1$
\begin{align}\label{eq:BesovHomBounded}
{\dot\besov}{_{p}^{m-1+s}}(\partial\Omega;V):=\overline{\{(\partial^{\alpha}v|_{\partial\Omega})_{|\alpha|\leq m-1}\colon\;v\in\hold_{c}^{\infty}(\R^{n};V)\}}^{({\besov}{_{p,p}^{s}}(\partial\Omega;V))^{\mathtt{M}}}, 
\end{align}
where $\mathtt{M}:=\#\{\alpha\in\mathbb{N}_{0}^{n}\colon\;|\alpha|\leq m-1\}$. By \cite[Prop.~6.9]{MazMitSha05}, $(f^{\alpha})_{|\alpha|\leq m-1}$ belongs to ${\dot\besov}{_{p}^{m-1+s}}(\partial\Omega;V)$ if and only if 
\begin{align}\label{eq:BWarray}
\begin{split}
& f^{\alpha}\in{\besov}{_{p,p}^{s}}(\partial\Omega;V)\qquad\text{for all}\;|\alpha|\leq m-1,\;\;\;\text{and}\\
&(\nu_{j}\partial_{k}-\nu_{k}\partial_{j})f^{\alpha}=\nu_{j}f^{\alpha+e_{k}}-\nu_{k}f^{\alpha+e_{j}}\;\;\text{for all $|\alpha|\leq m-2$, $j,k\in\{1,...,n\}$}.
\end{split}
\end{align}
The natural compatibility condition~$\eqref{eq:BWarray}_{2}$ on the tangential derivatives is a higher order variant of the Whitney array compatibility conditions considered by \textsc{Adolfsson \& Pipher} \cite{AdoPip98}. As in \cite{AdoPip98}, the equations of $\eqref{eq:BWarray}_{2}$ are understood in the sense of being fulfilled by an arbitrary extension $\widetilde{f}^{\alpha}$ of $f^{\alpha}$ to a neighbourhood of $\partial\Omega$; since the operator $\nu_{j}\partial_{k}-\nu_{k}\partial_{j}$ is a tangential derivative on $\partial\Omega$, $(\nu_{j}\partial_{k}-\nu_{k}\partial_{j})\widetilde{f}^{\alpha}$ will be independent of the particular extension and hence is well-defined.

\begin{lemma}[{\cite[Prop.~6.5]{MazMitSha05}}]\label{lem:MMS}
Let $\Omega\subset\R^{n}$ be an open and bounded domain with $\hold^{\infty}$-boundary $\partial\Omega$. Let $1<p<\infty$, $m\in\mathbb{N}$ and define $\mathrm{tr}_{m-1,\partial\Omega}$ on $\sobo^{m,p}(\Omega;V)$ by
\begin{align}
\mathrm{tr}_{m-1,\partial\Omega}f:= \big(\trace_{\partial\Omega}(\partial^{\alpha}f)\big)_{|\alpha|\leq m-1},\qquad f\in\sobo^{m,p}(\Omega;V),
\end{align}
where $\trace_{\partial\Omega}\colon\sobo^{1,p}(\Omega;V)\to{\besov}{_{p,p}^{1-1/p}}(\partial\Omega;V)$ is the trace operator on $\sobo^{1,p}(\Omega;V)$ as usual. Then 
\begin{align}
\mathrm{tr}_{m-1,\partial\Omega}\colon\sobo^{m,p}(\Omega;V)\to{\dot\besov}{_{p}^{m-1/p}}(\partial\Omega;V)
\end{align}
is a well-defined, linear, bounded operator \emph{which is onto}. Moreover, there exists a bounded, linear right-inverse $\mathcal{L}\colon {\dot\besov}{_{p}^{m-1/p}}(\partial\Omega;V) \to \sobo^{m,p}(\Omega;V)$ of $\mathrm{tr}_{m-1,\partial\Omega}$ in the sense that, whenever $\dot{f}=(f^{\alpha})_{|\alpha|\leq m-1}\in{\dot\besov}{_{p}^{m-1/p}}(\partial\Omega;V)$, then
\begin{align}\label{eq:rightinverse}
\trace_{\partial\Omega}(\partial^{\alpha}(\mathcal{L}\dot{f}))=f^{\alpha}\qquad\text{for all}\;|\alpha|\leq m-1.
\end{align}
\end{lemma}
\subsection{Geometric setup and potential estimates}\label{sec:geometricsetup}
In this and the following subsection, we suppose that $\A$ is a $\mathbb{C}$-elliptic differential operator of the form \eqref{eq:form}. Moreover, we shall tacitly work with the particular halfspace $\bbH:=\R^{n-1}\times (0,\infty)$, but our arguments are easily seen to generalise to any halfspace of $\R^{n}$.

We start by fixing notation: For some given $z'\in\R^{n-1}$ and $t>0$ we denote 
\begin{align}
K_{z',t}:=\{(y',y_{n})\in\R^{n}\colon 0\leq y_{n}\leq t,\;|y'-z'|^{2}<9y_{n}^{2} \}
\end{align}
the single-sided cone emanating from $(z',0)$ in direction $e_{n}$, of height $t$ and base radius $3t$. To distinguish between $n$- and $(n-1)$-dimensional balls, we put for $x'\in\R^{n-1}$, $x\in\R^{n}$ and $r>0$ 
\begin{align*}
\ball_{n}(x,r):=\ball(x,r)\;\;\;\text{and}\;\;\;\ball_{n-1}(x',r):=\partial\bbH\cap\ball((x',0),r). 
\end{align*} 
Adopting this notation, let us note that in this situation there holds 
\begin{align}\label{eq:coneinclusion}
\ball_{n}((x',t),\tfrac{t}{2})\subset K_{y',2t}\qquad\text{for all}\; y'\in\R^{n-1}\;\text{with}\;|x'-y'|<t. 
\end{align}
In fact, let $(\xi',\xi_{n})\in \ball_{n}((x',t),\frac{t}{2})$ so that, firstly, $\frac{1}{2}t\leq \xi_{n}\leq \frac{3}{2}t$ and hence $\frac{3}{2}t\leq 3\xi_{n}$. As a consequence, 
\begin{align*}
|\xi'-y'| \leq |\xi'-x'| + |x'-y'| < \frac{3}{2}t \leq 3\xi_{n}
\end{align*}
and so \eqref{eq:coneinclusion} follows. This situation is depicted in Figure~\ref{fig:umbra}. Combining this setup with Proposition~\ref{prop:decomposition}, we arrive at the following potential estimates on cones:
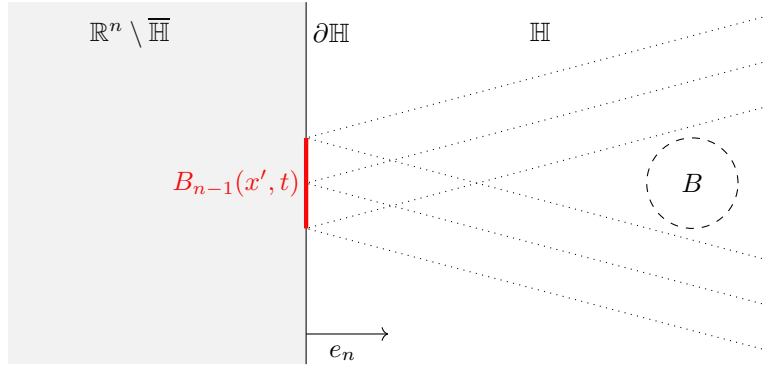
\begin{figure}   
\begin{tikzpicture}[scale=2]
  \begin{scope}
    \draw[-] (-0.04,-1.2) -- (-0.04,1.2);
  \end{scope} 
  \node at (0.125,1.0) {$\partial\bbH$};
  \node at (1.5,1.0) {$\bbH$};
  \node at (-1.2,1.0) {$\R^{n}\setminus\overline{\bbH}$};
  \begin{scope}[shift={(-0.04,0)}]
    \draw[dotted] (3,.8) -- (0,0) -- (3,-.8);
  \end{scope}
  \begin{scope}[shift={(-0.04,.3)}]
    \draw[dotted] (3,.8) -- (0,0) -- (3,-.8);
  \end{scope}
  \begin{scope}[shift={(-0.04,-.3)}]
    \draw[dotted] (3,.8) -- (0,0) -- (3,-.8);
  \end{scope}
  \draw[dashed] (2.5,0) circle (0.3cm);
  \node at (2.5,0)  {$B$};
  \draw[red,ultra thick]  (-0.04,-0.3) -- (-0.04,0.3);
  \node[red] at (-0.5,0) {$\ball_{n-1}(x',t)$};
  \path[fill=gray,opacity=0.1](-2,-1.2) -- (-0.04,-1.2) -- (-0.04,1.2) -- (-2,1.2) -- cycle;
  \draw[->] (-0.04,-1)--(0.5,-1);
  \node at (0.2,-1.125) {$e_{n}$};
\end{tikzpicture}
\caption{Not-to-scale construction of the the umbra shadow estimate underlying the proof of Theorem~\ref{thm:main}. To control the oscillation of $u(\cdot,0)$, the Riesz potential estimate from Lemma~\ref{lem:conical} is applied to shifted cones emanating from points in the red area, where all cones are star-shaped with respect to a joint ball $\ball$.
} \label{fig:umbra}
\end{figure}
\begin{lemma}[Conical Riesz potential estimate]\label{lem:conical}
There exists $c=c(\A)>0$ with the following property: If $x'\in\R^{n-1}$ and $t>0$, then for every $y'\in\R^{n-1}$ with $|x'-y'|<t$ there holds for all $u\in\hold^{\infty}(\overline{\bbH};V)$ and all $\alpha\in\mathbb{N}_{0}^{n}$ with $|\alpha|\leq k-1$
\begin{align}\label{eq:conical}
\begin{split}
|(\partial^{\alpha}u)(y',0)-(\partial^{\alpha}&\Pi_\A^{\ball_{n}((x',t),\frac{t}{2})}u)(y',0)|\\ &  \leq c \int_{K_{y',2t}}\frac{|\A u(z',z_{n})|}{(|z'-y'|+|z_{n}|)^{n+|\alpha|-k}}\dif\,(z',z_{n}).
\end{split}
\end{align}
\end{lemma}
\begin{proof}
Let $x',y'$ and $t$ be as in the lemma. By \eqref{eq:coneinclusion}, $\ball_{n}((x',t),\tfrac{t}{2})\subset K_{y',2t}$ and clearly $K_{y',2t}$ is star-shaped with respect to $\ball_{n}((x',t),\tfrac{t}{2})$. Thus we may apply Proposition~\ref{prop:decomposition} to $\Omega=K_{y',2t}$ and $\ball=\ball_{n}((x',t),\tfrac{t}{2})$. Note that $\Pi_\A^Bu:=\widetilde{\Pi}_{\A}\mathbb{T}_{m-1}^{\ball}$ is independent of $y'$ and the diameters of $K_{y',2t}$ and $\ball_{n}((x',t),\tfrac{t}{2})$ are uniformly comparable. Hence \eqref{eq:conical} follows from \ref{item:decomp-kernel} via
\begin{align*}
\left\vert \int_{K_{y',2t}}(\partial_{y}^{\alpha}\frK_\A^\ball)((y',0),z)\A u(z)\dif z \right\vert \leq C\int_{K_{y',2t}}\frac{|\A u(z',z_{n})|}{(|y'-z'|+|z_{n}|)^{n+|\alpha|-k}}\dif\,(z',z_{n}). 
\end{align*}
The proof is complete.  
\end{proof}

\subsection{The trace estimate for smooth maps}\label{sec:proofhalfspace}
After the preparations of the preceding subsection, we can now proceed to showing that $\mathbb{C}$-ellipticity suffices for the trace estimate of Theorem~\ref{thm:main}~\ref{item:mainB}. We provide a slightly stronger variant that is also applicable to the $1<p<\infty$-scenario of Theorem~\ref{thm:mainLp}   below.
\begin{proof}[Proof of sufficiency of $\mathbb{C}$-ellipticity for \eqref{eq:tracemaininequality0}]
Let $1\leq p < \infty$ and let $\A$ be a $\mathbb{C}$-elliptic differential operator of the form \eqref{eq:form} with $k\geq 2$. Moreover, let $u\in\hold_{c}^{\infty}(\overline{\bbH};V)$, $x'\in\partial\bbH$ and $t>0$. We distinguish the cases $|\alpha|\leq k-2$ and $|\alpha|=k-1$. In the first case, we will establish a ${\dot\besov}{_{p,p}^{k-|\alpha|-1/p}}(\partial\mathbb{H})$-bound on $(\partial^{\alpha}u)(\cdot,0)$ for all $1\leq p < \infty$, whereas in the second case we will only strive for such a bound provided $1<p<\infty$.

Let $|\alpha|\leq k-2$ and $1\leq p <\infty$. Since $\A$ is $\mathbb{C}$-elliptic, we may choose $m:=\mathrm{deg}_{\mathscr{P}}(\A)\in\mathbb{N}$ (cf.~\eqref{eq:deg-pol}) so that $\partial^{\alpha}\ker(\A)\subset\mathscr{P}_{m}(\R^{n};V)$ and, with $s:=k-|\alpha|-1/p>0$, 
\begin{align*}
m > \mathrm{deg}_{\mathscr{P}}(\A)-1\geq k-|\alpha|-1=\lfloor k-|\alpha|-1/p\rfloor = \lfloor s\rfloor.
\end{align*}
Since $p$ and $m$ are fixed, we put for a $\mathscr{H}^{n-1}$-locally integrable map $v\colon\partial\bbH\to V$
\begin{align}\label{eq:oschalfspacedef}
\osc_{p,\partial\bbH}^{m}v(x',t):=\inf \Big\{\Big(\dashint_{\ball_{n-1}(x',t)}|v-\pi|^{p}\dif^{n-1}y\Big)^{\frac{1}{p}}\colon\;\pi\in\mathscr{P}_{m}(\R^{n};V)\Big\}
\end{align}
In this situation, Lemma~\ref{lem:Besovosc} implies with $U^{\alpha}(x'):=(\partial^{\alpha}u)(x',0)$ 
\begin{align}\label{eq:osc1}
\|U^{\alpha}\|_{{\dot\besov}{_{p,p}^{k-|\alpha|-1/p}(\partial\bbH)}}^{p} \leq c \int_{0}^{\infty}\frac{\|\osc_{p,\partial\bbH}^{m}U^{\alpha}(\cdot,t)\|_{\lebe^{p}(\partial\bbH)}^{p}}{t^{(k-|\alpha|)p}}\dif t.
\end{align}
To bound the right-hand side, we use Lemma~\ref{lem:conical} and $\partial^{\alpha}\ker(\A)\subset\mathscr{P}_{m}(\R^{n};V)$ to obtain 
\begin{align}\label{eq:osc1A}
\begin{split}
&|\osc_{p,\partial\bbH}^{m}U^{\alpha}(x',t)|^{p}  \leq \dashint_{\ball_{n-1}(x',t)}|U^{\alpha}(y')-\partial^{\alpha}\Pi_\A^{\ball_{n}((x',t),\frac{t}{2})}u(y',0)|^{p}\dif^{n-1}y'\\
& \;\;\;\;\;\;\leq c\dashint_{\ball_{n-1}(x',t)}\left\vert\int_{0}^{2t}\int_{\partial\bbH}\mathbbm{1}_{\{|z'-y'|<3\tau\}}\frac{|\A u(z',\tau)|}{(|z'-y'|+\tau)^{n-k+|\alpha|}}\dif^{n-1} z'\dif\tau\right\vert^{p}\dif^{n-1} y'\\
& \;\;\;\;\;\;\leq c\dashint_{\ball_{n-1}(x',t)}t^{p-1}\int_{0}^{2t}\tau^{(n-1)(p-1)}\times \\ & \;\;\;\;\;\;\times \Big(\int_{\partial\bbH}\mathbbm{1}_{\{|z'-y'|<3\tau\}}\frac{|\A u(z',\tau)|^{p}}{(|z'-y'|+\tau)^{p(n-k+|\alpha|)}}\dif^{n-1} z'\Big)\dif\tau \dif^{n-1}y'. 
\end{split}
\end{align}
Let us note that the very first of the previous chain of inequalities is where the $\mathbb{C}$-ellipticity enters as otherwise the nullspace does not consist of polynomials of a fixed maximal degree. In combination with \eqref{eq:osc1}, we thus obtain 
\begin{align*}
& \|U^{\alpha}\|_{{\dot\besov}{_{p,p}^{k-|\alpha|-1/p}}(\partial\bbH)}^{p} \leq c\int_{0}^{\infty}\int_{\partial\bbH}\dashint_{\ball_{n-1}(x',t)}t^{p-1}\int_{0}^{2t}\tau^{(n-1)(p-1)}\times \\ & \times \Big(\int_{\partial\bbH}\mathbbm{1}_{\{|z'-y'|<3\tau\}}\frac{|\A u(z',\tau)|^{p}}{(|z'-y'|+\tau)^{p(n-k+|\alpha|)}}\dif^{n-1}z'\Big)\dif\tau \dif^{n-1} y'\dif^{n-1} x'\frac{\dif t}{t^{(k-|\alpha|)p}}\\
&\leq c\int_{0}^{\infty}\int_{\partial\bbH}\int_{\partial\bbH}\int_{0}^{2t}\tau^{(n-1)(p-1)}\mathbbm{1}_{\{|x'-y'|<t\}}\times \\ & \;\;\;\times \Big(\int_{\partial\bbH}\frac{\mathbbm{1}_{\{|z'-y'|<3\tau\}}|\A u(z',\tau)|^{p}}{(|z'-y'|+\tau)^{p(n-k+|\alpha|)}}\dif^{n-1} z'\Big)\dif\tau \dif^{n-1} y'\dif^{n-1} x'\frac{\dif t}{t^{(k-|\alpha|-1)p+n}} =(*)
\end{align*}
Now, if $\mathbbm{1}_{\{|z'-y'|<3\tau\}}=1$, then $\tau\leq |z'-y'|+\tau \leq 4\tau$. Therefore, 
\begin{align*}
(*)&\leq c\int_{0}^{\infty}\int_{\partial\bbH}\int_{\partial\bbH}\int_{\partial\bbH}\int_{0}^{2t}\tau^{(n-1)(p-1)}\mathbbm{1}_{\{|z'-y'|<3\tau\}}\mathbbm{1}_{\{|x'-y'|<t\}}\times \\ & \;\;\;\;\;\;\;\;\;\;\;\;\;\;\;\;\;\;\;\;\;\;\;\;\;\;\;\;\times\frac{|\A u(z',\tau)|^{p}}{\tau^{p(n-k+|\alpha|)}}\dif\tau \dif^{n-1} z'\dif^{n-1} y'\dif^{n-1} x'\frac{\dif t}{t^{(k-|\alpha|-1)p+n}}\\
&\leq c\int_{0}^{\infty}\int_{\partial\bbH}\int_{\partial\bbH}\int_{0}^{2t}\mathbbm{1}_{\{|z'-y'|<3\tau\}}\times \\ & \;\;\;\;\;\;\;\;\;\;\;\;\;\;\;\;\;\;\;\;\;\;\;\;\;\;\;\;\;\;\;\;\;\;\;\;\;\;\;\;\;\times\frac{|\A u(z',\tau)|^{p}}{\tau^{n-1+p(1+|\alpha|-k)}}\dif\tau \dif^{n-1} z'\dif^{n-1} y'\frac{\dif t}{t^{(k-|\alpha|-1)p+1}}\\
&\leq c\int_{0}^{\infty}\int_{\partial\bbH}\int_{0}^{2t}\frac{|\A u(z',\tau)|^{p}}{\tau^{p(1+|\alpha|-k)}}\dif\tau \dif^{n-1} z'\frac{\dif t}{t^{(k-|\alpha|-1)p+1}}\\
&= c\int_{0}^{\infty}\int_{\partial\bbH}\int_{0}^{\infty}\mathbbm{1}_{(0,2t)}(\tau)|\A u(z',\tau)|^{p}\frac{\dif \tau}{\tau^{p(1+|\alpha|-k)}}\dif^{n-1} z'\frac{\dif t}{t^{(k-|\alpha|-1)p+1}}\\
&= c\int_{0}^{\infty}\int_{\partial\bbH}\int_{0}^{\infty}\mathbbm{1}_{(0,2t)}(\tau)|\A u(z',\tau)|^{p}\frac{\dif t}{t^{(k-1-|\alpha|)p+1}}\dif^{n-1} z'\frac{\dif \tau}{\tau^{p(1+|\alpha|-k)}}\\
&\leq c\int_{0}^{\infty}\int_{\partial\bbH}\Big(\int_{\frac{\tau}{2}}^{\infty}\frac{\dif t}{t^{(k-1-|\alpha|)p+1}}\Big)|\A u(z',\tau)|^{p}\dif^{n-1}z'\frac{\dif \tau}{\tau^{p(1+|\alpha|-k)}}\\
&\!\!\!\!\!\!\!\!\!\!\stackrel{k-|\alpha|-1>0}{\leq} c\int_{0}^{\infty}\int_{\partial\bbH}|\A u(z',\tau)|^{p}\dif^{n-1}z'\dif \tau\\
& = c\|\A u\|_{\lebe^{p}(\bbH)}^{p}.
\end{align*}
The reader will notice that this argument requires refinement for $|\alpha|=k-1$ since in this case, the penultimate estimation cannot be accomplished. 

Let $1<p<\infty$ and $|\alpha|=k-1$; in this case, the reader will notice that we may even allow for $k=1$. Defining $\mathrm{osc}_{1,\partial\mathbb{H}}^{m}$ in the analogous manner as in \eqref{eq:oschalfspacedef}, we imitate \eqref{eq:osc1A} to find by changing variables $z'=x'-\xi'$ and $\eta'=x'-y'$ in the third step
\begin{align*}
|\osc_{1,\partial\mathbb{H}}^{m}U^{\alpha}&(x',t)|  \leq \frac{c}{t^{n-1}}\int_{0}^{2t}\int_{\partial\mathbb{H}}\int_{\partial\bbH}\mathbbm{1}_{\{|x'-y'|<t\}}\mathbbm{1}_{\{|z'-y'|<3\tau\}}\times \\ 
& \;\;\;\;\;\;\;\;\;\;\;\;\;\;\;\;\;\;\;\;\;\;\;\;\;\;\;\;\;\;\;\;\times\frac{|\A u(z',\tau)|}{(|z'-y'|+\tau)^{n-1}}\dif^{n-1} y'\dif^{n-1} z'\dif\tau \\ 
& = \frac{c}{t^{n-1}}\int_{0}^{2t}\int_{\partial\mathbb{H}}|\A u(z',\tau)|\Big(\int_{\partial\bbH}\mathbbm{1}_{\{|x'-y'|<t\}}\mathbbm{1}_{\{|z'-y'|<3\tau\}}\times \Big.\\ 
& \Big. \;\;\;\;\;\;\;\;\;\;\;\;\;\;\;\;\;\;\;\;\;\;\;\;\;\;\;\;\;\;\;\;\times\frac{1}{(|z'-y'|+\tau)^{n-1}}\dif^{n-1} y'\Big)\dif^{n-1} z'\dif\tau\\
& = \frac{c}{t^{n-1}}\int_{0}^{2t}\int_{\partial\mathbb{H}}|\A u(x'-\xi',\tau)|\mathbbm{1}_{\{|\xi'|\leq 7t\}}\Big(\int_{\partial\bbH}\mathbbm{1}_{\{|\eta'|<t\}}\mathbbm{1}_{\{|\eta'-\xi'|<3\tau\}}\times \Big.\\ 
& \Big. \;\;\;\;\;\;\;\;\;\;\;\;\;\;\;\;\;\;\;\;\;\;\;\;\;\;\;\;\;\;\;\;\times\frac{1}{(|\eta'-\xi'|+\tau)^{n-1}}\dif^{n-1}\eta'\Big)\dif^{n-1} \xi'\dif\tau = (**), 
\end{align*}
where we have used that if $|\xi'|>7t$, then $|\eta'|<t$ yields
\begin{align*}
3\tau\stackrel{\tau\leq 2t}{\leq} 6t = 7t - t < |\xi'| - |\eta'| \leq |\xi'-\eta'|
\end{align*}
and so $\mathbbm{1}_{\{|\eta'|<t\}}\mathbbm{1}_{\{|\eta'-\xi'|<3\tau\}}=\mathbbm{1}_{\{|\eta'|<t\}}\mathbbm{1}_{\{|\eta'-\xi'|<3\tau\}}\mathbbm{1}_{\{|\xi'|\leq 7t\}}$. Since we have 
\begin{align*}
c\int_{\ball_{n-1}(0,3\tau)}\frac{\dif^{n-1}\theta}{(|\theta|+\tau)^{n-1}} \leq c\frac{\mathscr{H}^{n-1}(\ball_{n-1}(0,3\tau))}{\tau^{n-1}}\leq c
\end{align*}
as a bound for the very inner integral in $(**)$, we thus conclude
\begin{align*}
|\osc_{1,\partial\mathbb{H}}^{m}U^{\alpha}(x',t)| \leq (**) \leq \frac{c}{t^{n-1}}\int_{0}^{2t}\int_{\partial\mathbb{H}}|\A u(x'-\xi',\tau)|\mathbbm{1}_{\{|\xi'|\leq 7t\}}\dif^{n-1}\xi'\dif\tau.
\end{align*}
Integrating the previous inequality with respect to $x'\in\partial\mathbb{H}$, we choose $0<\delta<p-1$ to find by H\"{o}lder's inequality and a change of variables
\begin{align}\label{eq:osc2}
\begin{split}
\|&\osc_{1,\partial\mathbb{H}}^{m}U^{\alpha}(\cdot,t)\|_{\lebe^{p}(\partial\mathbb{H})}^{p} \\ & \leq \frac{c}{t^{p(n-1)}}\int_{\partial\mathbb{H}}\left\vert\int_{0}^{2t}\frac{1}{\tau^{\frac{\delta}{p}}}\int_{\partial\mathbb{H}}\tau^{\frac{\delta}{p}}\mathbbm{1}_{\{|\xi'|\leq 7t\}}|\A u(x'-\xi',\tau)|\dif^{n-1}\xi'\dif\tau\right\vert^{p}\dif^{n-1}x'\\
& \leq \frac{ct^{(p-1)(n-1)}}{t^{p(n-1)}}\int_{\partial\mathbb{H}} \Big( \int_{0}^{2t}\frac{\dif\tau}{\tau^{\frac{\delta}{p-1}}}\Big)^{p-1}\times \\ & \;\;\;\;\;\;\;\times \Big(\int_{0}^{2t}\tau^{\delta}\int_{\partial\mathbb{H}}\mathbbm{1}_{\{|\xi'|\leq 7t\}}|\A u(x'-\xi',\tau)|^{p}\dif^{n-1}\xi'\dif\tau\Big) \dif^{n-1}x' \\ 
& \leq \frac{ct^{(p-1)(n-1)}}{t^{p(n-1)}}t^{p+n-2-\delta}\Big(\int_{0}^{2t}\tau^{\delta}\int_{\partial\mathbb{H}}|\A u(\xi',\tau)|^{p}\dif^{n-1}\xi'\dif\tau\Big)  \\
& \leq ct^{p-1-\delta}\int_{0}^{2t}\tau^{\delta}\|\A u(\cdot,\tau)\|_{\lebe^{p}(\partial\mathbb{H})}^{p}\dif\tau.
\end{split}
\end{align}
Working from \eqref{eq:osc2}, we then obtain  
\begin{align*}
\|U^{\alpha}\|_{{\dot\besov}{_{p,p}^{1-1/p}}(\partial\mathbb{H})}^{p} & \leq c\int_{0}^{\infty}t^{p-1-\delta}\int_{0}^{2t}\tau^{\delta}\|\A u(\cdot,\tau)\|_{\lebe^{p}(\partial\mathbb{H})}^{p}\dif\tau\frac{\dif t}{t^{p}} \\ 
& = c\int_{0}^{\infty}\Big(\int_{\frac{\tau}{2}}^{\infty}\frac{\dif t}{t^{1+\delta}}\Big)\tau^{\delta}\|\A u(\cdot,\tau)\|_{\lebe^{p}(\partial\mathbb{H})}^{p} \dif\tau \\ 
& \leq c\|\A u\|_{\lebe^{p}(\mathbb{H})}^{p}. 
\end{align*}
The proof is complete. 
\end{proof}
Let us note that for $p=1$ and $|\alpha|=0$, the above proof works \emph{and is bound to work} for $k\geq 2$ only. Indeed, for $k=1$, $\lebe^{1}(\partial\bbH;V)$-trace estimates are optimal \cite{Gal57,BreDieGme20}. However,
\begin{align}\label{eq:triebelmention}
{\dot\besov}{_{1,1}^{0}}(\partial\bbH)\hookrightarrow {\dot{\mathrm{F}}}{_{1,2}^{0}}(\partial\bbH)=\mathcal{H}^{1}(\partial\bbH)\subsetneq\lebe^{1}(\partial\bbH),
\end{align} 
with the homogeneous Hardy space $\mathcal{H}^{1}$, cf.~\textsc{Triebel} \cite[Chpt.~5.2.4]{Trie1}, so that maps with a finite ${\dot\besov}{_{1,1}^{0}}(\partial\bbH)$-norm need to have integral zero (cf.~\cite[Prop.~5.8]{MirRus15}). Also note that the oscillation characterisation in the zero order case does not hold, cf.~\cite[Chpt.~3.5.1]{Trie2}.
\subsection{Proof of Theorem~\ref{thm:main} and its counterpart for $1<p<\infty$}\label{sec:BVAsection}
In this section we prove Theorem~\ref{thm:main} and state and prove its counterpart for $1<p<\infty$, Theorem~\ref{thm:mainLp} below, where stronger statements are available. 
\begin{proof}[Proof of Theorem~\ref{thm:main}]
Implication \ref{item:mainA}$\Rightarrow$\ref{item:mainB} has been established in Section~\ref{sec:proofhalfspace}. On the other hand, since ${\dot\besov}{_{1,1}^{k-1}}(\partial\mathbb{H};V)\hookrightarrow {\dot\sobo}{^{k-1,1}}(\partial\mathbb{H};V)$ for any open halfspace $\mathbb{H}$, \ref{item:mainB} yields the inequality $\|D^{k-1}u\|_{\lebe^{1}(\partial\mathbb{H})}\leq c\|\A u\|_{\lebe^{1}(\mathbb{H})}$ for all $u\in\hold_{c}^{\infty}(\overline{\mathbb{H}};V)$. By \textsc{Raita, Van Schaftingen} and the second author \cite[Thm.~5.2]{GRVS}, this already implies that $\A$ is $\mathbb{C}$-elliptic, and the proof is complete. 
\end{proof}
The rest of the section is devoted to a self-contained proof of the following theorem in the spirit of \cite{Aro54,Smith70,Kal94}, where our main focus now is on the equivalence of sharp trace estimates and Korn-type inequalities for smooth domains. Here, sharp trace estimates give access to global Calder\'{o}n-Zygmund estimates, which is fundamentally different from the non-smooth context of John domains as discussed in Section~\ref{sec:Korn}. 
\begin{theorem}[Higher order traces, $1<p<\infty$] \label{thm:mainLp}
Let $k\in\mathbb{N}$, $k\geq 1$ and let $1<p<\infty$. Then the following are equivalent for a $k$-th order differential operator $\A$ of the form \eqref{eq:form}: 
\begin{enumerate}
\item\label{item:mainALp} $\A$ is $\mathbb{C}$-elliptic.
\item\label{item:mainBLp}
There exists a constant $c=c(p,\A)>0$ such that for any open halfspace $\mathbb{H}$ and all $\alpha\in\mathbb{N}_{0}^{n}$ with $|\alpha|\leq k-1$ there holds 
\begin{align}\label{eq:tracemaininequality0A}
\|\partial^{\alpha}u\|_{{\dot\besov}{_{p,p}^{k-|\alpha|-1/p}}(\partial\bbH)}\leq c\|\A u\|_{\lebe^{p}(\bbH)}\qquad\text{for all}\;u\in\hold_{c}^{\infty}(\overline{\bbH};V).
\end{align}
\item\label{item:mainCLp} For every open and bounded domain $\Omega$ with boundary of class $\hold^{\infty}$, there exists a constant $c=c(p,\A,\Omega)>0$ such that for all $\alpha\in\mathbb{N}_{0}^{n}$ with $|\alpha|\leq k-1$ there holds 
\begin{align}\label{eq:tracemaininequality0A}
\|\partial^{\alpha}u\|_{{\besov}{_{p,p}^{k-|\alpha|-1/p}}(\partial\Omega)}\leq c(\|u\|_{\sobo^{k-1,p}(\Omega)}+\|\A u\|_{\lebe^{p}(\Omega)})\;\;\text{for all}\;u\in\hold^{\infty}(\overline{\Omega};V).
\end{align} 
\item\label{item:mainDLp} For every open and bounded domain $\Omega$ with boundary of class $\hold^{\infty}$ and all $1< p < \infty$, there exists a constant $c=c(p,\A,\Omega)>0$ such that
\begin{align}\label{eq:KornLp}
\sum_{j=0}^{k}\|\nabla^{j}u\|_{\lebe^{p}(\Omega)}\leq c(\|u\|_{\sobo^{k-1,p}(\Omega)}+\|\A u\|_{\lebe^{p}(\Omega)})\qquad\text{for all}\;u\in\hold^{\infty}(\overline{\Omega};V), 
\end{align}
\end{enumerate}
In particular, if any of \ref{item:mainALp}--\ref{item:mainDLp} holds, then we have (with $\sobo^{\A,p}$ and $\widetilde{\sobo}{^{\A,p}}$ as in Section~\ref{sec:repr-form}) for any open and bounded domain with smooth boundary $\partial\Omega$
\begin{align}\label{eq:Kornvariant}
\sobo^{k,p}(\Omega;V)\simeq \sobo^{\A,p}(\Omega)\simeq  {\widetilde{\sobo}}{^{\A,p}}(\Omega).
\end{align} 
\end{theorem}
It will be clear from the proof that some implications also persist for less regular domains (so, e.g., \ref{item:mainALp}$\Rightarrow$\ref{item:mainCLp} would also follow for domains satisfying a cone condition by a slight variation of the arguments of Section~\ref{sec:proofhalfspace}), but we stick to smooth domains for ease of exposition.  For the proof of Theorem~\ref{thm:mainLp}, we require two preparatory lemmas:
\begin{lemma}[Peetre-Tartar lemma, {\cite[Lem.~11.1]{Tar07}}]\label{lem:PTequiv}
Let $(E_{i},\|\cdot\|_{E_{i}})$, $i\in\{1,2,3\}$, be three Banach spaces, $A\colon E_{1}\to E_{2}$ a linear and bounded operator and $B\colon E_{1}\to E_{3}$ be a linear and compact operator. If $x\mapsto \|Ax\|_{E_{2}}+\|Bx\|_{E_{3}}$ is a norm  on $E_{1}$ that is equivalent to
$\|\cdot\|_{E_{1}}$, then $\dim(\ker (A))<\infty$. 
\end{lemma}
\begin{lemma}[Traces and Gau\ss -Green]\label{lem:GaussGreen}
Suppose that assertion~\ref{item:mainCLp} of Theorem~\ref{thm:mainLp} holds. Then for every $1< p < \infty$ and every open and bounded domain $\Omega\subset\R^{n}$ with boundary of class $\hold^{\infty}$ there exists a linear and bounded \emph{boundary trace operator}
\begin{align}\label{eq:sharpBesovsmooth}
\mathrm{tr}_{k-1,\partial\Omega}^{\A}\colon{\widetilde{\sobo}}{^{\A,p}}(\Omega)\ni u \mapsto \big(\mathrm{tr}_{\alpha,\partial\Omega}^{\A}(u)\big)_{|\alpha|\leq k-1}\in{\dot\besov}{_{p}^{k-1/p}}(\partial\Omega;V)
\end{align}
such that, in particular, $\mathrm{tr}_{k-1,\partial\Omega}^{\A}(u)=((\partial^{\alpha}u)|_{\partial\Omega})_{|\alpha|\leq k-1}$ holds for all $u\in\hold^{\infty}(\overline{\Omega};V)$ and so $\mathrm{tr}_{k-1,\partial\Omega}^{\A}=\mathrm{tr}_{k-1,\partial\Omega}$ on $\sobo^{k,p}(\Omega;V)$ with $\mathrm{tr}_{k-1,\partial\Omega}$ as in Lemma~\ref{lem:MMS}.

For $\ell\in\{0,...,k-1\}$, put $\mathrm{tr}_{\partial\Omega}^{\A,\ell}(u):=(\mathrm{tr}_{\alpha,\partial\Omega}^{\A}(u))_{|\alpha|=\ell}$. For each $l\in\{1,...,n\}$ and each $s\in\{1,...,k\}$ there exist continuous bilinear forms $\mathscr{B}_{l,s}^{\A}\colon \odot^{k-s}(\R^{n};V)\times\odot^{s-1}(\R^{n};W)\to\R$ such that for all $u\in{\widetilde{\sobo}}{^{\A,p}}(\Omega)$ and all $\varphi\in\hold^{\infty}(\overline{\Omega};W)$ we have the \emph{Gau\ss -Green identity} (with $\nu=(\nu_{1},...,\nu_{n})$ denoting the outer unit normal to $\partial\Omega$)
\begin{align}\label{eq:GG}
\begin{split}
\int_{\Omega}\A u\cdot\varphi\dif x & = (-1)^{k}\int_{\Omega} u\cdot \A^{*}\varphi\dif x  \\ &  \!\!\!\!+ \sum_{l=1}^{n}\sum_{s=1}^{k}\int_{\partial\Omega}\mathscr{B}_{l,s}^{\A}\Big(\mathrm{tr}_{\partial\Omega}^{\A,k-s}(u),\nabla^{s-1}\varphi\Big)\nu_{l}\dif\mathscr{H}^{n-1},
\end{split}
\end{align}
where $\A^{*}=\sum_{|\alpha|=k}\A_{\alpha}^{\top}\partial^{\alpha}$ is the formal $\lebe^{2}$-adjoint of $\A$.
\end{lemma}
\begin{proof} 
Since $\Omega$ has $\hold^{\infty}$-boundary, an argument similar to that in, e.g., \cite[Chpt.~5.3.3, Thm.~5.3]{Eva98} yields that for every $u\in{\widetilde{\sobo}}{^{\A,p}}(\Omega)$ there exists $(u_{j})\subset\hold^{\infty}(\overline{\Omega};V)$ such that $\|u-u_{j}\|_{\widetilde{\sobo}^{\A,p}(\Omega)}\to 0$ as $j\to\infty$. Let $\alpha\in\mathbb{N}_{0}^{n}$ satisfy $|\alpha|\leq k-1$. By \eqref{eq:tracemaininequality0A}, $((\partial^{\alpha}u_{j})|_{\partial\Omega})$ is Cauchy in ${\besov}{_{p,p}^{k-|\alpha|-1/p}}(\partial\Omega;V)$ and therefore converges to some $u^{\alpha}=:\mathrm{tr}_{\alpha,\partial\Omega}^{\A}(u)\in {\besov}{_{p,p}^{k-|\alpha|-1/p}}(\partial\Omega;V)$, which is clearly independent of the approximating sequence and hereafter well-defined. To conclude \eqref{eq:sharpBesovsmooth}, we note that $\{u^{\alpha}\colon\,|\alpha|\leq k-1\}$ matches the Besov-Whitney array condition $\eqref{eq:BWarray}_{2}$ $\mathscr{H}^{n-1}$-a.e. on $\partial\Omega$; indeed, for any $|\alpha|\leq k-2$ and all $i,l\in\{1,...,n\}$ we have 
\begin{align*}
\|\nu_{i}\partial_{l}u^{\alpha} &- \nu_{l}\partial_{i}u^{\alpha} -(\nu_{i}u^{\alpha + e_{l}}-\nu_{l}u^{\alpha+e_{i}})\|_{\lebe^{p}(\partial\Omega)} \\ & \leq \|(\nu_{i}\partial_{l}u^{\alpha}-\nu_{l}\partial_{i}u^{\alpha})-(\nu_{i}\partial_{l}\partial^{\alpha}u_{j}-\nu_{l}\partial_{i}\partial^{\alpha}u_{j})\|_{\lebe^{p}(\partial\Omega)} \\ 
& + \|\nu_{i}u^{\alpha + e_{l}} - \nu_{i} \partial^{\alpha+e_{l}}u_{j}\|_{\lebe^{p}(\partial\Omega)} + \|\nu_{l}u^{\alpha+e_{i}} - \nu_{l}\partial^{\alpha+e_{i}}u_{j} \|_{\lebe^{p}(\partial\Omega)}\\ 
 & +\|\underbrace{\nu_{i}\partial_{l}\partial^{\alpha}u_{j}-\nu_{l}\partial_{i}\partial^{\alpha}u_{j}-(\nu_{i}\partial^{\alpha+e_{l}}u_{j}-\nu_{l}\partial^{\alpha+e_{i}}u_{j})}_{=0}\|_{\lebe^{p}(\partial\Omega)} =: \mathrm{I}_{j}+...+\mathrm{IV}_{j}. 
\end{align*}
Since $(\nu_{i}\partial_{l}-\nu_{l}\partial_{i})$ is a tangential differential operator on the smooth boundary $\partial\Omega$, we have with the full tangential gradient $\nabla_{\tau}$ along $\partial\Omega$ 
\begin{align*}
\mathrm{I}_{j} \leq \|\nabla_{\tau}(u^{\alpha}-\partial^{\alpha}u_{j})\|_{\lebe^{p}(\partial\Omega)} \leq \|u^{\alpha}-\partial^{\alpha}u\|_{\besov_{p,p}^{k-|\alpha|-1/p}(\partial\Omega)}\to 0,\qquad j\to\infty,
\end{align*}
as $k-|\alpha|-1/p\geq 1$ by our assumption on $\alpha$ and $p$. The terms $\mathrm{II}_{j},\mathrm{III}_{j}$ vanish in the limit by construction, and $\mathrm{IV}_{j}=0$ for all $j\in\mathbb{N}$; thus, \eqref{eq:sharpBesovsmooth} follows. Equation~\eqref{eq:GG} follows for maps $u$ and $\varphi$ which are smooth up to $\partial\Omega$ by repeated use of the usual Gau\ss -Green theorem, and then inherits to $u\in{\widetilde{\sobo}}{^{\A,p}}(\Omega)$ by means of smooth approximation. 
\end{proof}
\begin{proof}[Proof of Theorem~\ref{thm:mainLp}] Let $1<p<\infty$. 
Implication~\ref{item:mainALp}$\Rightarrow$\ref{item:mainBLp} has already been proved in Section~\ref{sec:proofhalfspace}. For implication~\ref{item:mainBLp}$\Rightarrow$\ref{item:mainCLp}, we assume that $\Omega$ is connected, open and bounded with $\hold^{\infty}$-boundary; if it is not connected, we apply the following to each of the finitely many connected components of $\Omega$. Let $u\in\hold^{\infty}(\overline{\Omega};V)$. Adopting the notation from \eqref{eq:BesovBoundary}, for $j\in\{1,...,N\}$, \ref{item:mainBLp} and \eqref{eq:BergLoef} yield for any $\alpha\in\mathbb{N}_{0}^{n}$ with $|\alpha|\leq k-1$ and all $\beta\in\mathbb{N}_{0}^{n}$ for the compactly supported, smooth map $\mathrm{I}_{j,\beta}:=((\partial^{\beta}\psi_{j})u)({F^{(j)}}^{-1}(\cdot,0))$
\begin{align*}
\|\partial^{\alpha}\mathrm{I}_{j,\beta}&\|_{{\besov}{_{p,p}^{k-|\alpha|-1/p}}(\R^{n-1})}  \stackrel{\eqref{eq:BergLoef}}{\leq}  c\|\partial^{\alpha}\mathrm{I}_{j,\beta}\|_{\lebe^{p}(\R^{n-1})} + c\|\partial^{\alpha}\mathrm{I}_{j,\beta}\|_{{\dot\besov}{_{p,p}^{k-|\alpha|-1/p}}(\R^{n-1})}\\ & \;\leq c\|\partial^{\alpha}\mathrm{I}_{j,\beta}\|_{{\dot\besov}{_{p,p}^{k-|\alpha|-1/p}}(\R^{n-1})}  \stackrel{\ref{item:mainBLp}}{\leq} c\|\A(((\partial^{\beta}\psi_{j})u)({F^{(j)}}^{-1}))\|_{\lebe^{p}(\mathbb{H})},
\end{align*}
where $\mathbb{H}=\R^{n-1}\times(0,\infty)$ and $c$ also depends on $\mathrm{diam}(\spt(\psi_{j}))$. Thus, by the Leibniz rule and the properties of $\psi_{j}$ and $F^{(j)}$, 
\begin{align}\label{eq:besovcenter1}
\|\partial^{\alpha}\mathrm{I}_{j,\beta}\|_{{\besov}{_{p,p}^{k-|\alpha|-1/p}}(\R^{n-1})} \leq c(\|u\|_{\sobo^{k-1,p}(\Omega)}+\|\A u\|_{\lebe^{p}(\Omega)})
\end{align}
where $c=c(\alpha,\beta,\psi_{j},F^{(j)},p,\A,\Omega)>0$. Based on \eqref{eq:besovcenter1} and again by the Leibniz rule, we conclude inductively for suitable choices of $\beta$ and by the properties of $\psi_{j},F^{(j)}$,
\begin{align*}
\|(\psi_{j}\partial^{\alpha}u)({F^{(j)}}^{-1}(\cdot,0))\|_{{\besov}{_{p,p}^{k-|\alpha|-1/p}}(\R^{n-1})} \leq c(\|u\|_{\sobo^{k-1,p}(\Omega)}+\|\A u\|_{\lebe^{p}(\Omega)})
\end{align*}
and \ref{item:mainCLp} follows.  Ad~\ref{item:mainCLp}$\Rightarrow$\ref{item:mainDLp}. Let $\ball\subset\R^{n}$ be an open ball with $\Omega\Subset\ball$. Because of \ref{item:mainCLp}, we may consider the bounded, linear operator $\mathrm{tr}_{k-1,\partial\Omega}^{\A}\colon{\widetilde{\sobo}}{^{\A,p}}(\Omega)\to{\dot\besov}{_{p}^{k-1/p}}(\partial\Omega;V)$ from Lemma~\ref{lem:GaussGreen}. We subsequently pick the bounded and linear operator $\mathcal{L}\colon {\dot\besov}{_{p}^{k-1/p}}(\partial(\ball\setminus\overline{\Omega});V)\to\sobo^{k,p}(\ball\setminus\overline{\Omega};V)$ from Lemma~\ref{lem:MMS} and fix some $\rho\in\hold_{c}^{\infty}(\R^{n};[0,1])$ with $\mathbbm{1}_{\Omega}\leq\rho\leq\mathbbm{1}_{\ball}$. Then we define for $u\in{\widetilde{\sobo}}{^{\A,p}}(\Omega)$ 
\begin{align}\label{eq:defextension}
\mathscr{E}u:=\begin{cases} u&\;\text{in}\;\Omega,\\
\rho\mathcal{L}(\mathbbm{1}_{\partial\Omega}\mathrm{tr}_{k-1,\partial\Omega}^{\A}(u))&\;\text{in}\;\R^{n}\setminus\overline{\Omega}.
\end{cases}
\end{align}
By construction of $\mathscr{E}$, the Gau\ss -Green formula \eqref{eq:GG} then implies 
\begin{align}\label{eq:absolutecontinuity}
\int_{\R^{n}}\A\mathscr{E}u\cdot\varphi\dif x = (-1)^{k} \int_{\R^{n}}\mathscr{E}u\cdot\A^{*}\varphi\dif x\qquad\text{for all}\;\varphi\in\hold_{c}^{\infty}(\R^{n};W)
\end{align}
and so we obtain $\mathscr{E}u\in{\widetilde{\sobo}}{^{\A,p}}(\R^{n})$ together with $\|\mathscr{E}u\|_{{\widetilde{\sobo}}{^{\A,p}}(\R^{n})}\leq c\|u\|_{{\widetilde{\sobo}}{^{\A,p}}(\Omega)}$ for all $u\in{\widetilde{\sobo}}{^{\A,p}}(\Omega)$ with $c>0$ solely depending on $\A$, $p$ and $\Omega$. 

Now note that, if \ref{item:mainCLp} holds, then the differential operator $\A$ is necessarily elliptic in the sense of \eqref{eq:ellipticIntro}. This can be seen by a modification of a standard construction (see, e.g., \cite{Van13} or \cite{ConGme20}): If it were not elliptic, we would find $\xi\in\R^{n}\setminus\{0\}$ and $v\in V\setminus\{0\}$ such that $\A[\xi]v=0$. For any $h\in\hold^{\infty}(\R)$ and any open and bounded set $\Omega\subset\R^{n}$, the function $u_{h}(x):=h(x\cdot\xi)v$ belongs to $\sobo^{k-1,p}(\Omega;V)$ and satisfies both $\partial^{\alpha}u_{h}(x)=h^{(|\alpha|)}(x\cdot\xi)\xi^{\alpha}v$ and $\A u_{h}=0$. We extend $\{\xi\}$ to an orthonormal basis $\{\xi,\eta_{2},...,\eta_{n}\}$ and choose an open and bounded set $\Omega\subset\R^{n}$ with smooth boundary such that $\Gamma:=(0,1)\xi\times(0,1)\eta_{2}\times...\times(0,1)\eta_{n-1}\times\{0\}\subset\partial\Omega$ and $\Omega\subset (-2,2)\xi\times (-2,2)\eta_{2}\times...\times(-2,2)\eta_{n}$. Choose a sequence $(h_{j})\subset \hold^{\infty}((-2,2))$ such that $\sup_{j}\|h_{j}\|_{\sobo^{k-1,p}((-2,2))}<\infty$ but $\limsup_{j\to\infty}\|h_{j}^{(k-1)}\|_{\lebe^{q}((0,1))}=\infty$ for $q>p$. Let $\alpha\in\mathbb{N}_{0}^{n}$ with $|\alpha|=k-1$ be such that $\xi^{\alpha}\neq 0$. By construction and Fubini's theorem, $\sup_{j}\|u_{h_{j}}\|_{{\widetilde{\sobo}}{^{\A,p}}(\Omega)}<\infty$ but, with a constant $c=c(n,\xi,\alpha,v)>0$, we have for $q>p$ 
\begin{align}\label{eq:tracecontra}
\int_{\partial\Omega}|\partial^{\alpha}u_{h_{j}}|^{q}\dif^{n-1}x \geq \int_{\Gamma}|\partial^{\alpha}u_{h_{j}}|^{q}\dif^{n-1}x = c\int_{0}^{1}|h_{j}^{(k-1)}(t)|^{q}\dif t \to \infty
\end{align}
at least for a suitable subsequence. Now, if $1<p<n$, choose $q=\frac{p(n-1)}{n-p}$ and let $q>p$ be arbitrary provided $p\geq n$. Since \ref{item:mainCLp} implies that $(\partial^{\alpha}u_{h_{j}}|_{\partial\Omega})$ is bounded in ${\besov}{_{p,p}^{1-1/p}}(\partial\Omega;V)$ and so, by embedding theorems for Besov spaces, in $\lebe^{q}(\partial\Omega;V)$, \eqref{eq:tracecontra} yields a contradiction. Therefore, $\A$ must be elliptic. 

Now, \eqref{eq:absolutecontinuity}ff. and an argument as in \cite[Lem.~6.1]{GRVS} or \cite[Prop.~4.1]{ConGme20} allow to conclude \ref{item:mainDLp}: Let $\alpha\in\mathbb{N}_{0}^{n}$ satisfy $|\alpha|=k$. Consider for $\psi\in\hold_{c}^{\infty}(\R^{n};W)$ the Fourier multiplication operators $\psi\mapsto \Phi_{\alpha}(\psi)$, where 
\begin{align}\label{eq:diffopFourier}
\Phi_{\alpha}(\psi)(x):=\mathscr{F}^{-1}[\xi^{\alpha}(\A^{*}[\xi]\A[\xi])^{-1}\A^{*}[\xi]\widehat{\psi}(\xi)](x),\qquad x\in\R^{n}.
\end{align}
Then we have $\Phi_{\alpha}(\A\varphi)=\partial^{\alpha}\varphi$ for $\varphi\in\hold_{c}^{\infty}(\R^{n};V)$. Since $\A$ is elliptic, $\A[\xi]\colon V\to W$ is injective and hence $\A^{*}[\xi]\A[\xi]\colon V\to V$ is bijective for any $\xi\in\mathbb{R}^{n}\setminus\{0\}$; in particular, it is bijective for any $\xi\in\R^{n}\setminus\{0\}$, and so $\eqref{eq:diffopFourier}$ is well-defined. The Fourier multiplier 
\begin{align}\label{eq:Fourmult}
\R^{n}\setminus\{0\}\ni\xi\mapsto \mathbf{m}_{\alpha}^{\A}(\xi):=\xi^{\alpha}(\A^{*}[\xi]\A[\xi])^{-1}\A^{*}[\xi]\in\mathscr{L}(W;V)
\end{align}
is of class $\hold^{\infty}(\R^{n}\setminus\{0\};\mathscr{L}(W;V))$, homogeneous of degree zero, and so \cite[Thm.~4.13]{Duo00} yields that $\Phi_{\alpha}$ extends to a bounded operator $\Phi_{\alpha}\colon \lebe^{p}(\R^{n};W)\to\lebe^{p}(\R^{n};V)$. Thus, 
\begin{align}\label{eq:Kornintheproof}
\begin{split}
\|u&\|_{\sobo^{k,p}(\Omega)} \leq \|\mathscr{E}u\|_{\sobo^{k-1,p}(\R^{n})}+\|\nabla^{k}\mathscr{E}u\|_{\lebe^{p}(\R^{n})} \\ & \leq \|\mathscr{E}u\|_{\sobo^{k-1,p}(\R^{n})} + c \|\A\mathscr{E}u\|_{\lebe^{p}(\R^{n})}\leq c\|\mathscr{E}u\|_{{\widetilde{\sobo}}{^{\A,p}}(\R^{n})} \stackrel{\eqref{eq:absolutecontinuity}ff.}{\leq} c\|u\|_{{\widetilde{\sobo}}{^{\A,p}}(\Omega)}
\end{split}
\end{align}
for all $u\in\sobo^{\A,p}(\Omega)$ with some constant 
$c=c(\A,p,\rho,\Omega)>0$. This is  \ref{item:mainDLp}. Now, if \ref{item:mainDLp} holds, we apply Lemma~\ref{lem:PTequiv} to $E_{1}=\sobo^{k,p}(\Omega;V)$,
$E_{2}=\lebe^{p}(\Omega;W)$, $E_{3}=\sobo^{k-1,p}(\Omega;V)$, $A=\A$ and $B$ being the compact embedding
$\sobo^{k,p}(\Omega;V)\hookrightarrow\sobo^{k-1,p}(\Omega;V)$. The equivalence of the $\sobo^{k,p}$- and ${\widetilde{\sobo}}{^{\A,p}}$-norms implied by \ref{item:mainDLp} now yields $\dim(\ker(A))<\infty$ by virtue of Lemma~\ref{lem:PTequiv}, and by \eqref{eq:deg-pol}ff., this implies that $\A$ is $\mathbb{C}$-elliptic. Hence \ref{item:mainA} follows, and so all of \ref{item:mainALp}--\ref{item:mainDLp} are equivalent. 

Lastly, by equivalence of \ref{item:mainALp} and \ref{item:mainDLp}, \eqref{eq:Kornvariant} follows if we can establish $\sobo^{\A,p}(\Omega)\hookrightarrow{\widetilde{\sobo}}{^{\A,p}}(\Omega)$ provided any of the mutually equivalent conditions \ref{item:mainALp}--\ref{item:mainDLp} holds. We thus assume that $\A$ is $\mathbb{C}$-elliptic and assume without loss of generality that $\Omega$ is an open, connected and bounded domain with $\hold^{\infty}$-boundary. Since every such domain is John, the Poincar\'{e}-type inequality \eqref{eq:poincare-star2} now yields $\sobo^{\A,p}(\Omega)\hookrightarrow{\widetilde{\sobo}}{^{\A,p}}(\Omega)$, and the proof is complete. 
\end{proof}
By Theorem~\ref{thm:mainLp} and Lemma~\ref{lem:GaussGreen}, we obtain the following theorem for $1<p<\infty$; for $p=1$ and $\bv^{\A}$, $\hold^{\infty}\cap\bv^{\A}$ in $\bv^{\A}$ is only dense in $\bv^{\A}$ for the $\A$-strict metric and not the norm topology, and so the underlying approximation must be approached differently. For the reader's convenience, the precise argument (hinging on a multiplicative inequality) is given in the Appendix, Section~\ref{sec:BVAstrict}.
\begin{theorem}[Trace operator]\label{rem:tracestrict}
Let $\Omega\subset\R^{n}$ be open and bounded with boundary $\partial\Omega$ of class $\hold^{\infty}$. Then, for any $k$-th order $\mathbb{C}$-elliptic differential operator $\A$ of the form \eqref{eq:form} and $1\leq p <\infty$, there exists a bounded and linear trace operator 
\begin{align*}
\mathrm{Tr}_{\partial\Omega}\colon\sobo^{\A,p}(\Omega)\to{\besov}{_{p,p}^{k-1/p}}(\partial\Omega;V)
\end{align*}
which is onto. Moreover, for $p=1$, there exists a linear trace operator
\begin{align*}
&\mathrm{Tr}_{\partial\Omega}\colon\bv^{\A}(\Omega)\to{\besov}{_{1,1}^{k-1}}(\partial\Omega;V)
\end{align*} 
which is onto and continuous for the $\A$-strict metric defined by \eqref{eq:Astrictmetric}. 
\end{theorem}
%\begin{proof}
%By Theorem~\ref{thm:poincare-john}, $\bv^{\A}(\Omega)=\widetilde{\bv}{^{\A}}(\Omega)$. Given $u\in\bv^{\A}(\Omega)$, we may thus choose a sequence $(u_{j})\subset\sobo^{\A,1}(\Omega)$ such that $\widetilde{d}_{\A}(u_{j},u)\to 0$ as $j\to\infty$. We pick a sequence $(\rho_{l})\subset\hold^{\infty}(\overline{\Omega};[0,1])$ such that $\rho_{l}=1$ in a $\tfrac{1}{l}$-neighbourhood of $\partial\Omega$ and $\rho_{l}(x)=0$ whenever $\dista(x,\partial\Omega)>\frac{2}{l}$. Since $|\A u|$ is Radon, we find a measurable set $\Omega_{l}\subset\Omega$ with $\spt(\rho_{l})\subset\Omega_{l}\subset\{x\in\Omega\colon\;\dista(x,\partial\Omega)<\frac{2}{l}\}$ with $|\A u|(\Omega\cap\partial\Omega_{l})=0$. With the trace operator $\trace_{\partial\Omega}\colon\sobo^{\A,1}(\Omega)\to {\besov}{_{1,1}^{k-1}}(\partial\Omega;V)$ we then have 
%\begin{align*}
%\|\trace_{\partial\Omega}(u_{j})-\trace_{\partial\Omega}(u_{m})\|_{{\besov}{_{1,1}^{k-1}}(\partial\Omega)} & \leq c \sum_{\substack{|\alpha|+|\beta|\leq k,\\ |\beta|\leq k-1}}\|(\partial^{\alpha}\rho_{l})\partial^{\beta}(u_{j}-u_{m})\|_{\lebe^{1}(\Omega)} \\ 
%& + c \|\rho_{l}\A u_{j}\|_{\lebe^{1}(\Omega)} + c\|\rho_{l}\A u_{m}\|_{\lebe^{1}(\Omega)}. 
%\end{align*}
%First sending $j,m\to\infty$, using that $\A u_{j}\mres\Omega\stackrel{*}{\rightharpoonup}\A u\mres\Omega$ and then letting $l\to\infty$, 
%\begin{align*}
%\lim_{j,m\to\infty}\|\trace_{\partial\Omega}(u_{j})-\trace_{\partial\Omega}(u_{m})\|_{{\besov}{_{1,1}^{k-1}}(\partial\Omega)} \leq c|\A u|(\Omega'_{l})\to 0,\;\; l\to\infty. 
%\end{align*}
%\end{proof}
\section{Proof of Theorem~\ref{thm:main1}: The trace estimate via reduction to $\sobo^{k,1}(\Omega;V)$}\label{sec:trac-reduct-class}
Let $p=1$. In comparison with the previous section, we now establish the novel reduction principle \eqref{eq:reduxMP} which leads to the trace estimates underlying Theorem~\ref{thm:main1}. This manifests the metaprinciple of $\mathbb{C}$-ellipticity helping to overcome \textsc{Ornstein}'s Non-Inequality by passing \emph{to inverse estimates for polynomials}. Throughout this section, we hereafter suppose that $\A$ is a $\mathbb{C}$-elliptic differential operator of the form \eqref{eq:form} with $k\geq 2$.

\subsection{Geometric setup}\label{sec:geometric}
Throughout this section, we suppose that $\Omega\subset\R^{n}$ is an open and bounded $\mathrm{NTA}_{n-1}$ domain, cf.~Definition~\ref{def:NTA}. To arrive at the conclusion of Theorem~\ref{thm:main1}, we borrow and adapt some technology from the precursor \cite{BreDieGme20} to the present paper. Given $j\in\mathbb{Z}$, pick a countable covering $(\ball_{j,i})_{i\in\mathbb{N}}$ of $\R^{n}$ by balls, each having diameter $\diam(\ball_{j,i})=2r(\ball_{j,i})$ and satisfying each of the following properties: 
\begin{enumerate}[label={(C\arabic{*})},start=1]
\item\label{item:C1} For all $j\in\mathbb{Z},i\in\mathbb{N}$, $2^{-j-4}\leq \ell(\ball_{j,i}) \leq 2^{-j-3}$, 
\item\label{item:C2} For each $j\in\mathbb{Z}$, $\bigcup_{i\in\mathbb{N}}\frac{7}{8}\ball_{j,i}=\R^{n}$, 
\item\label{item:C3} There exists $c>0$ such that $\sup_{j\in\mathbb{Z}}\sum_{i\in\mathbb{N}}\mathbbm{1}_{\ball_{j,i}}\leq c$.
\end{enumerate}
In the sequel, we let $\Omega_{j}:=\{x\in\Omega\colon\;\dista(x,\partial\Omega)<2^{-j}\}$. Since $\Omega$ is a non-tangentially accessible domain and thereby satisfies the interior corkscrew as well as the Harnack chain condition, we may record the following further properties, see \cite{BreDieGme20}: 
\begin{enumerate}[label={(C\arabic{*})},start=4]
\item\label{item:C4} There exists an index $j_{0}\in\mathbb{Z}$ such that for any ball $\ball_{j,i}$ with $j\geq j_{0}$, $i\in\mathbb{N}$ and $\ball_{j,i}\cap \Omega_{j}\neq\emptyset$, there exists a \emph{reflected ball} $\ball_{j,i}^{\sharp}\subset\Omega$ satisfying 
\begin{align*}
\diam(\ball_{j,i})\simeq\diam(\ball_{j,i}^{\sharp})\simeq \dista(\ball_{j,i}^{\sharp},\partial\Omega),\;\;\;\dista(\ball_{j,i},\ball_{j,i}^{\sharp})\lesssim\diam(\ball_{j,i}). 
\end{align*}
Here, the constants implicit in '$\simeq$' or '$\lesssim$' are independent of $j$ and $i$.
\item\label{item:C5} With $j_{0}$ as in \ref{item:C4}, if $\ball_{j,i}\subset\Omega$ and $\ball_{j,i}\cap\Omega_{j}\neq\emptyset$ for some $j\geq j_{0}$, then there exists a chain of balls $\mathbb{B}_{1},...,\mathbb{B}_{\gamma}\subset\Omega$ with $\gamma\in\mathbb{N}$ independent of $j,i$ such that 
\begin{enumerate}
\item $\mathbb{B}_{1}=\ball_{j,i}$, $\mathbb{B}_{\gamma}=\ball_{j,i}^{\sharp}$, 
\item $\mathscr{L}^{n}(\mathbb{B}_{\beta}\cap\mathbb{B}_{\beta+1})\simeq \mathscr{L}^{n}(\mathbb{B}_{\beta})\simeq\mathscr{L}^{n}(\mathbb{B}_{\beta+1})\simeq\mathscr{L}^{n}(\ball_{j,i})$ for all $\beta\in\{1,...,\gamma-1\}$, 
\item $\diam(\mathbb{B}_{\beta})\simeq \diam(\ball_{j,i})$ for all $\beta\in\{1,...,\gamma\}$. 
\end{enumerate}
Here, the constants implicit in '$\simeq$' are independent of $j,i$ and $\beta$. In this situation, we define the \emph{chain set}
$\ch(\ball_{j,i},\ball_{j,i}^{\sharp})=\bigcup_{\beta=1}^{\gamma}\mathbb{B}_{\beta}$.
\item\label{item:C6} If $j_{0}$ is as in \ref{item:C4} and $j,l\in\mathbb{Z}$ are such that $j,l\geq j_{0}$ and $|j-l|\leq 1$, then for all balls $\ball_{j,i}$ and $\ball_{l,m}$ with $\ball_{j,i}\cap\ball_{l,m}\neq\emptyset$, $\ball_{j,i}\cap\Omega_{j}\neq\emptyset$, $\ball_{l,m}\cap\Omega_{l}\neq\emptyset$ there exists a chain of balls $\mathbb{B}_{1},...,\mathbb{B}_{\gamma}\subset\Omega$ with $\gamma$ independent of $j,l,m,i$ such that 
\begin{enumerate}
\item $\mathbb{B}_{1}=\ball_{j,i}^{\sharp}$, $\mathbb{B}_{\gamma}=\ball_{l,m}^{\sharp}$, 
\item $\mathscr{L}^{n}(\mathbb{B}_{\beta}\cap\mathbb{B}_{\beta+1})\simeq \mathscr{L}^{n}(\mathbb{B}_{\beta})\simeq \diam(\ball_{j,i})$ for all $\beta\in\{1,...,\gamma-1\}$,
\item $\dista(\mathbb{B}_{\beta},\partial\Omega)\simeq \diam(\mathbb{B}_{\beta})\simeq \diam(\ball_{j,i})$ for all $\beta\in\{1,...,\gamma\}$. 
\end{enumerate}
Here, the constants implicit in '$\simeq$' are independent of $j,i$ and $\beta$, and we define similarly as above the respective chain set $\ch(\ball_{j,i}^{\sharp},\ball_{l,m}^{\sharp}):=\bigcup_{\beta=1}^{\gamma}\mathbb{B}_{\beta}$.
\end{enumerate}
Working from \ref{item:C4}--\ref{item:C6}, we further obtain the existence of an $i_{0}\in\mathbb{N}_{\geq 2}$ such that for all $j\geq j_{0}$ there holds with a universal constant $c>0$
\begin{align}\label{eq:sharpsum}
\begin{split}
&\sum_{\substack{m\colon\\ \ball_{j,m}\cap\Omega_{j}\neq\emptyset}}\mathbbm{1}_{\ball_{j,m}^{\sharp}}\leq c\mathbbm{1}_{\Omega_{j-i_{0}}\setminus\Omega_{j+i_{0}}},\\
&\sum_{\substack{m\colon\\ \ball_{j,m}\cap\Omega_{j}\neq\emptyset}}\sum_{\substack{i\colon\\\ball_{j+1,i}\cap\ball_{j,m}\neq\emptyset\\ \ball_{j+1,i}\cap\Omega_{j+1}\neq\emptyset}}\mathbbm{1}_{\ch(\ball_{j,m}^{\sharp},\ball_{j+1,i}^{\sharp})}\leq c\mathbbm{1}_{\Omega_{j-i_{0}}\setminus\Omega_{j+i_{0}}},\\
& \sum_{\substack{m\colon\\ \ball_{j,m}\cap (\Omega_{j}\setminus\Omega_{j+2})\neq\emptyset}}\mathbbm{1}_{\ch(\ball_{j,m},\ball_{j,m}^{\sharp})}\leq c\mathbbm{1}_{\Omega_{j-i_{0}}\setminus\Omega_{j+i_{0}}}.
\end{split}
\end{align}
For $\eqref{eq:sharpsum}_{3}$, note that if $\ball_{j,m}\cap (\Omega_{j}\setminus\Omega_{j+2})\neq\emptyset$, then we have $\ball_{j,m}\subset\Omega$ by construction. Given $u\in\sobo^{\A,1}(\Omega)$ and $j\geq j_{0}$, $i\in\mathbb{N}$, we introduce the shorthand notation
\begin{align}\label{eq:projdef}
\Pi_{j,i}u:={\Pi}{_{\A}^{\ball_{j,i}^{\sharp}}}u, 
\end{align}
with the projection ${\Pi}{_{\A}^{\ball_{j,i}^{\sharp}}}$ from Proposition~\ref{prop:decomposition}. The proof of the next lemma is accomplished by an argument similar to that underlying Proposition~\ref{prop:john-riesz-est} and is therefore omitted. 
\begin{lemma}\label{lem:aux}
There exists a constant $c>0$ that only depends on $\A$ and the \emph{NTA}-parameters of $\Omega$ such that the following hold for all $j\geq j_{0}$ and all $u\in\sobo^{\A,1}(\Omega)$: 
\begin{enumerate}
\item\label{item:proj1} For all $\ball_{j,i}$ with $\ball_{j,i}\cap\Omega_{j}\neq\emptyset$ and all $\alpha\in\mathbb{N}_{0}^{n}$ there holds  
\begin{align*}
\|\partial^{\alpha}\Pi_{j,i}u\|_{\lebe^{\infty}(\ball_{j,i})}\leq c\diam(\ball_{j,i})^{-|\alpha|}\dashint_{\ball_{j,i}^{\sharp}}|u|\dif x.
\end{align*}
\item\label{item:proj2} Whenever $\ball_{j,m}\cap (\Omega_{j}\setminus\Omega_{j+2})\neq\emptyset$, then $\ball_{j,m}\subset\Omega$ and we have, for all $\alpha\in\mathbb{N}_{0}^{n}$ with $|\alpha|\leq k-1$,
\begin{align*}
\|\partial^{\alpha}(u-\Pi_{j,m}u)\|_{\lebe^{1}(\ball_{j,m})}\leq c\diam(\ball_{j,m})^{k-|\alpha|}\int_{\ch(\ball_{j,m},\ball_{j,m}^{\sharp})}|\A u|\dif x.
\end{align*}
\item\label{item:proj3} \emph{Chain control:} Whenever $\ball_{j,m}\cap\Omega_{j}\neq\emptyset$, $\ball_{j+1,i}\cap\Omega_{j+1}\neq\emptyset$ and $\ball_{j+1,i}\cap\ball_{j,m}\neq\emptyset$, then we have for all $\alpha\in\mathbb{N}_{0}^{n}$ with $|\alpha|\leq k-1$
\begin{align*}
\dashint_{\ball_{j,m}}|\partial^{\alpha}(\Pi_{j+1,i}u-\Pi_{j,m}u)|\dif x\leq c \diam(\ball_{j,m})^{k-|\alpha|}\dashint_{\ch(\ball_{j+1,i}^{\sharp},\ball_{j,m}^{\sharp})}|\A u|\dif x.
\end{align*}
\end{enumerate}
\end{lemma}

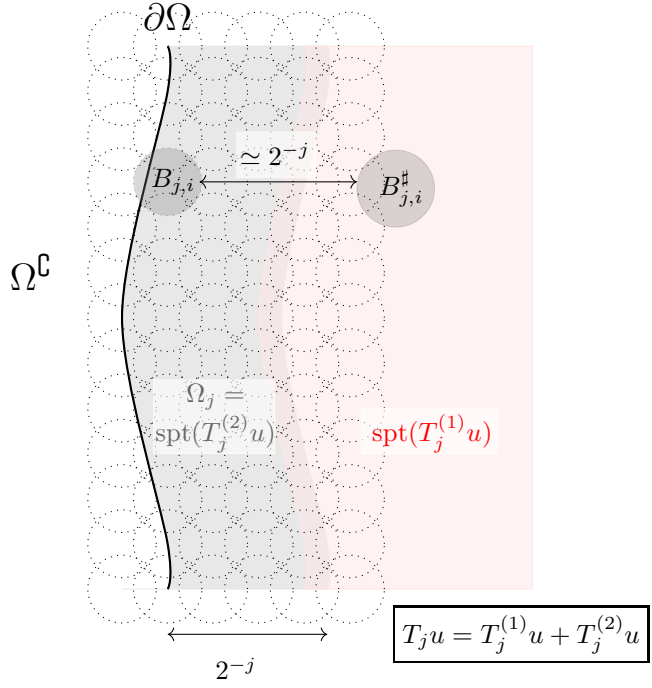
\begin{figure}
\begin{tikzpicture}[scale=3]
  \begin{scope}
    \draw  plot [smooth] coordinates { (0.1,-1.2) (0.1,-1) (-.1,0) (0.1,1) (0.1,1.2)};
    \draw [ color=black, thick] plot [smooth] coordinates { (0.7,-1.2) (0.7,-1) (0.5,0) (0.7,1) (0.7,1.2)};
       \draw [fill= black!9!white, color=black!9!white, fill opacity=1] plot [smooth] coordinates { (0.1,-1.2) (0.1,-1) (-0.1,0) (0.1,1) (0.1,1.2)} -- (0.4,1.2) -- (0.4,-1.2) -- (-0.1,-1.2);
      \draw [fill= black!9!white, color=black!9!white, fill opacity=1] plot [smooth] coordinates { (0.8,-1.2) (0.8,-1) (0.6,0) (0.8,1) (0.8,1.2)} -- (0.4,1.2) -- (0.4,-1.2) -- (0.8,-1.2);
       \draw [fill= red!30!white, color=red!10!white, fill opacity=0.5] plot [smooth] coordinates { (0.7,-1.2) (0.7,-1) (0.5,0) (0.7,1) (0.7,1.2)} -- (1.7,1.2) -- (1.7,-1.2) -- (0.7,-1.2);
            \draw[<->] (0.1,-1.4) -- (0.8,-1.4);
      \node at (0.4,-1.55) {$2^{-j}$};
      \node at (0.1,1.33) {\LARGE $\partial\Omega$};
      \node at (-0.5,0.2) {\LARGE $\Omega^{\complement}$};
      \draw [ color=black, thick] plot [smooth] coordinates { (0.1,-1.2) (0.1,-1) (-0.1,0) (0.1,1) (0.1,1.2)};
      %covering
      \draw[dotted] (-0.1,-1.2) circle (0.15);
      \draw[dotted] (-0.1,-1.0) circle (0.15);
      \draw[dotted] (-0.1,-0.8) circle (0.15);
      \draw[dotted] (-0.1,-0.6) circle (0.15);
      \draw[dotted] (-0.1,-0.4) circle (0.15);
      \draw[dotted] (-0.1,-0.2) circle (0.15);
      \draw[dotted] (-0.1,-0.0) circle (0.15);
      \draw[dotted] (-0.1,0.2) circle (0.15);
      \draw[dotted] (-0.1,0.4) circle (0.15);
      \draw[dotted] (-0.1,0.6) circle (0.15);
      \draw[dotted] (-0.1,0.8) circle (0.15);
      \draw[dotted] (-0.1,1.0) circle (0.15);
      \draw[dotted] (-0.1,1.2) circle (0.15);
      %newcolumn
      \draw[dotted] (0.1,-1.2) circle (0.15);
      \draw[dotted] (0.1,-1.0) circle (0.15);
      \draw[dotted] (0.1,-0.8) circle (0.15);
      \draw[dotted] (0.1,-0.6) circle (0.15);
      \draw[dotted] (0.1,-0.4) circle (0.15);
      \draw[dotted] (0.1,-0.2) circle (0.15);
      \draw[dotted] (0.1,-0.0) circle (0.15);
      \draw[dotted] (0.1,0.2) circle (0.15);
      \draw[dotted] (0.1,0.4) circle (0.15);
      \draw[dotted, fill=gray, opacity = 0.3] (0.1,0.6) circle (0.15);
      \draw[color=gray, fill=gray, opacity = 0.3] (1.1,0.57) circle (0.17);
      \node at (0.125,0.6) {$\ball_{j,i}$};
      \node at (1.125,0.5725) {$\ball_{j,i}^{\sharp}$};
      \draw [<->] (0.24,0.6) -- (0.93,0.6);
      \draw[dotted] (0.1,0.8) circle (0.15);
      \draw[dotted] (0.1,1.0) circle (0.15);
      \draw[dotted] (0.1,1.2) circle (0.15);
      %newcolumn
      \draw[dotted] (0.3,-1.2) circle (0.15);
      \draw[dotted] (0.3,-1.0) circle (0.15);
      \draw[dotted] (0.3,-0.8) circle (0.15);
      \draw[dotted] (0.3,-0.6) circle (0.15);
      \draw[dotted] (0.3,-0.4) circle (0.15);
      \draw[dotted] (0.3,-0.2) circle (0.15);
      \draw[dotted] (0.3,-0.0) circle (0.15);
      \draw[dotted] (0.3,0.2) circle (0.15);
      \draw[dotted] (0.3,0.4) circle (0.15);
      \draw[dotted] (0.3,0.6) circle (0.15);
      \draw[dotted] (0.3,0.8) circle (0.15);
      \draw[dotted] (0.3,1.0) circle (0.15);
      \draw[dotted] (0.3,1.2) circle (0.15);
      %\newcolumn
      \draw[dotted] (0.5,-1.2) circle (0.15);
      \draw[dotted] (0.5,-1.0) circle (0.15);
      \draw[dotted] (0.5,-0.8) circle (0.15);
      \draw[dotted] (0.5,-0.6) circle (0.15);
      \draw[dotted] (0.5,-0.4) circle (0.15);
      \draw[dotted] (0.5,-0.2) circle (0.15);
      \draw[dotted] (0.5,-0.0) circle (0.15);
      \draw[dotted] (0.5,0.2) circle (0.15);
      \draw[dotted] (0.5,0.4) circle (0.15);
      \draw[dotted] (0.5,0.6) circle (0.15);
      \draw[dotted] (0.5,0.8) circle (0.15);
      \draw[dotted] (0.5,1.0) circle (0.15);
      \draw[dotted] (0.5,1.2) circle (0.15);
      %\newcolumn
      \draw[dotted] (0.7,-1.2) circle (0.15);
      \draw[dotted] (0.7,-1.0) circle (0.15);
      \draw[dotted] (0.7,-0.8) circle (0.15);
      \draw[dotted] (0.7,-0.6) circle (0.15);
      \draw[dotted] (0.7,-0.4) circle (0.15);
      \draw[dotted] (0.7,-0.2) circle (0.15);
      \draw[dotted] (0.7,-0.0) circle (0.15);
      \draw[dotted] (0.7,0.2) circle (0.15);
      \draw[dotted] (0.7,0.4) circle (0.15);
      \draw[dotted] (0.7,0.6) circle (0.15);
      \draw[dotted] (0.7,0.8) circle (0.15);
      \draw[dotted] (0.7,1.0) circle (0.15);
      \draw[dotted] (0.7,1.2) circle (0.15);
      %\newcolumn
      \draw[dotted] (0.9,-1.2) circle (0.15);
      \draw[dotted] (0.9,-1.0) circle (0.15);
      \draw[dotted] (0.9,-0.8) circle (0.15);
      \draw[dotted] (0.9,-0.6) circle (0.15);
      \draw[dotted] (0.9,-0.4) circle (0.15);
      \draw[dotted] (0.9,-0.2) circle (0.15);
      \draw[dotted] (0.9,-0.0) circle (0.15);
      \draw[dotted] (0.9,0.2) circle (0.15);
      \draw[dotted] (0.9,0.4) circle (0.15);
      \draw[dotted] (0.9,0.6) circle (0.15);
      \draw[dotted] (0.9,0.8) circle (0.15);
      \draw[dotted] (0.9,1.0) circle (0.15);
      \draw[dotted] (0.9,1.2) circle (0.15);
      \draw[fill=white, color=white, opacity=0.6] (0.95,-0.6) -- (1.55,-0.6) -- (1.55,-0.4) -- (0.95,-0.4) -- (0.95,-0.6);
      \node[red] at (1.26,-0.51) { $\mathrm{spt}(T_{j}^{(1)}u)$};
      \draw[fill=white, color=white, opacity=0.6] (0.038,-0.6) -- (0.57,-0.6) -- (0.57,-0.275) -- (0.038,-0.275) -- (0.038,-0.6);
      \node[black!60!white] at (0.31,-0.51) { $\mathrm{spt}(T_{j}^{(2)}u)$};
      \node[black!60!white] at (0.31,-0.35) { $\Omega_{j}=$};
      \draw[fill=white, color=white, opacity=0.6] (0.4,0.62) -- (0.75,0.62) -- (0.75,0.82) -- (0.4,0.82);
      \node[black] at (0.57,0.725) { $\simeq 2^{-j}$};
      \node at (1.65,-1.4) {\fbox{$T_{j}u=T_{j}^{(1)}u+T_{j}^{(2)}u$}};
  \end{scope}
\end{tikzpicture}
\caption{Construction of $T_{j}$ and idea of the (trace) estimates of Proposition~\ref{lem:convergenceofTj} and Corollary~\ref{cor:main}. With the dashed balls indicating the covering $(\ball_{j,i})_{i}$, the approximation $T_{j}u$ given by \eqref{eq:replacementsequence} splits into two parts: At a distance at least $2^{-j}$ from $\partial\Omega$ (i.e., in the red area), $T_{j}u$ leaves $u$ unchanged, whereas close to $\partial\Omega$ (i.e., on $\Omega_{j}:=\{\dist(\cdot,\partial\Omega)< 2^{-j}\}$ represented by the grey area), $T_{j}u$ essentially replaces $u$ by its projections onto the nullspace of $\A$ on the reflected balls $\ball_{j,i}^{\sharp}$.}
\label{fig:replacement}
\end{figure}
\subsection{Proof of Theorem~\ref{thm:main1}}
In view of the requisite trace estimates, we now introduce a suitable
approximation $T_{j}u$ of a given function
$u\in\sobo^{\A,1}(\Omega)$. For each $j\in\mathbb{Z}$, pick the covering $(\ball_{j,i})_{i}$ from Section~\ref{sec:geometric} and let $(\rho_{j,i})_{i}$ be a smooth partition of unity subject to $(\ball_{j,i})_{i}$ such that, with a universal constant $c>0$, 
\begin{align}\label{eq:gradbound0}
\sum_{|\alpha|\leq k}\diam(\ball_{j,i})^{|\alpha|}\|\partial^{\alpha}\rho_{j,i}\|_{\lebe^{\infty}(\Omega)}\leq c\qquad\text{for all}\;i\in\mathbb{N},j\in\mathbb{Z}.
\end{align}
Moreover, let $\rho_{j}\in\hold^{\infty}(\overline{\Omega};[0,1])$ be such that $\mathbbm{1}_{\Omega_{j+1}}\leq\rho_{j}\leq\mathbbm{1}_{\Omega_{j}}$ and 
\begin{align}\label{eq:gradbound}
\|\partial^{\alpha}\rho_{j}\|_{\lebe^{\infty}(\Omega)}\leq c 2^{j|\alpha|}\qquad\text{for all}\;\alpha\in\mathbb{N}_{0}^{n},\;|\alpha|\leq k.
\end{align}
We then define the \emph{replacement sequence} $(T_{j}u)_{j\geq j_{0}}$ with $j_{0}\in\mathbb{Z}$ as in \ref{item:C4} from above by 
\begin{align}\label{eq:replacementsequence}
T_{j}u:=T_{j}^{(1)}u+T_{j}^{(2)}u:=(1-\rho_{j})u + \rho_{j}\sum_{i\in\mathbb{N}}\rho_{j,i}\Pi_{j,i}u
\end{align}
in $\Omega$. Note that, as a locally finite sum and by construction,
$T_{j}^{(2)}u\in\hold^{\infty}(\overline{\Omega};V)$. Thus, in
passing from $u$ to $T_{j}u$ we essentially leave the map $u$
unchanged away from $\partial\Omega$ while close to $\partial\Omega$,
we replace it by its projections onto $\ker(\A)$ on balls nearby, see Figure~\ref{fig:replacement}. The proof of Theorem~\ref{thm:main1} crucially hinges on the decomposition
\begin{align}\label{eq:tracerewrite}
\begin{split}
T_{j+1}u-T_{j}u & = (\rho_{j}-\rho_{j+1})\sum_{m\in\mathbb{N}}\rho_{j,m}(u-\Pi_{j,m}u) \\ 
& \;\;\;\; + \rho_{j+1}\sum_{i,m\in\mathbb{N}}\rho_{j,m}\rho_{j+1,i}(\Pi_{j+1,i}u-\Pi_{j,m}u) =: \mathrm{I}_{j}[u] + \mathrm{II}_{j}[u],
\end{split} 
\end{align}
which is implied by $\sum_{m}\rho_{j,m}\equiv 1$ for all $j$. In comparison with the precursor \cite{BreDieGme20}, the key observation is now given by Proposition~\ref{lem:convergenceofTj} below: The terms $\mathrm{I}_{j}[u]$ will have trace zero along $\partial\Omega$ and $\mathrm{II}_{j}[u]$ \emph{uniformly belongs to} $\sobo^{k,1}(\Omega;V)$. For the reduction argument underlying the proof of Theorem~\ref{thm:main1}, it will then suffice to apply the trace estimate for $\sobo^{k,1}$ to $\mathrm{II}_{j}[u]$, cf.~Corollary~\ref{cor:main}. 
\begin{proposition}\label{lem:convergenceofTj}
Let $u\in \sobo^{\A,1}(\Omega)$ and define $T_{j}u$, $\mathrm{I}_{j}[u]$ and $\mathrm{II}_{j}[u]$ for $j\geq j_{0}$ by \eqref{eq:replacementsequence} or \eqref{eq:tracerewrite}, respectively, and let $i_{0}$ be as in \eqref{eq:sharpsum}. Then the following holds:
\begin{enumerate}
\item\label{item:conv00} There exists $c>0$ such that $\|\mathrm{I}_{j}[u]\|_{\sobo^{\A,1}(\Omega)}\leq c\|\A u\|_{\lebe^{1}(\Omega_{j-i_{0}}\setminus\Omega_{j+i_{0})}}$.
\item\label{item:conv0} There exists $c>0$ such that $\|\mathrm{II}_{j}[u]\|_{\sobo^{k,1}(\Omega)}\leq c\|\A u\|_{\lebe^{1}(\Omega_{j-i_{0}}\setminus\Omega_{j+i_{0})}}$.
\item\label{item:conv1} We have $T_{j}u\to u$ with respect to the $\sobo^{\A,1}$-norm. 
\item\label{item:conv2a} There exists $(u_{j})\subset\hold^{\infty}(\overline{\Omega};V)$ such that $u_{j}\to u$ with respect to the $\sobo^{\A,1}$-norm. 
\item\label{item:conv1a} If $u\in\hold(\overline{\Omega};V)\cap\sobo^{\A,1}(\Omega)$, then $T_{j}u\to u$ uniformly in $\Omega$. 
\end{enumerate}
In both \ref{item:conv00} and \ref{item:conv0}, the constant $c>0$ is independent of $u$. 
\end{proposition}
\begin{proof}
Let $u\in\sobo^{\A,1}(\Omega)$. Ad~\ref{item:conv00}. For $j\geq j_{0}$, denote $\mathcal{I}_{j}$ the set of all indices $m\in\mathbb{N}$ such that $(\rho_{j}-\rho_{j+1})\rho_{j,m}\neq 0$. Since $\spt(\rho_{j}-\rho_{j+1})\subset\Omega_{j}\setminus\Omega_{j+2}$, $m\in\mathcal{I}_{j}$ implies that $\ball_{j,m}\cap (\Omega_{j}\setminus\Omega_{j+2})\neq\emptyset$. For such indices, we may use Lemma~\ref{lem:aux}~\ref{item:proj2}. In consequence, 
\begin{align*}
\|\mathrm{I}_{j}[u]\|_{\sobo^{\A,1}(\Omega)} & \leq \Big(\sum_{m\in\mathcal{I}_{j}}\sum_{l=0}^{k-1}\|D^{\ell}((\rho_{j}-\rho_{j+1})\rho_{j,m}(u-\Pi_{j,m}u))\|_{\lebe^{1}(\Omega)}\Big) \\ 
& \!\!\!\!\!\!\!\!\!\!\!\!\!\!\!\!\!\!\!\!\!\!\!\! + \Big(\sum_{m\in\mathcal{I}_{j}}\|\A ((\rho_{j}-\rho_{j+1})\rho_{j,m}(u-\Pi_{j,m}u))\|_{\lebe^{1}(\Omega)}\Big) \\ & \!\!\!\!\!\!\!\!\!\!\!\!\!\!\!\!\!\!\!\!\!\!\!\!\!\stackrel{(*)}{\leq} c\Big(\sum_{m\in\mathcal{I}_{j}}\sum_{l=0}^{k-1}\sum_{|\alpha|+|\beta|=\ell}\|\partial^{\alpha}((\rho_{j}-\rho_{j+1})\rho_{j,m})\|_{\lebe^{\infty}(\Omega)}\|\partial^{\beta}(u-\Pi_{j,m}u)\|_{\lebe^{1}(\ball_{j,m})}\Big) \\ 
& \!\!\!\!\!\!\!\!\!\!\!\!\!\!\!\!\!\!\!\!\!\!\!\! + c\Big(\sum_{m\in\mathcal{I}_{j}}\sum_{\substack{|\alpha|+|\beta|=k \\ |\beta|\leq k-1}}\|\partial^{\alpha}((\rho_{j}-\rho_{j+1})\rho_{j,m})\|_{\lebe^{\infty}(\Omega)}\|\partial^{\beta}(u-\Pi_{j,m}u)\|_{\lebe^{1}(\ball_{j,m})} \Big.\\ 
& + \sum_{m\in\mathcal{I}_{j}}\|(\rho_{j}-\rho_{j+1})\rho_{j,m}\A u\|_{\lebe^{1}(\Omega)}\Big) \\ 
& \!\!\!\!\!\!\!\!\!\!\!\!\!\!\!\!\!\!\!\!\!\!\!\!\!\!\!\!\!\!\!\!\!\!\!\!\!\!\!\!\!\!\stackrel{\text{Lem.}~\ref{lem:aux}\ref{item:proj2}, \eqref{eq:gradbound0}, \eqref{eq:gradbound}}{\leq} c\Big(\sum_{m\in\mathcal{I}_{j}}\sum_{l=0}^{k-1}\sum_{|\alpha|+|\beta|=\ell}2^{j|\alpha|}2^{-j(k-|\beta|)}\|\A u\|_{\lebe^{1}(\mathrm{ch}(\ball_{j,m},\ball_{j,m}^{\sharp}))}\Big) \\ 
& \!\!\!\!\!\!\!\!\!\!\!\!\!\!\!\!\!\!\!\!\!\!\!\! + c\Big(\sum_{m\in\mathcal{I}_{j}}\sum_{\substack{|\alpha|+|\beta|=k \\ |\beta|\leq k-1}}2^{j|\alpha|}2^{-j(k-|\beta|)}\|\A u\|_{\lebe^{1}(\mathrm{ch}(\ball_{j,m},\ball_{j,m}^{\sharp}))} + \sum_{m\in\mathcal{I}_{j}}\|\A u\|_{\lebe^{1}(\ball_{j,m})} \Big)\\ 
& \!\!\!\!\!\!\!\!\!\!\!\!\!\!\!\!\!\!\!\!\!\!\!\! \leq c\sum_{m\in\mathcal{I}_{j}}\|\A u\|_{\lebe^{1}(\mathrm{ch}(\ball_{j,m},\ball_{j,m}^{\sharp}))}\\ 
& \!\!\!\!\!\!\!\!\!\!\!\!\!\!\!\!\!\!\!\!\!\!\!\!\!\!\!\stackrel{\eqref{eq:sharpsum}_{3}}{\leq} c \|\A u\|_{\lebe^{1}(\Omega_{j-i_{0}}\setminus\Omega_{j+i_{0}})}. 
\end{align*}
Here we used at $(*)$ that, if we expand $\A(\varphi v)$ for sufficiently regular functions $\varphi\colon\Omega\to\R$ and $v\colon \Omega\to V$ by the Leibniz rule, then the only summand that involves $k$-th order derivatives of $v$ is of the form $\varphi\A v$. Hence, \ref{item:conv00} follows. Ad~\ref{item:conv0}. Let $\mathcal{J}_{j}^{m}$ be the set of all indices $i\in\mathbb{N}$ such that $\rho_{j+1}\rho_{j,m}\rho_{j+1,i}\neq 0$. Then, by \ref{item:C3}, $\# \mathcal{J}_{j}^{m}$ is uniformly bounded in $j$ and $m$. If $i\in\mathcal{J}_{j}^{m}$, then $\Omega_{j+1}\cap\ball_{j,m}\neq\emptyset$, $\Omega_{j+1}\cap\ball_{j+1,i}\neq\emptyset$ and $\ball_{j,m}\cap\ball_{j+1,i}\neq\emptyset$. As such, $\eqref{eq:sharpsum}_{2}$ is available for such indices. Let $|\beta|\leq k$. Since $\ker(\A)$ is a finite dimensional space by the $\mathbb{C}$-ellipticity of $\A$, we may invoke the inverse estimates $\eqref{eq:inverseFull}_{2}$ on elements of the nullspace of $\A$ to obtain for all $j\geq j_{0}$, $m\in\mathbb{N}$ and $i\in\mathcal{J}_{j}^{m}$: 
\begin{align*}
\|\partial^{\beta}(&\Pi_{j+1,i}u-\Pi_{j,m}u)\|_{\lebe^{1}(\ball_{j,m})} \leq \|D^{|\beta|}(\Pi_{j+1,i}u-\Pi_{j,m}u)\|_{\lebe^{1}(\ball_{j,m})} \\
& \;\;\leq C\diam(\ball_{j,m})^{n-|\beta|}\diam(\ball_{j,m})^{|\beta|}\dashint_{\ball_{j,m}}|D^{|\beta|}(\Pi_{j+1,i}u-\Pi_{j,m}u)|\dif x\\
& \;\;\leq C\diam(\ball_{j,m})^{n-|\beta|} \times \\ 
& \;\;\;\;\;\;\;\;\;\;\;\;\;\;\;\times \Big(\sum_{\ell=0}^{k}\diam(\ball_{j,m})^{\ell}\dashint_{\ball_{j,m}}|D^{\ell}(\Pi_{j+1,i}u-\Pi_{j,m}u)|\dif x\Big)\\
& \!\stackrel{\eqref{eq:inverseFull}_{2}}{\leq} C\diam(\ball_{j,m})^{n-|\beta|}\dashint_{\ball_{j,m}}|\Pi_{j+1,i}u-\Pi_{j,m}u|\dif x\\
& \!\!\!\!\!\stackrel{\text{Lem.~\ref{lem:aux}~\ref{item:proj3}}}{\leq} C\diam(\ball_{j,m})^{(k-|\beta|)}\|\A u\|_{\lebe^{1}(\ch(\ball_{j+1,i}^{\sharp},\ball_{j,m}^{\sharp}))}\\
& \,\,\,\leq C 2^{-j(k-|\beta|)}\|\A u\|_{\lebe^{1}(\ch(\ball_{j+1,i}^{\sharp},\ball_{j,m}^{\sharp}))}.
\end{align*}
As $\mathcal{J}_{j}^{m}$ is non-empty only if $m\in\mathcal{K}_{j}:=\{l\in\mathbb{N}\colon\;\Omega_{j+1}\cap\ball_{j,l}\neq\emptyset\}$, we obtain by \eqref{eq:gradbound0}, \eqref{eq:gradbound} and the preceding inequality
\begin{align*}
\|\mathrm{II}_{j}[u]\|_{\sobo^{k,1}(\Omega)} & \leq c \sum_{m\in\mathcal{K}_{j}}\sum_{i\in\mathcal{J}_{j}^{m}}\sum_{\ell=0}^{k}\sum_{|\alpha|+|\beta|=\ell}\Big(\|\partial^{\alpha}(\rho_{j+1}\rho_{j,m}\rho_{j+1,i})\|_{\lebe^{\infty}(\Omega)} \times \Big. \\ & \Big. \;\;\;\;\;\;\;\;\;\;\;\;\;\;\;\;\;\;\;\;\;\;\;\;\;\;\;\;\;\;\;\;\;\;\;\;\;\;\;\;\;\;\;\;\times\|\partial^{\beta}(\Pi_{j+1,i}u-\Pi_{j,m}u)\|_{\lebe^{1}(\ball_{j,m})}\Big)\\ 
& \leq c\sum_{m\in\mathcal{K}_{j}}\sum_{i\in\mathcal{J}_{j}^{m}}\sum_{\ell=0}^{k}\sum_{|\alpha|+|\beta|=\ell} 2^{j|\alpha|}2^{-j(k-|\beta|)}\|\A u\|_{\lebe^{1}(\ch(\ball_{j+1,i}^{\sharp},\ball_{j,m}^{\sharp}))} \\ 
& \leq c\sum_{m\in\mathcal{K}_{j}}\sum_{i\in\mathcal{J}_{j}^{m}}\sum_{\ell=0}^{k}2^{(\ell - k)j}\|\A u\|_{\lebe^{1}(\ch(\ball_{j+1,i}^{\sharp},\ball_{j,m}^{\sharp}))} \\ 
& \leq c\sum_{m\in\mathcal{K}_{j}}\sum_{i\in\mathcal{J}_{j}^{m}}\|\A u\|_{\lebe^{1}(\ch(\ball_{j+1,i}^{\sharp},\ball_{j,m}^{\sharp}))} \\ & \!\!\!\stackrel{\eqref{eq:sharpsum}_{2}}{\leq} c \|\A u\|_{\lebe^{1}(\Omega_{j-i_{0}}\setminus\Omega_{j+i_{0}})}
\end{align*}
This establishes \ref{item:conv0}. Ad~\ref{item:conv1}. By construction of the sets $\Omega_{j}$, $\sum_{j\geq l}\mathbbm{1}_{\Omega_{j-i_{0}}\setminus\Omega_{j+i_{0}}}\leq c(i_{0})\mathbbm{1}_{\Omega_{l-i_{0}}}$ for all $l\in\mathbb{N}$. Since $\|\cdot\|_{\sobo^{\A,1}}$ always can be dominated by $\|\cdot\|_{\sobo^{k,1}}$, we obtain by \ref{item:conv00} and \ref{item:conv0} for all $l'\geq l\geq j_{0}$ by virtue of a telescope sum argument
\begin{align*}
\|T_{l'}u-T_{l}u\|_{\sobo^{\A,1}(\Omega)}\leq \sum_{j=l}^{\infty}\|T_{j+1}u-T_{j}u\|_{\sobo^{\A,1}(\Omega)} \leq c \|\A u\|_{\lebe^{1}(\Omega_{l-i_{0}})}\to 0
\end{align*}
as $l,l'\to\infty$. Thus $(T_{l}u)$ is Cauchy in $\sobo^{\A,1}(\Omega)$ and thus converges to some $v\in\sobo^{\A,1}(\Omega)$. On the other hand, since by construction there holds $\spt(u-T_{l}u)\subset\Omega_{l}$, we have $\|u-T_{l}u\|_{\sobo^{\A,1}(\omega)}\to 0$ for any open $\omega\Subset\Omega$ as $l\to\infty$, and so we must have $v=u$. Therefore, $T_{j}u\to u$ in $\sobo^{\A,1}(\Omega)$ as $j\to\infty$, which is \ref{item:conv1}.

Ad~\ref{item:conv2a}. We recall the decomposition~\eqref{eq:replacementsequence}. Since $T_{j}^{(2)}u$ is a locally finite sum of $\hold^{\infty}$-maps, $T_{j}^{(2)}u\in\hold^{\infty}(\overline{\Omega};V)$. Moreover, since $T_{j}^{(1)}u\in\sobo^{\A,1}(\Omega)$ is compactly supported in $\Omega$, standard mollification yields some $v_{j}\in\hold_{c}^{\infty}(\Omega;V)$ such that $\|T_{j}^{(1)}u-v_{j}\|_{\sobo^{\A,1}(\Omega)}<\frac{1}{j}$.  Put $u_{j}:=v_{j}+T_{j}^{(2)}u$ so that $u_{j}\in\hold^{\infty}(\overline{\Omega};V)$, and then  \ref{item:conv1} implies $\|u-u_{j}\|_{\sobo^{\A,1}(\Omega)}\to 0$ as $j\to\infty$. 

To see \ref{item:conv1a}, let $\varepsilon>0$. By uniform continuity of $u$, we find $\delta>0$ such that $|x-y|<\delta$ implies $|u(x)-u(y)|<\varepsilon$. Let $\ball_{j,i}$ be such that $\ball_{j,i}\cap\Omega_{j}\neq\emptyset$. By \ref{item:C4} and assuming that $j\geq j_{0}$, there exists a constant $\mathtt{C}>0$ independent of $j$ and $i$ such that $|x-y|<\mathtt{C}2^{-j}$ holds for all $x\in\ball_{j,i}$ and $y\in\ball_{j,i}^{\sharp}$; moreover, we may assume that $\diam(\ball_{j,i}^{\sharp})<\mathtt{C}2^{-j}$ as well. We thus find $J_{0}\geq j_{0}$ such that for all $j\geq J_{0}$ there holds $\mathtt{C}2^{-j}\leq \delta$. Thus, invoking Lemma~\ref{lem:aux}\ref{item:proj1} and because of $(u)_{\ball_{j,i}^{\sharp}}=\Pi_{j,i}(u)_{\ball_{j,i}^{\sharp}}$,
\begin{align*}
\|\rho_{j}\rho_{j,i}(u-\Pi_{j,i}u)\|_{\lebe^{\infty}(\Omega)} & \leq \|u-(u)_{\ball_{j,i}^{\sharp}}\|_{\lebe^{\infty}(\ball_{j,i}\cap\Omega_{j})}+\|\Pi_{j,i}(u-(u)_{\ball_{j,i}^{\sharp}})\|_{\lebe^{\infty}(\ball_{j,i}\cap\Omega_{j})} \\ 
& \leq \sup_{x\in\ball_{j,i}\cap\Omega_{j}}\dashint_{\ball_{j,i}^{\sharp}}|u(x)-u(y)|\dif y + c\dashint_{\ball_{j,i}^{\sharp}}|u-(u)_{\ball_{j,i}^{\sharp}}|\dif y \\ 
& \leq (1+c)\varepsilon. 
\end{align*}
Since $u-T_{j}u=\sum_{i}\rho_{j}\rho_{j,i}(u-\Pi_{j,i}u)$ and the $\ball_{j,i}$'s have mutual uniformly finite overlap (in $i$ for each fixed $j$), we thus obtain $T_{j}u\to u$ in $\lebe^{\infty}(\Omega;V)$ and hence \ref{item:conv1a} follows. The proof is complete. 
\end{proof}
For the following, let $(\mathscr{X}(\partial\Omega;V);\|\cdot\|_{\mathscr{X}(\partial\Omega)})$ be a Banach space with $\mathscr{X}(\partial\Omega;V)\subset\lebe_{\locc}^{1}(\partial\Omega;V)$ such that $\sobo^{k,1}(\Omega;V)$ has trace space $\mathscr{X}(\partial\Omega;V)$ (cf. Theorem~\ref{thm:main1}~ff. for this terminology); we denote $\trace_{\partial\Omega}\colon\sobo^{k,1}(\Omega;V)\to\mathscr{X}(\partial\Omega;V)$ the underlying trace operator.
\begin{corollary}\label{cor:main}
There exists a constant $c=c(\Omega,\A)>0$ such that the following holds: Whenever $u\in\hold^{\infty}(\overline{\Omega};V)$ and $j\geq j_{0}$ (with $j_{0}\in\mathbb{Z}$ as in \ref{item:C4}), then there holds 
\begin{align}\label{eq:tracesconverge}
\sum_{j=j_{0}}^{\infty}\|\trace_{\partial\Omega}(T_{j+1}u)-\trace_{\partial\Omega}(T_{j}u)\|_{\mathscr{X}(\partial\Omega)}\leq c\|\A u\|_{\lebe^{1}(\Omega)}. 
\end{align}
\end{corollary}
\begin{proof}
Let $u\in\hold^{\infty}(\overline{\Omega};V)$ and $j\geq j_{0}$. Then \eqref{eq:tracerewrite} implies that $T_{j+1}u-T_{j}u$ is a locally finite sum of $\hold^{\infty}(\overline{\Omega};V)$-maps and so is of class $\hold^{\infty}(\overline{\Omega};V)$, too. Since $\spt(\mathrm{I}_{j}[u])\subset\{x\in\Omega\colon\;\dista(x,\partial\Omega)> 2^{-j-2}\}$ because of $\spt(\rho_{j}-\rho_{j+1})\subset\Omega_{j}\setminus\Omega_{j+2}$, we have $\trace_{\partial\Omega}(\mathrm{I}_{j}[u])=0$. 
Since $\sobo^{k,1}(\Omega;V)$ has trace space $\mathscr{X}(\partial\Omega;V)$, there exists a constant $c>0$ such that $\|\trace_{\partial\Omega}(v)\|_{\mathscr{X}(\partial\Omega)}\leq c\|v\|_{\sobo^{k,1}(\Omega)}$ for all $v\in\sobo^{k,1}(\Omega;V)$. Now, by \eqref{eq:tracerewrite} and Proposition~\ref{lem:convergenceofTj}~\ref{item:conv0}, this entails
\begin{align}\label{eq:tracesconvergeA}
\begin{split}
\sum_{j=j_{0}}^{\infty}\|\trace_{\partial\Omega}(T_{j+1}u)-\trace_{\partial\Omega}(T_{j}u)\|_{\mathscr{X}(\partial\Omega)} & = \sum_{j=j_{0}}^{\infty}\|\trace_{\partial\Omega}(\mathrm{II}_{j}[u])\|_{\mathscr{X}(\partial\Omega)} \\ 
& \leq c\sum_{j=j_{0}}^{\infty}\|\mathrm{II}_{j}[u]\|_{\sobo^{k,1}(\Omega)} \\ 
& \!\!\!\!\!\!\!\!\!\stackrel{\text{Prop.}~\ref{lem:convergenceofTj}~\ref{item:conv0}}{\leq} c\sum_{j=j_{0}}^{\infty}\|\A u\|_{\lebe^{1}(\Omega_{j-i_{0}}\setminus\Omega_{j+i_{0}})} \\ 
& \leq c\|\A u\|_{\lebe^{1}(\Omega_{j_{0}-i_{0}})}, 
\end{split}
\end{align}
where in the ultimate step we used that the sets $\Omega_{j-i_{0}}\setminus\Omega_{j+i_{0}}$ only have a uniformly finite mutual overlap. This is \eqref{eq:tracesconverge}, and the proof is complete. 
\end{proof}
Before we proceed, we briefly pause to comment on the preceding proof in  
\begin{remark}\label{rem:keyidea2ndapproach}
Based on Proposition~\ref{lem:convergenceofTj}~\ref{item:conv0}, in the second line of \eqref{eq:tracesconvergeA} we have employed \emph{full $k$-th order gradient $\lebe^{1}$-estimates} to obtain the requisite $\mathscr{X}(\partial\Omega;V)$-estimates for the traces of $T_{j+1}u-T_{j}u$. We wish to emphasize that a similar strategy \emph{does not} yield any uniform $\lebe^{1}$-bounds on the sequence $(D^{k}(T_{j+1}u-T_{j}u))_{j}$, and so this is in line with \textsc{Ornstein}'s Non-Inequality. Indeed, even for $u\in\hold^{\infty}(\overline{\Omega};V)$ the term $\mathrm{I}_{j}[u]$ in \eqref{eq:tracerewrite}, for which Proposition~\ref{lem:convergenceofTj} only provides $\sobo^{\A,1}$- but not $\sobo^{k,1}$-bounds, might have diverging $\sobo^{k,1}$-norm as $j\to\infty$: Precisely by \textsc{Ornstein}'s Non-Inequality, when $D^{k}$ acts on $u$ in the term $\mathrm{I}_{j}[u]$ in \eqref{eq:tracerewrite}, one has $\limsup_{j\to\infty}\|\mathrm{I}_{j}[u]\|_{\sobo^{k,1}(\Omega)}=\infty$ in general. By the inverse estimates on polynomials of a fixed degree, this obstruction becomes invisible for $\mathrm{II}_{j}[u]$, and only $\mathrm{II}_{j}[u]$ matters for the requisite trace estimate \eqref{eq:tracesconverge}. 
\end{remark}
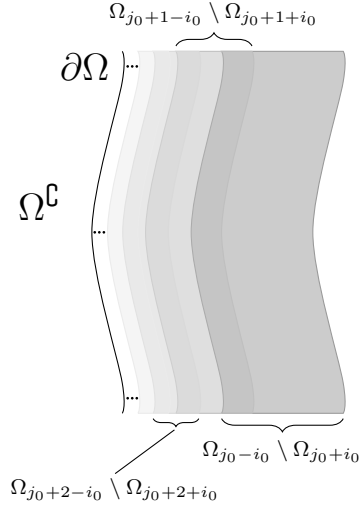
\begin{figure}
\begin{tikzpicture}[scale=2]
  \begin{scope}
    \draw  plot [smooth] coordinates { (0.047,-1.2) (0.047,-1) (-.153,0) (0.047,1) (0.047,1.2)};
       \draw [fill= black!40!white, color=black!40!white, fill opacity=0.5] plot [smooth] coordinates { (0.7,-1.2) (0.7,-1) (0.5,0) (0.7,1) (0.7,1.2)} -- (1.5,1.2) -- plot [smooth] coordinates  { (1.5,1.2) (1.5,1) (1.3,0) (1.5,-1) (1.5,-1.2)}   -- (0.7,-1.2);
       \draw [fill= black!25!white, color=black!25!white, fill opacity=0.5] plot [smooth] coordinates { (0.4,-1.2) (0.4,-1) (0.2,0) (0.4,1) (0.4,1.2)} -- (0.9,1.2) -- plot [smooth] coordinates  { (0.9,1.2) (0.9,1) (0.7,0) (0.9,-1) (0.9,-1.2)}   -- (0.4,-1.2);
        \draw [fill= black!15!white, color=black!15!white, fill opacity=0.6] plot [smooth] coordinates { (0.25,-1.2) (0.25,-1) (0.05,0) (0.25,1) (0.25,1.2)} -- (0.55,1.2) -- plot [smooth] coordinates  { (0.55,1.2) (0.55,1) (0.35,0) (0.55,-1) (0.55,-1.2)}   -- (0.25,-1.2);
         \draw [fill= black!8!white, color=black!8!white, fill opacity=0.5] plot [smooth] coordinates { (0.15,-1.2) (0.15,-1) (-0.05,0) (0.15,1) (0.15,1.2)} -- (0.35,1.2) -- plot [smooth] coordinates  { (0.35,1.2) (0.35,1) (0.15,0) (0.35,-1) (0.35,-1.2)}   -- (0.15,-1.2);
      \node at (-0.2,1.1) {\LARGE $\partial\Omega$};
      \node at (-0.5,0.2) {\LARGE $\Omega^{\complement}$};
      \draw [decorate,decoration={brace,amplitude=5pt,mirror,raise=2pt},yshift=0pt]
(0.7,-1.2) -- (1.5,-1.2) node [black,midway,xshift=0cm,yshift = -0.5cm] {\footnotesize $\Omega_{j_{0}-i_{0}}\setminus\Omega_{j_{0}+i_{0}}$};
\draw [decorate,decoration={brace,amplitude=5pt,raise=2pt},yshift=0pt]
(0.4,1.2) -- (0.9,1.2) node [black,midway,xshift=0cm,yshift = 0.5cm] {\footnotesize $\Omega_{j_{0}+1-i_{0}}\setminus\Omega_{j_{0}+1+i_{0}}$};
  \draw [decorate,decoration={brace,amplitude=3pt,mirror,raise=2pt},yshift=0pt]
(0.25,-1.2) -- (0.55,-1.2) node [black,midway,xshift=0cm,yshift = -0.5cm] {};
\node at (0,-1.7) {\footnotesize $\Omega_{j_{0}+2-i_{0}}\setminus\Omega_{j_{0}+2+i_{0}}$};
\draw [-] (0,-1.6) -- (0.39,-1.31);
\node at (0.0825,1.1) [circle,fill,inner sep=0.1pt]{};
\node at (0.1125,1.1) [circle,fill,inner sep=0.1pt]{};
\node at (0.1425,1.1) [circle,fill,inner sep=0.1pt]{};
\node at (-0.074,0) [circle,fill,inner sep=0.1pt]{};
\node at (-0.104,0) [circle,fill,inner sep=0.1pt]{};
\node at (-0.134,0) [circle,fill,inner sep=0.1pt]{};
\node at (0.0825,-1.1) [circle,fill,inner sep=0.1pt]{};
\node at (0.1125,-1.1) [circle,fill,inner sep=0.1pt]{};
\node at (0.1425,-1.1) [circle,fill,inner sep=0.1pt]{};
      \end{scope}
      \end{tikzpicture}
      \caption{Idea of the trace estimate of Theorem~\ref{thm:main1}. In the telescope sum \eqref{eq:telescope}, the trace of $u\in\hold^{\infty}(\overline{\Omega};V)$ is decomposed into a contribution from $T_{j_{0}}u$ and a sum over $\trace_{\partial\Omega}(T_{j+1}u-T_{j}u)=\trace_{\partial\Omega}(\mathrm{II}_{j}[u])$. By construction, $\mathrm{II}_{j}[u]$ is determined by the projections of $u$ onto $\ker(\A)$ on reflected balls contained in the strips $\Omega_{j-i_{0}}\setminus\Omega_{j+i_{0}}$; the latter have uniformly finite mutual overlap and yield the entire $\Omega$ when united.}
\end{figure}
Towards the proof of Theorem~\ref{thm:main1}, let us note that it is only the sufficiency part that requires proof; similarly as for Theorem~\ref{thm:main}, the necessity part can be established by localising the construction in \cite[Proof of Theorem~5.2]{GRVS}. Indeed, if $\A$ is not $\mathbb{C}$-elliptic, the aforementioned construction leads to some smoothly bounded $\Omega\subset\R^{n}$ and a $\sobo^{\A,1}$-bounded sequence $(u_{j})\subset\hold^{\infty}(\overline{\Omega};V)$ for which $(D^{k-1}u_{j})$ is unbounded in $\lebe^{1}(\partial\Omega;\odot^{k-1}(\R^{n};V))$. Since $\sobo^{k,1}(\Omega;V)$ has trace space $\besov_{1,1}^{k-1}(\Omega;V)$, $\sobo^{\A,1}(\Omega)$ cannot have the same trace space as $\sobo^{k,1}(\Omega)$. It thus suffices to give the
\begin{proof}[Proof of the sufficiency part of Theorem~\ref{thm:main1}]
Let $\A$ be a $\mathbb{C}$-elliptic operator and let $u\in\hold^{\infty}(\overline{\Omega};V)$ be given. Since $u\in\hold^{\infty}(\overline{\Omega};V)$, Lemma~\ref{lem:convergenceofTj}~\ref{item:conv1a} implies by a telescope sum argument
\begin{align}\label{eq:telescope}
\trace_{\partial\Omega}(u)-\trace_{\partial\Omega}(T_{j_{0}}u)=\sum_{j=j_{0}}^{\infty}\trace_{\partial\Omega}(T_{j+1}u)-\trace_{\partial\Omega}(T_{j}u)
\end{align}
everywhere on $\partial\Omega$. Based on the decomposition \eqref{eq:replacementsequence}, we recall that for $u\in\hold^{\infty}(\overline{\Omega};V)$ there holds $\|\trace_{\partial\Omega}(T_{j_{0}}^{(1)}u)\|_{\mathscr{X}(\partial\Omega)}=0$, and so we proceed  by proving that the function $T_{j_{0}}^{(2)}u$ belongs to $\sobo^{k,1}(\Omega;V)$ together with the requisite estimates. Since $T_{j_{0}}^{(2)}u\in\hold^{\infty}(\overline{\Omega};V)$, it possesses classical traces and so it suffices to give an estimate of the $\sobo^{k,1}$-norm of $T_{j_{0}}^{(2)}u$. Denoting $\mathcal{L}_{j_{0}}$ the set of indices $i\in\mathbb{N}$ such that $\Omega_{j_{0}}\cap\ball_{j_{0},i}\neq\emptyset$,
\begin{align}\label{eq:Tj0PreEst}
\begin{split}
\|T_{j_{0}}^{(2)}u\|_{\sobo^{k,1}(\Omega)} & \leq c\sum_{\ell=0}^{k}\sum_{\substack{|\alpha|+|\beta|=\ell}}\sum_{i\in\mathcal{L}_{j_{0}}}\|(\partial^{\alpha}(\rho_{j_{0}}\rho_{j_{0},i}))(\partial^{\beta}\Pi_{j_{0},i}u)\|_{\lebe^{1}(\Omega)}\\
& \!\!\!\!\!\!\!\stackrel{\eqref{eq:gradbound0},\,\eqref{eq:gradbound}}{\leq}  c\sum_{\ell=0}^{k}\sum_{\substack{|\alpha|+|\beta|=\ell}}\sum_{i\in\mathcal{L}_{j_{0}}}2^{j_{0}|\alpha|} \times \\ & \;\;\;\;\;\;\;\;\;\;\;\;\;\;\;\;\;\;\;\;\;\;\;\;\times 2^{-j_{0}(n-|\beta|)}2^{-j_{0}|\beta|}\dashint_{\ball_{j_{0},i}}|\partial^{\beta}\Pi_{j_{0},i}u|\dif x \\ 
& \!\!\!\!\!\!\!\!\!\!\!\!\!\!\!\!\stackrel{\eqref{eq:inverseFull}_{2},\,\text{Lem.}~\ref{lem:aux}\ref{item:proj2}}{\leq} c\sum_{\ell=0}^{k}2^{j_{0}\ell}2^{-j_{0}n}\sum_{i\in\mathcal{L}_{j_{0}}}\dashint_{\ball_{j_{0},i}^{\sharp}}|u|\dif x\\
& \!\!\!\stackrel{\eqref{eq:sharpsum}_{1}}{\leq} c2^{j_{0}k}\|u\|_{\lebe^{1}(\Omega_{j_{0}-i_{0}}\setminus\Omega_{j_{0}+i_{0}})},
\end{split}
\end{align}
where we additionally used the uniform finite overlap property of the $\ball_{j,i}$'s. Since, by assumption, $\sobo^{k,1}(\Omega;V)$ has trace space $\mathscr{X}(\partial\Omega;V)$, we consequently obtain 
\begin{align}\label{eq:Tj0estimate}
\begin{split}
\|\trace_{\partial\Omega}(T_{j_{0}}u)\|_{\mathscr{X}(\partial\Omega)} & = \|\trace_{\partial\Omega}(T_{j_{0}}^{(2)}u)\|_{\mathscr{X}(\partial\Omega)}  \leq c2^{j_{0}k}\|u\|_{\lebe^{1}(\Omega)}.
\end{split}
\end{align}
Combining \eqref{eq:telescope}, \eqref{eq:Tj0estimate} and \eqref{eq:tracesconverge}, we then obtain
\begin{align}\label{eq:traceWin1}
\|\trace_{\partial\Omega}(u)\|_{\mathscr{X}(\partial\Omega)} & \leq c2^{j_{0}k}\|u\|_{\lebe^{1}(\Omega)} + c\|\A u\|_{\lebe^{1}(\Omega)}\qquad\text{for all}\;u\in\hold^{\infty}(\overline{\Omega};V). 
\end{align}
If $u\in\sobo^{\A,1}(\Omega)$, we invoke Lemma~\ref{lem:convergenceofTj}~\ref{item:conv2a} to find $(u_{j})\subset\hold^{\infty}(\overline{\Omega};V)$ such that $u_{j}\to u$ in $\sobo^{\A,1}(\Omega)$. As in the proof of Lemma~\ref{lem:GaussGreen} and since $(\mathscr{X}(\partial\Omega),\|\cdot\|_{\mathscr{X}(\partial\Omega)})$ is assumed Banach, we may define as a limit in $\mathscr{X}(\partial\Omega;V)$
\begin{align}\label{eq:traceWin2}
\widetilde{\trace}_{\partial\Omega}(u):=\lim_{j\to\infty}\trace_{\partial\Omega}(u_{j}), 
\end{align}
being independent of the approximating sequence $(u_{j})\subset\hold^{\infty}(\overline{\Omega};V)$. It is then straightforward to verify that $\widetilde{\trace}_{\partial\Omega}\colon\sobo^{\A,1}(\Omega)\to\mathscr{X}(\partial\Omega;V)$ is bounded and linear. Clearly $\widetilde{\trace}_{\partial\Omega}$ and $\trace_{\partial\Omega}$ coincide on $\sobo^{k,1}(\Omega;V)$, and since $\trace_{\partial\Omega}\colon\sobo^{k,1}(\Omega;V)\to\mathscr{X}(\partial\Omega;V)$ is surjective, $\widetilde{\trace}_{\partial\Omega}\colon\sobo^{\A,1}(\Omega)\to\mathscr{X}(\partial\Omega;V)$ is onto, too. The proof is complete.
\end{proof}

 We conclude this section with the following
\begin{remark}[Homogeneous trace inequalities]
To connect with the results of Section~\ref{sec:traces-estimates-via}, let $\Omega=\mathbb{H}$ be an open halfspace and recall that ${\dot\sobo}{^{k,1}}(\mathbb{H};V)$ has trace space ${\dot\besov}{_{1,1}^{k-1}}(\partial\mathbb{H};V)$. Given $u\in\hold_{c}^{\infty}(\overline{\bbH};V)$, we define $T_{j}u$ as in \eqref{eq:replacementsequence}. Note that in this particular geometric situation, the number as it appears in \ref{item:C4} can be chosen to be arbitrarily negative. Proposition~\ref{lem:convergenceofTj} directly inherits to the homogeneous situation and hence one obtains \eqref{eq:tracesconverge} with $j_{0}=-\infty$. As a substitute of \eqref{eq:Tj0PreEst} we then obtain 
\begin{align}\label{eq:Tj0PreEst1}
\|T_{j_{0}}^{(2)}u\|_{{\dot\sobo}{^{k,1}}(\Omega)} \leq c2^{j_{0}k}\|u\|_{\lebe^{1}(\Omega_{j_{0}-i_{0}}\setminus\Omega_{j_{0}+i_{0}})}.
\end{align} 
Since $u\in\hold_{c}^{\infty}(\overline{\mathbb{H}};V)$, the term on the right-hand side of \eqref{eq:Tj0PreEst1} will vanish as $j_{0}\to-\infty$.
Based on the \textsc{Uspenski\u{\i}} estimate \eqref{eq:Uspenskiimain}, an analogous argument as that for \eqref{eq:traceWin1} then equally yields $
\|u\|_{{\dot\besov}{_{1,1}^{k-1}}(\partial\mathbb{H})}\leq c\|\A u\|_{\lebe^{1}(\bbH)}$ for all  $u\in\hold_{c}^{\infty}(\overline{\bbH};V)$. 
\end{remark}

\section{Proof of Theorem~\ref{thm:Korn}: Korn without global singular integral estimates}\label{sec:Korn}
We finally turn to the Korn-type inequality of Theorem~\ref{thm:Korn} for John domains, which form the canonical class of domains for such inequalities. Since such domains in general need not be extension domains for $\sobo^{\A,p}$, the global Calder\'{o}n-Zygmund estimates from the proof of Theorem~\ref{thm:mainLp} cannot be utilised and a local argument is required. In Section~\ref{sec:decomp}, which should be of independent interest, we provide a generalisation of a decomposition for maps on John domains by  \textsc{\Ruzicka{}, Schumacher} and the first author \cite{DieRuzSch10}, the latter being only applicable to maps with zero mean. Building on a best approximation result and an inequality of \textsc{Fefferman-Stein}-type (Sections~\ref{sec:bestapprox} and \ref{sec:feffermanstein}),  we obtain even more general Korn-type inequalities by means of extrapolation (so e.g. involving Orlicz- or Lorentz norms) once Korn-type inequalities are provided for weighted Lebesgue spaces, see Sections~\ref{sec:Kornmain} and \ref{sec:Kornextend}. Hence, even though Theorem~\ref{thm:Korn} is only stated for the usual $\lebe^{p}$-spaces, we directly deal with the weighted setting throughout.

\subsection{Weights and a decomposition theorem for John domains}\label{sec:decomp}
As usual, a weight $w\,:\, \Rn \to \setR$ with $w > 0$ almost everywhere is said to be a $A_q$-\emph{Muckenhoupt weight} with $1 \leq q < \infty$, in short $w \in A_q$, if
\begin{align}\label{eq:weights}
\begin{split}
 & [w]_{A_q} := \sup_Q \dashint_Q w\dif x \bigg(\dashint_Q w^{-\frac{1}{q-1}}\dif x\bigg)^{q-1} < \infty,\qquad\text{if}\;1<q<\infty,\\ 
 & [w]_{A_1} := \sup_{Q} \Bigg[ \dashint_Q w\,\dif x\; \sup_{x\in Q} \frac{1}{w(x)}
  \bigg] < \infty,\qquad\qquad\;\;\,\text{if}\;q=1.
\end{split}
\end{align}
The class $A_\infty$ is defined by $A_\infty=\bigcup_{q>1} A_q$ and endowed with $[w]_{A_{\infty}}:=\lim_{p\to\infty}[w]_{A_{p}}$. Now let $\Omega\subset\R^{n}$ be open and bounded. For $w\in A_{\infty}$, $\lebe_{w}^{q}(\Omega;V)$ is defined in the obvious manner. For a finite dimensional real vector space $E$, a subspace $\mathcal{N}\subset\mathscr{P}_{m}(\R^{n};E)$ and a space $X(\Omega;E)\subset\lebe^{1}(\Omega;E)$, we define 
\begin{align}\label{eq:orthspace}
X_{\mathcal{N}}(\Omega;E):=\left\{f\in X(\Omega;E)\colon\;\int_{\Omega}f\cdot\pi\dif x = 0\;\;\;\text{for all}\;\pi\in\mathcal{N} \right\}.
\end{align}
If $w\in A_{q}$ with $1\leq q < \infty$ and $f\in \lebe_{w}^{q}(\Omega;E)$, then H\"{o}lder's inequality and \eqref{eq:weights} imply that $f\in\lebe^{1}(\Omega;E)$. In particular, for such weights $w$, $\lebe_{w,\mathcal{N}}^{q}(\Omega;E)$ is well-defined. We can now state the decomposition theorem that shall later be applied to~$\mathcal{N} = \nabla^{k} \ker(\bbA)$; as usual, $\mathcal{M}$ denotes the Hardy-Littlewood maximal operator. 
\begin{theorem}[Decomposition theorem]
  \label{thm:decomposition}
  Let $\Omega\subset \setR^n$ be an open and bounded John domain or an open and bounded
  domain satisfying the emanating chain condition with constants
  $\sigma_1, \sigma_2$ and chain-covering
  $\mathcal{W} = \set{W_i \,:\, i \in \setN_0}$.  Then there exists a
  family of linear operators
  $T_{i}\colon \hold^\infty_{c,\mathcal{N}}(\Omega;E) \to \hold^\infty_{c,\mathcal{N}}(W_i ;E)$,
  $i \in \setN_0$, such that for all $1<q <\infty$ and all $w \in A_q$
  the following holds:
  \begin{enumerate}
   \item \label{itm:mainMf} For each $i \in \setN_0$ and all $f \in
    \lebe^q_{w,\mathcal{N}}(\Omega ;E)$ there holds
    \begin{align} \label{eq:20}
    \abs{T_i f} &\leq c(\sigma_2,\mathcal{N})\,\mathbbm{1}_{W_i}\, \mathcal{M}(\mathbbm{1}_{\Omega}f) \quad \text{$\mathscr{L}^{n}$-almost
        everywhere}.
    \end{align}
  \item \label{itm:Lqw0Wi} For each $i \in \setN_0$ the operator $T_i$
    maps $\lebe^q_{w,\mathcal{N}}(\Omega;E)$ boundedly to $\lebe^q_{w,\mathcal{N}}(W_i ;E)$.
     \item\label{itm:mainsum} The family $\{T_{i} f\colon\;i\in\mathbb{N}_{0}\}$ is a decomposition
    of~$f$ in $\lebe^q_{w,\mathcal{N}}(\Omega;E)$, i.e.,
    \begin{align}
      \label{eq:mainsum}
      f = \sum_{i \in\mathbb{N}_{0}} T_i f \qquad \text{in }\,\lebe^q_{w,\mathcal{N}}(\Omega;E)
    \end{align}
    for all $f \in \lebe^q_{w,\mathcal{N}}(\Omega;E)$. The convergence is {\em
      unconditionally}.
  \item \label{itm:mainlqLq} The mapping $f \mapsto (\norm{T_i
      f}_{\lebe^q_{w}(W_i)})_i$ from $\lebe^q_{w,\mathcal{N}}(\Omega;E)$ to
    $l^q(\setN_0)$ is bounded and 
    \begin{align}
      \label{eq:19}
      \frac{1}{c} \norm{f }_{\lebe^q_w(\Omega)}
      &\leq
        \bigg( \sum_{i\in\mathbb{N}_{0}} \norm{T_i f}_{\lebe^q_w(W_i)}^q
        \bigg)^{\frac{1}{q}} \leq c\,\norm{f}_{\lebe^q_w(\Omega)}
    \end{align}
    with a constant $c=c(\sigma_1, \sigma_2, q, [w]_{A_{q}},\mathcal{N})>0$.
  \item \label{itm:Cinfty} If $f \in
    \hold^\infty_{c,\mathcal{N}}(\Omega;E)$, then $\set{i\in\mathbb{N}_{0}\,:\, T_i f \not= 0}$
    is finite.
  \end{enumerate}
\end{theorem}
\begin{proof}
  The proof is an adaptation of the proof of~\cite[Thm.~4.2]{DieRuzSch10} by
\textsc{\Ruzicka, Schumacher} and the first author, and so we 
confine ourselves to the points that need refinement in the setting considered here. 

 The idea in~\cite{DieRuzSch10}, where $\mathcal{N}$ is given by the constants, is that
  each~$f$ is first decomposed into local functions~$S_i f :=\xi_i f$
  by means of a partition of unity~$\xi_i$ subject to the
  covering~$\mathcal{W}=\{W_{i}\colon\;i\in\mathbb{N}_{0}\}$. Then each $S_i f$ is corrected
  in~\cite{DieRuzSch10} such that it has vanishing integral. To
  compensate this change, the corrections are transported along the
  chains to the central ball~$W_0$. In the much more general  setting as considered here we have to correct
  the~$S_i f$'s by local projections to~$\mathcal{N}$ and transport
  those corrections along the chains to the central ball~$W_0$. Throughout, we tacitly suppose $f$ to be extended to $\R^{n}$ by zero. 

Adopting the notation of Definition~\ref{def:boman_chain}, we may suppose that the balls $\ball_{i,l}$ as in \ref{itm:C2} belong to a family $\mathscr{B}$ such that $\sum_{B\in\mathscr{B}}\mathbbm{1}_{\ball}\leq \sigma_{2}\mathbbm{1}_{\Omega}$; see Remark~\ref{rem:chooseBALLS}. For $B\in\mathscr{B}$ we find, by translating and scaling a fixed non-negative $\eta\in\hold_{c}^{\infty}(\ball(0,1))$ with $\int_{\ball(0,1)}\eta\dif x=1$, a  non-negative function $\eta_{\ball}\in\hold_{c}^{\infty}(\ball)$ with $\int_{\ball}\eta_{\ball}\dif x = 1$ and $|\eta_{\ball}|\leq c/\mathscr{L}^{n}(\ball)$, where $c=c(n)>0$. Setting $\eta_{i,l}:=\eta_{\ball_{i,l}}$, we may record that
  \begin{align}
    \label{eq:eta}
    \begin{aligned}
      \eta_{i,l} &\in \hold^\infty_c(W_{i,l}) \cap
      \hold^\infty_c(W_{i,l+1}),
      \\
      \norm{\eta_{i,l}}_{\lebe^{\infty}} &\leq c\, \min \biggset{
        \frac{1}{\mathscr{L}^{n}(W_{i,l})}, \frac{1}{\mathscr{L}^{n}(W_{i,l+1})}}
    \end{aligned}
  \end{align}
  for all $i,l \geq 0$ with $0 \leq l \leq m_i -1$. Equally,  we pick a non-negative $\eta_{0}\in\hold_{c}^{\infty}(W_{0})$ with $\int_{W_{0}}\eta_{0}\dif x = 1$ and $|\eta_{0}|\leq c/\mathscr{L}^{n}(W_{0})$; since every chain ends in $W_{i,m_{i}}=W_{0}$, we put $\eta_{i,m_{i}}:=\eta_{0}$ for all $i\geq 0$. 
    
Now define an operator
  $\Pi_{j,l}\,:\, \lebe^1(\sigma_2 W_{j,l};E) \to \mathcal{N}$ by 
  \begin{align}\label{eq:projDEFKorn}
    \int_{W_{j,l}} \eta_{j,l} (\Pi_{j,l} f)\, \pi\dif x
    &= \int_{\sigma_2 W_{j,l}} f\, \pi\dif x \qquad \text{for all $\pi
      \in \mathcal{N}$.} 
  \end{align}
Inserting $\pi = \Pi_{j,l} f$, inverse estimates for polynomials and using that all $\eta_{j,l}$'s are translated and scaled versions of some $\eta$, we obtain
\begin{align*}
\|\Pi_{j,l}f\|_{\lebe^{\infty}(\sigma_{2}W_{j,l})} &  \leq c \Big(\int_{\sigma_{2}W_{j,l}}\eta_{j,l}|\Pi_{j,l}f|^{2}\dif x\Big)^{\frac{1}{2}} \\ & \!\!\!\!\!\!\!\! \leq c \Big(\int_{\sigma_{2}W_{j,l}}f\cdot\Pi_{j,l}f\dif x\Big)^{\frac{1}{2}} \leq c\|\Pi_{j,l}f\|_{\lebe^{\infty}(\sigma_{2}W_{j,l})}^{\frac{1}{2}}\Big(\int_{\sigma_{2}W_{j,l}}|f|\dif x \Big)^{\frac{1}{2}}, 
\end{align*}
which implies 
  \begin{align}
    \label{eq:est-Pjk}
    \|\Pi_{j,l}f\|_{\lebe^{\infty}(\sigma_2 W_{j,l})} 
    & \leq c
      \int_{\sigma_2 W_{j,l}} \abs{f}\dif x.
  \end{align}
  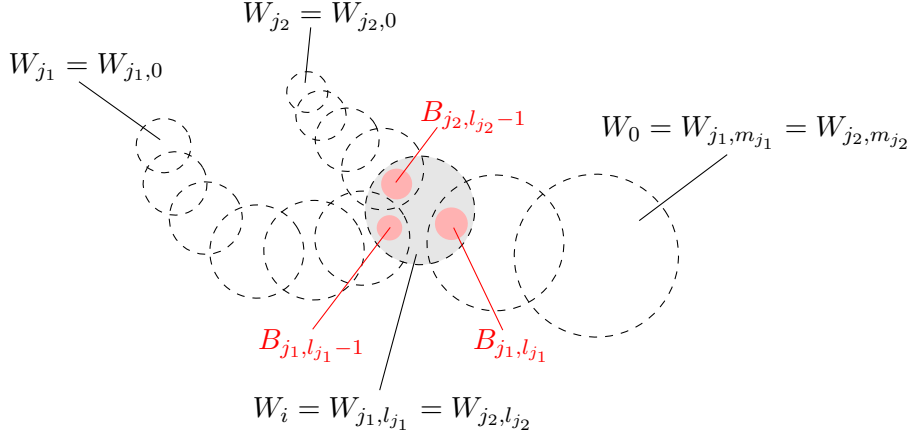
\begin{figure}
\begin{tikzpicture}[rotate=30, scale=0.12] 
\draw[dashed, fill=black!10!white, opacity=1] (48.75,58) circle [radius=6]; %thirdball
\fill [red!30!white, opacity=1] (44.9,58) circle [radius=1.4];
\fill [red!30!white, opacity=1] (48,61.8) circle [radius=1.7];
\fill [red!30!white, opacity=1] (51,55) circle [radius=1.8];
\draw[dashed] (63,44) circle [radius=9]; %firstball
\draw[dashed] (54.125,50.75) circle [radius=7.5];
\draw[dashed] (47.5,64) circle [radius=4.5]; %fourthball
\draw[dashed] (45.75,68.75) circle [radius=3.5]; %fifthball 
\draw[dashed] (44.6,72.5) circle [radius=2.75]; %sixthball$
\draw[dashed] (44.6,75.5) circle [radius=2.25];
%\fill[fill=black!10!red, fill opacity=0.25] (45.1,78) circle [radius=2]; 
%\fill[fill=black!10!red, fill opacity=0.25] (45.8,80) circle [radius=1.75]; 
%\fill[fill=black!10!red, fill opacity=0.25] (43.7,86.5) circle [radius=0.6];
%\fill[fill=black!10!red, fill opacity=0.25] (43.1,87) circle [radius=0.5];
%\fill[fill=black!10!red, fill opacity=0.25] (41.1,49.25) circle [radius=7]; %secondballshifted
\draw[dashed] (36.5,60) circle [radius=5.5];
\draw[dashed] (41.75,58.5) circle [radius=5.2]; %thirdballshifted
\draw[dashed] (31,63) circle [radius=5.15]; %fourthballshifted
\draw[dashed] (28,69) circle [radius=4];
\draw[dashed] (27,74) circle [radius=3.5]; 
\draw[dashed] (28,78.25) circle [radius=3];
\node[black] at (85,47) {\large $W_{0}=W_{j_{1},m_{j_{1}}}=W_{j_{2},m_{j_{2}}}$};
\draw (81,46.5) -- (69,45);
\node[black] at (25,90) {\large $W_{j_{1}}=W_{j_{1},0}$};
\draw (23.5,89) -- (28,79);
\node[black] at (50,82) {\large $W_{j_{2}}=W_{j_{2},0}$};
\draw (48.5,81.5) -- (44.9,76.5);
\node[red] at (59,64) {\large $\ball_{j_{2},l_{j_{2}}-1}$};
\draw[red] (55,65) -- (48,62);
\draw[red] (32,52.5) -- (45,58);
\node[red] at (31,51) {\large $\ball_{j_{1},l_{j_{1}}-1}$};
\node[red] at (50,40) {\large $\ball_{j_{1},l_{j_{1}}}$};
\draw[red] (50,42) -- (51,54.5);
\node[black] at (35,40) {\large $W_{i}=W_{j_{1},l_{j_{1}}}=W_{j_{2},l_{j_{2}}}$};
\draw (35,43) -- (46,54);
\end{tikzpicture}
\caption{Construction of the decomposition blocks $T_{i}$, cf.~\eqref{eq:defTi}. In \eqref{eq:defTi}, one sums over all chains $W_{j,0}=W_{j},...,W_{j,m_{j}}=W_{0}$ that contain $W_{i}$; here, two such chains and the location of the relevant balls that matter in the definition of $T_{i}$ are depicted. }
\end{figure}
We now modify the decomposition operators
  of~\cite{DieRuzSch10} and, adopting the notation of Definition~\ref{def:boman_chain}, set for $f\in\lebe_{w,\mathcal{N}}^{q}(\Omega;E)$
  \begin{align}
    \label{eq:defTi}
    \begin{split}
      T_i f &:= (S_i f - \eta_{i,0} \Pi_{i,0}(S_i f))
      \\
      &\qquad + \sum_{\substack{j \geq 0\\j \not= i}}
      \sum_{\substack{l:0< l \leq m_j:\\W_{j,l} = W_i}} \big(
      \eta_{j,l-1} \Pi_{j,l-1}(S_j f) - \eta_{j,l} \Pi_{j,l}(S_j
      f)\big).
    \end{split}
  \end{align}
In the ultimate double sum, $j$ effectively ranges over all numbers for which the chain $W_{j,m_{j}}=W_{0},...,W_{j,0}=W_{j}$ from \ref{itm:C2} connects $W_{j}$ with $W_{0}$ where, in addition, $W_{j,l}=W_{i}$ for some $0<l\leq m_{j}$ asserts that $W_{i}$ must be contained in this chain. If $j\neq i$ and $0<l\leq m_{j}$ are such that $W_{j,l}=W_{i}$, then \eqref{eq:eta} implies that $\spt(\eta_{j,l-1})\cup\spt(\eta_{j,l})\subset W_{j,l}=W_{i}$. We hence obtain $\spt(T_{i}f)\subset W_{i}$ and, by \eqref{eq:eta}, $\eta_{j,l-1}+\eta_{j,l}\leq c\mathbbm{1}_{W_{i}}/\mathscr{L}^{n}(W_{i})$. Also note that the second sum over $0<l\leq m_{j}$ such that $W_{j,l}=W_{i}$ contains only one summand. 
  
We begin by proving that $T_{i}f$ is well-defined. Similarly to~\cite{DieRuzSch10}, we define the majorant
  \begin{align}\label{eq:majorantUi}
    U_i &:= 
          \sum_{\substack{j \geq 0\\j \not= i}} 
    \sum_{\substack{l:0< l \leq m_j:\\W_{j,l} = W_i}}
    \big(
    \eta_{j,l-1} \bigabs{\Pi_{j,l-1}(S_j f)} + \eta_{j,l} \bigabs{\Pi_{j,l}(S_j f)}\big).
  \end{align}
Next, if $W_{j,0}=W_{j},...,W_{j,m_{j}}=W_{0}$ is a chain with $W_{j,l}=W_{i}$ for some $0<l\leq m_{j}$, then we have by \eqref{eq:eta}, $\ball_{j,l-1}\subset\sigma_{2}W_{j,l-1}\subset\sigma_{2}^{2}W_{j,l}$, $\ball_{j,l}\subset W_{j,l}\subset\sigma_{2}W_{j,l}$ and the uniform comparability $\mathscr{L}^{n}(\ball_{j,l-1})\simeq\mathscr{L}^{n}(\ball_{j,l})\simeq\mathscr{L}^{n}(W_{j,l})$ the pointwise estimate 
\begin{align*}
U_{i} & \leq c\frac{\mathbbm{1}_{W_{i}}}{\mathscr{L}^{n}(W_{i})}\sum_{\substack{j \geq 0\\j \not= i}} 
    \sum_{\substack{l:0< l \leq m_j:\\W_{j,l} = W_i}}\big(\|\Pi_{j,l-1}(S_{j}f)\|_{\lebe^{\infty}(\ball_{j,l-1})}+\|\Pi_{j,l}(S_{j}f)\|_{\lebe^{\infty}(\ball_{j,l})}\big) \\ 
    & \!\stackrel{\eqref{eq:est-Pjk}}{\leq} c\frac{\mathbbm{1}_{W_{i}}}{\mathscr{L}^{n}(W_{i})}\sum_{\substack{j \geq 0\\j \not= i}} 
    \sum_{\substack{l:0< l \leq m_j:\\W_{j,l} = W_i}}\|S_{j}f\|_{\lebe^{1}(\sigma_{2}^{2}W_{j,l})} \\ 
    & \leq  c\frac{\mathbbm{1}_{W_{i}}}{\mathscr{L}^{n}(W_{i})}\sum_{\substack{j \geq 0\\j \not= i}} 
    \|S_{j}f\|_{\lebe^{1}(\sigma_{2}W_{i})} \\ 
    & \leq  c\mathbbm{1}_{W_{i}}\dashint_{\sigma_{2}W_{i}}|f|\dif x 
\end{align*}
where we used in the ultimate step that the $W_{j}$'s have uniformly finite overlap, cf.~\ref{itm:C1} and, by \ref{itm:C2}, there holds $\spt(S_{j}f)\subset W_{j}\subset \sigma_{2}W_{i}$. Since $f\in\lebe_{\locc}^{1}(\Omega;E)$, the sum in $T_{i}f$ converges absolutely $\mathscr{L}^{n}$ and so $T_{i}f$ is well-defined $\mathscr{L}^{n}$-a.e.. The term $S_{i}f-\eta_{i,0}\Pi_{i,0}(S_{i}f)$ can be estimated analogously, and so \ref{itm:mainMf} follows.

Since $\mathcal{M}$ is bounded on~$\lebe^q_w(\R^{n})$, we obtain that $T_{i}$ maps $\lebe_{w}^{q}(\Omega;E)\to\lebe_{w}^{q}(W_{i};E)$ boundedly. In combination with \ref{itm:mainMf} and recalling that $\lebe_{w}^{q}(W_{i};E)\hookrightarrow\lebe^{1}(W_{i};E)$, we obtain by Lebesgue's theorem on dominated convergence that the sum defining $T_{i}f$ also converges in $\lebe^{1}(W_{i};E)$. For the full statement of \ref{itm:Lqw0Wi}, let $f\in\lebe_{w,\mathcal{N}}^{q}(\Omega;E)$ and consider some $j\neq i$ such that $W_{i}$ is contained in the chain connecting $W_{j}$ and $W_{0}$, so $W_{i}=W_{j,l}$ for some $0<l\leq m_{j}$. By \ref{itm:C2}, we then have $W_{j}\subset\sigma_{2}W_{j,l-1}\cap\sigma_{2}W_{j,l}$. Therefore, $\mathrm{spt}(S_{j}f)\subset W_{j}$ yields for all $\pi\in\mathcal{N}$ 
\begin{align}\label{eq:cancelling}
\int_{\sigma_{2}W_{j,l-1}}
      S_j f\cdot\pi\dif x = \int_{\sigma_{2}W_{j,l}} S_j
      f\cdot\pi\dif x. 
\end{align}  
The support properties of the $S_{j}f$'s and the $\eta_{j,l}$'s, the definition of $\Pi_{j,l}$ and \eqref{eq:cancelling} imply that each summand in the definition of $T_{i}$ has integral zero when being tested against any $\pi\in\mathcal{N}$. Since the sum  defining $T_{i}f$ converges in $\lebe^{1}(\Omega;E)$, 
 we conclude \ref{itm:Lqw0Wi}. We now establish that $f=\sum_{i\in\mathbb{N}_{0}}T_{i}f$ holds pointwisely $\mathscr{L}^{n}$-a.e. for $f\in\lebe_{w,\mathcal{N}}^{q}(\Omega;E)$; working from here, the proof evolves exactly as in \cite{DieRuzSch10}. By pointwise absolute convergence of $\sum_{i\in\mathbb{N}_{0}}T_{i}f$ $\mathscr{L}^{n}$-a.e., we may change the order of summation in the following to find  
\begin{align*}
\sum_{i\in\mathbb{N}_{0}} T_i f & = \sum_{i\in\mathbb{N}_{0}} S_i f - \sum_{i\in\mathbb{N}_{0}}\eta_{i,0} \Pi_{i,0}(S_i f)
      \\
      & + \sum_{j \geq 0}
      \sum_{l:0< l \leq m_j} \sum_{\substack{i\in\mathbb{N}_{0}\setminus\{j\} \\ W_{i}=W_{j,l}} }\big(
      \eta_{j,l-1} \Pi_{j,l-1}(S_j f) - \eta_{j,l} \Pi_{j,l}(S_j f)\big)\\
      & \!\stackrel{(*)}{=} \sum_{i\in\mathbb{N}_{0}} S_i f - \sum_{i\in\mathbb{N}_{0}}\eta_{i,0} \Pi_{i,0}(S_i f)
   + \sum_{j \geq 0}
      \big(\eta_{j,0} \Pi_{j,0}(S_j f) - \eta_{j,m_{j}} \Pi_{j,m_{j}}(S_j f) \big)\\ 
      & =  \sum_{i\in\mathbb{N}_{0}} S_i f - \sum_{j \geq 0}
      \eta_{j,m_{j}} \Pi_{j,m_{j}}(S_j f) \\ 
      & \!\!\!\!\!\!\!\!\stackrel{\eta_{j,m_{j}}=\eta_{0}}{=}  \sum_{i\in\mathbb{N}_{0}} S_i f - \eta_{0}\sum_{j \geq 0}\Pi_{j,m_{j}}(S_j f) \\ 
      & = f - \eta_{0}\sum_{j \geq 0}\Pi_{W_{0}}(S_j f) \\
      & = f - \eta_{0}\Pi_{W_{0}}f = f, 
\end{align*}
since $\Pi_{W_{0}}:=\Pi_{j,m_{j}}$ is independent of $j$ and we have $\Pi_{W_{0}}f=0$ as $f\in\lebe_{w,\mathcal{N}}^{q}(\Omega;E)$. For $(*)$ note that, for fixed $j\geq 0$ and $0<l\leq m_{j}$, $\sum_{i\neq j,\,W_{i}=W_{j,l}}1=1$, cf.~\cite[Eq.~(4.18)ff.]{DieRuzSch10}. Working from here, 
 the rest of the proof is exactly as in~\cite{DieRuzSch10}. We thus omit the details.
\end{proof}
\begin{remark}[Unbounded domains] 
  Theorem~\ref{thm:decomposition} also holds for unbounded John
  domains and unbounded domains satisfying the emanating chain
  condition. The proof requires no change. We refer to the modified
  definitions for unbounded domains to~\cite{DieRuzSch10}.
\end{remark}
\subsection{A best approximation property}\label{sec:bestapprox}
For our objectives in Section~\ref{sec:feffermanstein} and \ref{sec:Kornmain}, we require the following best approximation property of the projections underlying Theorem~\ref{thm:poincare-john}:
\begin{proposition}[Best approximation property]
  \label{pro:best_approximation}
  Let $\bbA$ be a $k$-th order $\setC$-elliptic operator of the form \eqref{eq:form}.  Let
  $\Omega \subset \setR^n$ be an open and bounded domain satisfying the emanating chain condition with constants $\sigma_{1},\sigma_{2}$ and central ball~$B$. Moreover, let $1<p<\infty$ and $w\in A_{p}$. Then for
  all $\ell\in\{0,...,k\}$ and all 
  $u \in \sobo^{\ell,p}(\Omega;V)$ we have
  \begin{align*}
    \bigg(\dashint_\Omega \abs{D^\ell (u-\Pi_\bbA^B u)}^p\,w\dif x\bigg)^{\frac
    1p} &\leq C\,
          \inf_{q \in \ker(\bbA)} \bigg(\dashint_\Omega \abs{D^\ell
          (u-q)}^p\,w\dif x\bigg)^{\frac 1p}.
  \end{align*}
  where $C=C(\opA,\sigma_2,p,[w]_{A_{p}})>0$. 
\end{proposition}
\begin{proof}
We start by noting that, if $\ball\subset\Omega$ is an open ball with $\diam(\Omega)\leq c_{0}\diam(\ball)$, then for any $m\in\mathbb{N}_{0}$ there exists a constant $c_{1}=c_{1}(c_{0},m,n,V)>0$ such that for all $\pi\in\mathscr{P}_{m}(\R^{n};V)$ and all $\ell\in\mathbb{N}_{0}$ with $\ell\leq m+1$ there holds 
\begin{align}\label{eq:inverseweighted}
\begin{split}
\Big(\dashint_{\Omega}|D^{\ell}\pi|^{p}w\dif x\Big)^{\frac{1}{p}} & \leq (\sup_{\Omega}|D^{\ell}\pi|)\Big(\frac{w(\Omega)}{\mathscr{L}^{n}(\Omega)}\Big)^{\frac{1}{p}} \\ & \leq  c_{1}r(\ball)^{-\ell}\dashint_{\ball}|\pi|\dif x \Big(\frac{w(\Omega)}{\mathscr{L}^{n}(\Omega)}\Big)^{\frac{1}{p}}\\ 
& \leq c_{1}[w]_{A_{p}}^{\frac{1}{p}}r(\ball)^{-\ell}\Big(\dashint_{\ball}|\pi|^{p}w\dif x\Big)^{\frac{1}{p}}, 
\end{split}
\end{align}
where the ultimate line follows by H\"{o}lder's inequality and because $w$ is doubling. 

Now let $\ell\in\{0,...,k\}$ and let $q \in \ker(\A)$ be arbitrary, so that $\Pi_\A^B q=q$.  Moreover, 
  let $q_1 \in \Pol_{\ell-1}\subset\ker(\A)$ be arbitrary. Then $\Pi_\A^B q_1=q_1$ and
  therefore $D^\ell \Pi_\A^B q_1 = 0$.  Now,
  $D^\ell (u - \Pi_\A^B u) = D^\ell (u - q) - D^\ell \Pi_\A^B
  (u-q-q_1)$, inverse estimates for polynomials in conjunction with
  $\diameter(\Omega) \leq c_0(\sigma_2) \diameter(B)$ (cf.~\eqref{eq:mutualradiusECCbound}) and
 \eqref{eq:inverseweighted} imply
  \begin{align*}
    \lefteqn{\bigg(\dashint_\Omega \abs{D^\ell (u-\Pi_\bbA^B u)}^p w\dif x\bigg)^{\frac
    1p}} \qquad &
    \\
                &\leq C\,
                  \bigg(\dashint_\Omega \abs{D^\ell (u-q)}^p w\dif x\bigg)^{\frac
                  1p} + C\,     \bigg(\dashint_\Omega \abs{ D^\ell \Pi_\bbA^B(u-q-q_1)}^p w\dif x\bigg)^{\frac
                  1p}
    \\
                &  \!\!\!\stackrel{\eqref{eq:inverseweighted}}{\leq} C\,
                  \bigg(\dashint_\Omega \abs{D^\ell (u-q)}^p w\dif x\bigg)^{\frac
                  1p} + C\, \diameter(\Omega)^{-\ell}     \bigg(\dashint_B \abs{ u-q-q_1}^p w\dif x\bigg)^{\frac
                  1p}.
  \end{align*}
  Now we may fix~$q_1 \in \Pol_{\ell-1}(\R^{n};V)$ such that the standard $A_{p}$-weighted
  \Poincare{} inequality for balls (see \cite[Chpt.~1.4]{HeiKilMar}) yields
  \begin{align*}
    \diameter(\Omega)^{-\ell}     \bigg(\dashint_B \abs{ u-q-q_1}^p w\dif x\bigg)^{\frac
    1p}
    &\leq C\,
         \bigg(\dashint_B \abs{ D^\ell(u-q)}^p w\dif x\bigg)^{\frac
      1p}.
  \end{align*}
  This, the previous estimate and   $\diameter(\Omega) \leq
  c_0(\sigma_2) \diameter(B)$ prove the claim.
\end{proof}

\subsection{An inequality of Fefferman-Stein-type}\label{sec:feffermanstein}
In this section we derive an inequality which allows to control the information of a function by
its $\mathcal{N}$-sharp maximal operator. This serves as an instrumental tool in the proof of Theorem~\ref{thm:Korn} in Section~\ref{sec:Kornmain}, but should be of independent interest.

We require the following restricted variant of maximal operators: Given an open subset $\Omega\subset\R^{n}$, $\sigma\geq 1$ and, for some finite
dimensional real vector space $E$, a subset
$\mathcal{N}\subset\mathscr{P}_{m}(\R^{n};E)$ for some $m\in\mathbb{N}$ and $1\leq p <\infty$, we put
\begin{align}
&\mathcal{M}_{\mathrm{res},\Omega,\sigma,p}f(x) := \begin{cases}\displaystyle\sup_{Q\ni x:\,\sigma Q\subset\Omega}\Big(\dashint_{Q}|f|^{p}\dif y\Big)^{\frac{1}{p}},&\text{if}\;x\in\Omega\\ 
0&\;\text{otherwise}
\end{cases}\\
&\mathcal{M}_{\mathrm{res},\Omega,\sigma,p,\mathcal{N}}^{\sharp}f(x):=\begin{cases}\displaystyle\sup_{Q\ni x:\,\sigma Q\subset\Omega}\inf_{\pi\in\mathcal{N}}\Big(\dashint_{Q}|f-\pi|^{p}\dif y\Big)^{\frac{1}{p}}, &\text{if}\;x\in\Omega\\
0&\text{otherwise}
\end{cases}
\end{align}
for $f\in\lebe_{\locc}^{p}(\Omega;E)$ and non-degenerate cubes $Q$. For brevity, we also set $\mathcal{M}_{\mathrm{res},\Omega,\sigma}:=\mathcal{M}_{\mathrm{res},\Omega,\sigma,1}$ and $\mathcal{M}_{\mathrm{res},\Omega,\sigma,\mathcal{N}}^{\sharp}:=\mathcal{M}_{\mathrm{res},\Omega,\sigma,1,\mathcal{N}}^{\sharp}$. 

The main result of this section is a generalisation
of~\cite[Theorem~5.23]{DieRuzSch10} for~$\mathcal{N}=\setR$ and its unweighted version~\cite[Lemma~4]{Iwa82}. We present the proof later in this section.
\begin{theorem}[of Fefferman-Stein-type]
  \label{thm:fefferman-stein}
  Let $\Omega \subset \setR^n$ be an open and bounded domain satisfying the emanating chain condition with constants $\sigma_{1}>1,\sigma_{2}\geq 1$ and central ball~$B$.
  % Let $\Omega \subset \setR^n$ be a bounded domain satisfying the
  % emanating chain condition with constants $\sigma_1$ and $\sigma_2$.
  Let $1 < q < \infty$, $w \in A_q$ and $f \in \lebe^1_\loc(\Omega;E)$.  If $\mathcal{M}_{\mathrm{res},\Omega,\sigma_1,\mathcal{N}}^{\sharp}f \in \lebe^q_w(\Omega)$,
  then $f \in \lebe^q_w(\Omega;E)$ and
  \begin{align}
    \label{eq:9}
    \inf_{\pi \in \mathcal{N}}
    \bignorm{f-\pi}_{\lebe^q_w(\Omega)} &\leq c\,
    \bignorm{\mathcal{M}_{\mathrm{res},\Omega,\sigma_1,\mathcal{N}}^{\sharp}f}_{\lebe^q_w(\Omega)}.
  \end{align}
  The constant~$c$ depends only on~$\sigma_1$, $\sigma_2$, $\mathcal{N}$, $q$, and
  $[w]_{A_{q}}$.
  % If $f \in L^1(\Omega)$, then it suffices to assume
  % $\sigma_1 \geq 1$.
\end{theorem}
For a cube~$Q$ (or ball) we define $\Pi_{\mathcal{N}}^Q$ as the
$\lebe^2(Q;E)$-projection to~$\mathcal{N}$. It follows easily by inverse
estimates on~$\mathcal{N}$ that $\Pi_{\mathcal{N}}^Q$ is
$\lebe^1(Q;E)$-stable, i.e.
\begin{align}
  \label{eq:PiNstab}
  \dashint_Q \abs{\Pi_{\mathcal{N}}^Q f}\dif x \leq C_1
  \dashint_Q \abs{f}\dif x,
\end{align}
and, similarly to Proposition~\ref{pro:best_approximation}, has the approximation property
\begin{align}
  \label{eq:PiNapprox}
  \dashint_Q \abs{f - \Pi_\mathcal{N}^Q f}\dif x \leq C_2
  \inf_{q \in \mathcal{N}} \dashint_Q \abs{f - q}\dif x.
\end{align}
%\begin{remark}
%  Let $\Omega$, $B$, $w$ and $f$ be as in
%  Theorem~\ref{thm:fefferman-stein}. Similar to
%  Proposition~\ref{pro:best_approximation} it follows by the $\lebe^q$-stability of $\Pi^Q_{\mathcal{N}}$ and H\"older's inequality and $w\in A_q$
%  that
%  \begin{align*}
%    \inf_{q \in \mathcal{N}}
%    \bignorm{f-q}_{\lebe^q_{w,0}(\Omega)}
%    &\eqsim \bignorm{f-\Pi^B_{\mathcal{N}}f}_{\lebe^q_{w,0}(\Omega)}.
%  \end{align*}
%  Thus, we can replace the left-hand side of~\eqref{eq:9} by any of these expressions.  Alternatively, if $f\in \lebe^q_w(\Omega;E)$, then
%  we can also use $\norm{f}_{L^q(\Omega)}$ without any projection. 
%\end{remark}
We need the following
version of the \Calderon{}-Zygmund decomposition by \textsc{Iwaniec}:
\begin{lemma}[{\cite[Lemma~3]{Iwa82}}]
  \label{lem:cz}
  Let $Q_0 \subset \Rn$ be an open cube and let $f \in \lebe^1(Q_0;E)$. For
  every $\alpha \geq (\abs{f})_{Q_0}$ there exist pairwise disjoint cubes
  $Q^\alpha_j \subset Q_0$ indexed by $j \in \setN$ such that
  \begin{enumerate}
  \item \label{itm:cza} $\alpha < \dashint_{Q^\alpha_j} \abs{f}\dif x
    \leq 2^n \alpha$,
  \item \label{itm:czb} If $\alpha \geq \beta \geq
    (\abs{f})_{Q_0}$, then each cube $Q^\alpha_j$ is contained in an element of $\set{Q_j^\beta\,:\, j \in \setN}$,
  \item \label{itm:czc} $\abs{f} \leq \alpha$ on $Q_0
    \setminus \bigcup_j Q^\alpha_j$,
  \item \label{itm:czd} $\bigcup_j Q_j^\alpha \subset
    \set{\mathcal{M}_{\mathrm{res},Q_0,1} f > \alpha}$ for all $\alpha \geq
    (\abs{f})_{Q_0}$,
  \item \label{itm:cze} $\set{\mathcal{M}_{\mathrm{res},Q_0,1} f> 5^n \alpha} \subset
    \bigcup_j 5Q_j^\alpha$ for all $\alpha \geq
    (\abs{f})_{Q_0}$.
  \end{enumerate}
\end{lemma}
We start with a version of Theorem~\ref{thm:fefferman-stein} for the case
where~$\Omega$ is just a cube.
\begin{lemma}
  \label{lem:iwaniec_omega}
  Let $Q_0 \subset \Rn$ be an open cube, $1 \leq q < \infty$, $w
  \in A_\infty$, and let $f \in \lebe^1(Q_0;E)$. If $\mathcal{M}_{\mathrm{res},Q_0,1,\mathcal{N}}^{\sharp}f
  \in \lebe^q_w(Q_0)$, then $\mathcal{M}_{\mathrm{res},Q_0,1} f \in \lebe^q_w(Q_0)$ and there holds
  \begin{align*}
    \int_{Q_0} \abs{\mathcal{M}_{\mathrm{res},Q_0,1} f}^q w\dif x &\leq c\, \int_{Q_0}
    \abs{\mathcal{M}_{\mathrm{res},Q_0,1,\mathcal{N}}^{\sharp} f}^q w\dif x + c\,
    w(Q_0) \bigg(\dashint_{Q_0} \abs{f}\dif x \bigg)^q ,
  \end{align*}
  where $c=c(n,\mathcal{N},q,[w]_{A_{\infty}})>0$. 
\end{lemma}
\begin{proof}
  Since $w \in A_\infty$, there exists a constant $c_0 >0$, only depending on $[w]_{A_{\infty}}$, such that $w(5Q) \leq
  c_0 w(Q)$ for any cube~$Q \subset \Rn$, see~\cite[V~1.6,
  p.~196]{Ste93}. By the usual inverse estimates and \eqref{eq:PiNstab}, there exists $\mathtt{C}=\mathtt{C}(\mathcal{N},n)>0$ such that, for any open cube $Q\subset\R^{n}$,
\begin{align}\label{eq:stabboundweighted}
\|\Pi_{\mathcal{N}}^{Q}f\|_{\lebe^{\infty}(Q)}\leq \mathtt{C}\dashint_{Q}|f|\dif x.
\end{align}
Then we fix a number $m\in\mathbb{N}$ such that $\mathtt{C}2^{n-m}<\frac{1}{2}$. We claim that for every $\epsilon>0$ there exists $\delta=\delta(\varepsilon,n,m,\mathcal{N},[w]_{A_{\infty}})>0$ such that
  \begin{align}
    \label{eq:iw_o2}
    \begin{aligned}
      w(\set{\mathcal{M}_{\mathrm{res},Q_{0},1}f > 5^n \alpha}) &\leq c_0 w(
      \set{\mathcal{M}_{\mathrm{res},Q_0,1,\mathcal{N}}^{\sharp}f > \delta \alpha })
      \\
      &\qquad + c_0 \epsilon\,w( \set{\mathcal{M}_{\mathrm{res},Q_0,1} f > 2^{-m}
        \alpha} )
    \end{aligned}
  \end{align}
  for all $f \in \lebe^1(Q_0;E)$ and all $\alpha \geq 2^{m}
  (\abs{f})_{Q_0}$. Let $\varepsilon>0$. As $\alpha\geq 2^{-m}\alpha \geq (|f|)_{Q_{0}}$, we may pick $(Q_{j}^{\alpha})_{j\in\mathbb{N}}$ as in Lemma~\ref{lem:cz}. Toward \eqref{eq:iw_o2}, it suffices to show by virtue of Lemma~\ref{lem:cz}~\ref{itm:cze} and $w(5Q) \leq c_0
  w(Q)$ for all cubes $Q$ that
  \begin{align}
    \label{eq:iwa_om_good2}
    \sum_j w(Q_j^\alpha) &\leq w(
    \set{\mathcal{M}_{\mathrm{res},Q_0,1,\mathcal{N}}^{\sharp} f > \delta \alpha }) + \epsilon\,w(
    \set{\mathcal{M}_{\mathrm{res},Q_0,1} f > 2^{-m} \alpha} ).
  \end{align}
 By Lemma~\ref{lem:cz}~\ref{itm:czb}, every $Q_{j}^{\alpha}$ is contained in some $Q$, $Q\in\{Q_{i}^{2^{-m}\alpha}\colon i\in\mathbb{N}\}$. Then
\begin{align*}
\dashint_{Q_{j}^{\alpha}}|\Pi_{\mathcal{N}}^{Q}f|\dif x \leq \|\Pi_{\mathcal{N}}^{Q}f\|_{\lebe^{\infty}(Q)} \stackrel{\eqref{eq:stabboundweighted}}{\leq} \mathtt{C}\dashint_{Q}|f|\dif x \stackrel{\text{Lemma}~\ref{lem:cz}~\ref{itm:cza}}{\leq} \mathtt{C}2^{n-m}\alpha \stackrel{\eqref{eq:stabboundweighted}\mathrm{ff.}}{<}\frac{\alpha}{2}. 
\end{align*}
Thus, by~\ref{itm:cza} of Lemma~\ref{lem:cz} we obtain for all $j\in\mathbb{N}$
  \begin{align*}
    \dashint_{Q^\alpha_j}|f-\Pi_{\mathcal{N}}^Q f|\dif x \ge
    \dashint_{Q^\alpha_j}|f|\dif x-\dashint_{Q_j^\alpha} \abs{\Pi_{\mathcal{N}}^Q f}\dif x\ge \alpha - \frac{\alpha}{2} =\frac\alpha2 .
  \end{align*}
Multiplying the previous inequality by
  $\abs{Q_j^\alpha}$ and summing over all $Q^\alpha_j\subset Q$, we obtain
  \begin{align*}
    \sum_{Q^\alpha_j \subset Q} \abs{Q_j^\alpha} \leq \frac{2}{\alpha}
    \int_Q \abs{f-\Pi_{\mathcal{N}}^Q f}\dif x.
  \end{align*}
Since $w \in
  A_\infty$, there exists an $\epsilon_2>0$ such that if $\mathcal{E} \subset Q$
  with $\abs{\mathcal{E}} \leq \epsilon_2 \abs{Q}$, then $w(\mathcal{E}) \leq \epsilon
  w(Q)$, see~\cite[V~1.7, p.~196]{Ste93}; here, ~$\epsilon_2$ only
  depends on $\epsilon$ and the $A_\infty$-constant of $w$ but is
  independent of~$Q$ and~$\mathcal{E}$. With $C_{2}$ as in \eqref{eq:PiNapprox}, we define $\delta:=\frac{1}{2C_{2}}\varepsilon_{2}$ and then consider two cases:

$\bullet$  If $\dashint_Q \abs{f-\Pi_{\mathcal{N}}^Q f}\dif x \leq \frac{\alpha}{2}
  \epsilon_2$, then $\sum_{Q^\alpha_j \subset Q} \abs{Q_j^\alpha} \leq
  \epsilon_2 \abs{Q}$.  By our choice of $\epsilon_2$, 
  $\sum_{Q^\alpha_j \subset Q} w(Q_j^\alpha) \leq \epsilon\,
  w(Q)$, where we have used that the~$Q^\alpha_j$, $j \in \setN$, are pairwise disjoint.

$\bullet$ If $\dashint_Q \abs{f-\Pi_{\mathcal{N}}^Q f}\dif x > \frac{\alpha}{2}
  \epsilon_2$, then we have, with $C_{2}$ as in \eqref{eq:PiNapprox}, $\mathcal{M}_{\mathrm{res},Q_0,1,\mathcal{N}}^{\sharp} f (x)> \frac{\alpha}{2C_{2}}
  \epsilon_2 = \delta \alpha$ for all points~$x \in Q$. In particular, $Q \subset
  \set{\mathcal{M}_{\mathrm{res},Q_0,1,\mathcal{N}}^{\sharp} f>\delta \alpha}$ and
  \begin{align*}
    \sum_{Q^\alpha_j \subset Q} w(Q_j^\alpha) \leq
    w(Q) = w(Q \cap \set{\mathcal{M}_{\mathrm{res},Q_0,1,\mathcal{N}}^{\sharp} f>\delta \alpha}).
  \end{align*}
  Combining both alternatives, we obtain
  \begin{align}\label{eq:ComboRockoSchamoni}
    \sum_{Q^\alpha_j \subset Q} w(Q_j^\alpha) \leq w(Q \cap
    \set{\mathcal{M}_{\mathrm{res},Q_0,1,\mathcal{N}}^{\sharp} f>\delta \alpha}) +\epsilon w(Q).
  \end{align}
By Lemma~\ref{lem:cz}~\ref{itm:czb} and since
  $\alpha \geq 2^{m} (\abs{f})_{Q_0}$, summing \eqref{eq:ComboRockoSchamoni} over the pairwise disjoint cubes~$Q
  \in\set{Q^{2^{-m}\alpha}_i\,:\,i\in\setN}$, all cubes from
  $\set{Q^\alpha_j\,:\, j \in \setN}$ are counted exactly once in the overall sum. We thus obtain~\eqref{eq:iwa_om_good2} and
 hereby ~\eqref{eq:iw_o2}. Once \eqref{eq:iw_o2} is established, we may exactly follow the layer cake argument of \textsc{Iwaniec}~\cite[Lemma~4]{Iwa82} to conclude.
\end{proof}
\begin{corollary}
  \label{cor:iwaniec_omega2}
  Let $Q_0 \subset \Rn$ be an open cube, $1 \leq q < \infty$, $w \in A_q$,
  and $f \in \lebe^1_\mathcal{N}(Q_{0};E)$. If $\mathcal{M}_{\mathrm{res},Q_0,1,\mathcal{N}}^{\sharp} f \in \lebe^q_w(Q_0;E)$, then
  $f \in \lebe^q_{w,\mathcal{N}}(Q_0;E)$ and
  \begin{align*}
    \norm{f}_{\lebe^q_{w}(Q_0)} &\leq c\, \bignorm{ \mathcal{M}_{\mathrm{res},Q_0,1,\mathcal{N}}^{\sharp}
      f}_{\lebe^q_w(Q_0)},
  \end{align*}
  where the constant $c$ only depends on $\mathcal{N}$, $q$ and $[w]_{A_{q}}$.
\end{corollary}
\begin{proof}
  Let $f \in \lebe^1_{\mathcal{N}}(Q_0;E)$ so that $\Pi_{\mathcal{N}}^{Q_{0}} f=0$. In conjunction with Lemma~\ref{lem:iwaniec_omega} and \eqref{eq:PiNapprox}, this yields
  \begin{align*}
    \norm{f}_{\lebe^q_{w}(Q_0)}^q
    &\leq
      \norm{\mathcal{M}_{\mathrm{res},Q_0,1} f} _{\lebe^q_{w}(Q_0)}^q
    \\
    &\leq c\,
      \bignorm{ \mathcal{M}_{\mathrm{res},Q_0,1,\mathcal{N}}^{\sharp} f} _{\lebe^q_{w}(Q_0)}^q + c\,
    w(Q_0) \bigg(\dashint_{Q_0} \abs{f-\Pi_{\mathcal{N}}^{Q_{0}}f }\dif x\bigg)^q
    \\
    &\leq c\,
      \bignorm{ \mathcal{M}_{\mathrm{res},Q_0,1,\mathcal{N}}^{\sharp} f} _{\lebe^q_{w}(Q_0)}^q + c\, 
      \int_{Q_0}  \bigabs{ (\mathcal{M}_{\mathrm{res},Q_0,1,\mathcal{N}}^{\sharp}f)(y)}^q w \dif y
    \\
    &\leq c\,
      \bignorm{ \mathcal{M}_{\mathrm{res},Q_0,1,\mathcal{N}}^{\sharp} f} _{\lebe^q_{w}(Q_0)}^q,
  \end{align*}
  and the proof is complete. 
\end{proof}
Before we come to the proof of Theorem~\ref{thm:fefferman-stein}, the underlying duality argument requires
\begin{lemma}
  \label{lem:density}
  Let $\Omega \subset \Rn$ be as in Theorem~\ref{thm:fefferman-stein}.  Let $1 < q < \infty$ and $w \in A_q$. Then
  $\hold^\infty_{c,\mathcal{N}}(\Omega;E)$ is dense in $\lebe^q_{w,\mathcal{N}}(\Omega;E)$ for the norm topology.
\end{lemma}
\begin{proof}
  Let $\eta_0 \in \hold^\infty_c(B)$ with $\eta_0 \geq 0$ and
  $\int_B \eta_0\dif x=1$.  Similarly as in~\eqref{eq:projDEFKorn} we define a linear operator ~$\Pi_0\,:\, \lebe^{1}(\Omega;E) \to \mathcal{N}$ by
  \begin{align*}
    \int_B \eta_0 (\Pi_0 f)\xi \dif x = \dashint_\Omega f \xi\dif x \qquad \text{for all $\xi \in \mathcal{N}$}.
  \end{align*}
  By inverse estimates on~$\mathcal{N}$ and H\"older's inequality we
  obtain
  $ \|\Pi_0 f\|_{\lebe^{\infty}(\Omega)} \leq c \dashint_\Omega \abs{f}\dif x$. 
  This and~$w \in A_q$ imply 
  $\norm{\eta_0 \Pi_0 f}_{\lebe^q_w(\Omega)} \leq c \norm{f}_{\lebe^q_w(\Omega)}$. Now let $f\in\lebe_{w,\mathcal{N}}^{q}(\Omega;E)$ so that, in particular, $\Pi_0 f =0$. Then we find
  $(g_{j})\subset \hold^\infty_{c}(\Omega;E)$ such that $g_j \to f$ in
  $\lebe^q_{w}(\Omega;E)$, cf.~\textsc{Miller} \cite[Lem.~2.1]{Miller82}. Put $\widetilde{\Pi}_{0}:=|\Omega|\Pi_{0}$ and define
  $f_j := g_j - \eta_{0}\widetilde{\Pi}_0 g_j \in \hold^\infty_{c,\mathcal{N}}(\Omega;E)$. Since $\Pi_0 f=0$ we have $f_j= g_j - \eta_{0}\widetilde{\Pi}_0(g_j -f)$. Since $g_j \to f$ in $\lebe^q_w(\Omega;E)$ it follows with  $\norm{\eta_0 \Pi_0 (g_j-f)}_{\lebe^q_w(\Omega)} \leq c \norm{g_j-f}_{\lebe^q_w(\Omega)}$ that $f -f_j \to 0$ in $\lebe^q_w(\Omega;E)$, as desired.
\end{proof}
\begin{proof}[Proof of Theorem~\ref{thm:fefferman-stein}]
  Let $1<q<\infty$, $w \in A_q$, and $f \in \lebe^1_{\locc}(\Omega;E)$ with
  $\mathcal{M}_{\mathrm{res},\Omega,\sigma_1,\mathcal{N}}^{\sharp} f \in \lebe^q_{w}(\Omega)$.  Since $w \in
  A_q$, we have $w' := w^{\frac{1}{1-q}} \in A_{q'}$.

  Let $h \in \hold^\infty_{c,\mathcal{N}}(\Omega;E)$. By
  Theorem~\ref{thm:decomposition} we may then decompose $h$  into the sum of
  functions~$T_i h \in \hold^\infty_{c,\mathcal{N}}(W_i;E)\subset\lebe^{q'}_{w',\mathcal{N}}(W_i;E)$ such that
  \begin{align}
    \label{eq:18b}
    \bigg( \sum_{i \geq 0} \norm{T_i h}^{q'}_{\lebe^{q'}_{w'}(W_i)}
    \bigg)^{\frac{1}{q'}} &\leq c\,
    \norm{h}_{\lebe^{q'}_{w'}(\Omega)}.
  \end{align}
  Moreover, by Theorem~\ref{thm:decomposition}~\ref{itm:Cinfty}, only finitely many summands of $h = \sum_{i \geq 0}
  T_i h$ are non-zero.  Since $f \in \lebe^1_{\locc}(\Omega;E)$ and $h \in \hold^\infty_{c,\mathcal{N}}(\Omega;E)$,
  $\skp{f}{h}:=\int_{\Omega}fh\dif x$ is well-defined, and for all $i\in\mathbb{N}$, $\xi\in\mathcal{N}$ we have $\int_{W_i} \xi T_i h \,\dif x = 0$. Therefore, with $\Pi_{\mathcal{N}}^{W_{i}}$ as in \eqref{eq:PiNstab}ff.,
  \begin{align*}
    \skp{f}{h} &= \sum_i \skp{f}{T_i h} = \sum_i \skp{f -
      \Pi_{\mathcal{N}}^{W_{i}}f}{T_i h},
  \end{align*}
  where we also used that $f \in \lebe^1(W_i;E)$, since $\sigma_1>1$ and
  consequently $\overline {W}_1 \Subset \sigma_1 W_i \subset
  \Omega$.  We estimate with Corollary~\ref{cor:iwaniec_omega2} in the second
  step
  \begin{align}
    \label{eq:barf}
    \begin{aligned}
      \bigskp{f}{h} &\leq \sum_i \norm{f-\Pi_{\mathcal{N}}^{W_{i}}f}_{\lebe^q_{w}(W_i)}
      \norm{T_i h}_{\lebe^{q'}_{w'}(W_i)}
      \\
      &\leq c\,\sum_i \norm{\mathcal{M}_{\mathrm{res},W_i,1,\mathcal{N}}^{\sharp} f}_{\lebe^q_w(W_i)}
     \norm{T_i h}_{\lebe^{q'}_{w'}(W_i)}
      \\
      &\leq c\,\Big(\sum_i \norm{\mathbbm{1}_{W_{i}}\mathcal{M}_{\mathrm{res},\Omega,\sigma_{1},\mathcal{N}}^{\sharp} f}_{\lebe^q_w(\Omega)}^{q}\Big)^{\frac{1}{q}}\Big(\sum_{i}  \norm{T_i h}_{\lebe^{q'}_{w'}(W_i)}^{q'}\Big)^{\frac{1}{q'}}
      \\
      &\leq c\, \norm{\mathcal{M}_{\mathrm{res},\Omega,\sigma_1,\mathcal{N}}^{\sharp} f}_{\lebe^q_w(\Omega)}
      \norm{h}_{\lebe^{q'}_{w'}(\Omega)}
    \end{aligned}
  \end{align}
  for all $h \in \hold^\infty_{c,\mathcal{N}}(\Omega;E)$, where we have
  used in the last step $\sum_{W \in \mathcal{W}} \mathbbm{1}_{\sigma_1 W}
  \leq \sigma_2\,\mathbbm{1}_\Omega$ and~\eqref{eq:18b}. Therefore, invoking the smooth approximation from Lemma~\ref{lem:density}, 
  \begin{align}\label{eq:barfer}
  \sup_{\substack{h\in\lebe_{w',\mathcal{N}}^{q'}(\Omega;E)\\ \|h\|_{\lebe_{w'}^{q'}(\Omega)}\leq 1}}\bigskp{f}{h}  \leq c\norm{\mathcal{M}_{\mathrm{res},\Omega,\sigma_1,\mathcal{N}}^{\sharp} f}_{\lebe^q_w(\Omega)}.
  \end{align}
By an analogous argument to that of Proposition~\ref{pro:best_approximation} and because $\Pi_{\mathcal{N}}^{\ball}$ is $\lebe_{w}^{q}$-stable, 
\begin{align*}
\inf_{\pi\in\mathcal{N}}\|f-\pi\|_{\lebe_{w}^{q}(\Omega)} &\leq c \|f-\Pi_{\mathcal{N}}^{\ball}f\|_{\lebe_{w}^{q}(\Omega)} \\ & = c \sup_{\substack{h\in\lebe_{w'}^{q'}(\Omega;E)\\ \|h\|_{\lebe_{w'}^{q'}(\Omega)}\leq 1}}\bigskp{f-\Pi_{\mathcal{N}}^{\ball}f}{h-\Pi_{\mathcal{N}}^{\ball}h}+\underbrace{\bigskp{f-\Pi_{\mathcal{N}}^{\ball}f}{\Pi_{\mathcal{N}}^{\ball}h}}_{=0}  \\ 
& \leq c \sup_{\substack{\widetilde{h}\in\lebe_{w',\mathcal{N}}^{q'}(\Omega;E)\\ \|\widetilde{h}\|_{\lebe_{w'}^{q'}(\Omega)}\leq 1}}\bigskp{f}{\widetilde{h}} \stackrel{\eqref{eq:barfer}}{\leq} c\norm{\mathcal{M}_{\mathrm{res},\Omega,\sigma_1,\mathcal{N}}^{\sharp} f}_{\lebe^q_w(\Omega)}.
\end{align*} 
This completes the proof. 
\end{proof}

\subsection{Korn-type inequalities and the proof of Theorem~\ref{thm:Korn}}\label{sec:Kornmain} We now turn to Theorem~\ref{thm:KornMain} below, a special case of which is Theorem~\ref{thm:Korn}. Different from Theorem~\ref{thm:mainLp}, where the focus was on the equivalence of trace and Korn-type inequalities on \emph{smooth} domains, and previous contributions (cf.~\textsc{\Kalamajska{}} \cite{Kal93,Kal94}), the previous sections now allow us to cover the natural yet vast class of John domains which, in turn, might be very irregular:
\begin{theorem}[Weighted Korn inequality]\label{thm:KornMain}
Let $\Omega\subset\R^{n}$ be an open and bounded domain that satisfies the emanating chain condition with constants $\sigma_{1}>1$ and $\sigma_{2}\geq 1$ and central ball $\ball$. Moreover, let $k\in\mathbb{N}$, $1<p<\infty$ and $w\in A_{p}$. Then for any $k$-th order $\mathbb{C}$-elliptic operator $\A$ of the form \eqref{eq:form} there exists a constant $c=c(\sigma_{1},\sigma_{2},p,[w]_{A_{p}},\A)>0$ such that 
\begin{align}\label{eq:KornWeightedMain}
\begin{split}
\|D^{k}(u-\Pi_{\A}^{B}u)\|_{\lebe_{w}^{p}(\Omega)}& \leq c\|\A u\|_{\lebe_{w}^{p}(\Omega)}, \\ 
\|D^{k}u\|_{\lebe_{w}^{p}(\Omega)}& \leq c(\diam(\Omega)^{-k}\|u\|_{\lebe_{w}^{p}(\Omega)}+\|\A u\|_{\lebe_{w}^{p}(\Omega)})
\end{split}
\end{align}
hold for all $u\in\sobo^{\A}\lebe_{w}^{p}(\Omega):=\{v\in\lebe_{w}^{p}(\Omega;V)\colon\;\A v\in\lebe_{w}^{p}(\Omega;W)\}$. 
\end{theorem}
\begin{proof}
Since $w\in A_{p}$, by the open-endedness property of the Muckenhoupt classes \cite{CoiFef74}, there exists $\varepsilon=\varepsilon(n,p,[w]_{A_{p}})>0$ such that $w\in A_{p-\varepsilon}$, too. We then pick $1<q<p$ such that $p-\varepsilon < \frac{p}{q}<p$ and record in advance that $w\in A_{p/q}$. Let $\sigma_{1}>1$ and let $Q\subset\Omega$ be an open cube such that $\sigma_{1}Q\subset\Omega$. Choose a cut-off function $\rho\in\hold_{c}^{\infty}(\Omega;[0,1])$ such that $\mathbbm{1}_{Q}\leq\rho\leq\mathbbm{1}_{\sigma_{1}Q}$ and $|\nabla^{\ell}\rho|\leq c(\sigma_{1}-1)^{-l}\ell(Q)^{-l}$ for all $l\in\{0,...,k\}$. Then $\varphi:=\rho(u-\Pi_{\A}^{\sigma_{1}Q} u)$ (with $\Pi_{\A}^{\sigma_{1}Q}$ being as in Theorem~\ref{thm:poincare-john}) is compactly supported in $\R^{n}$. Since $\A$ in particular is elliptic, the Fourier multiplier argument underlying \eqref{eq:diffopFourier}ff. yields the existence of a constant $c=c(q,\A)>0$ such that
\begin{align}\label{eq:KornAux}
\|D^{k}\varphi\|_{\lebe^{q}(\R^{n})}\leq c\|\A\varphi\|_{\lebe^{q}(\R^{n})}. 
\end{align}
Moreover, note that $\Pi_{\A}^{Q}\Pi_{\A}^{\sigma_{1}Q}u=\Pi_{\A}^{\sigma_{1}Q}u$ on $Q$ and therefore, using the $\lebe^{q}$-stability of the projections $\Pi_{\A}^{Q}$ in the second and the Poincar\'{e}-type inequalities from Theorem~\ref{thm:poincare-john} in the fourth step,  
\begin{align}\label{eq:PoincAuxKorn}
\begin{split}
\Big(\dashint_{Q}|\Pi_{\A}^{\sigma_{1}Q}u & -\Pi_{\A}^{Q}u|^{q}\dif x \Big)^{\frac{1}{q}} = \Big(\dashint_{Q}|\Pi_{\A}^{Q}(u-\Pi_{\A}^{\sigma_{1}Q}u)|^{q}\dif x \Big)^{\frac{1}{q}}\\ 
& \leq c\Big(\dashint_{Q}|u-\Pi_{\A}^{\sigma_{1}Q}u|^{q}\dif x \Big)^{\frac{1}{q}}\\
& \leq c\Big(\dashint_{\sigma_{1}Q}|u-\Pi_{\A}^{\sigma_{1}Q}u|^{q}\dif x \Big)^{\frac{1}{q}} \stackrel{\eqref{eq:poincare-star}}{\leq} c\ell(Q)^{k}\Big(\dashint_{\sigma_{1}Q}|\A u|^{q}\dif x \Big)^{\frac{1}{q}},
\end{split}
\end{align}
where $c=c(\sigma_{1},q,\A)>0$. Now, using the Poincar\'{e}-type inequalities from Theorem~\ref{thm:poincare-john}, 
\begin{align*}
\Big(&\dashint_{Q}|D^{k}(u-\Pi_{\A}^{Q}u)|^{q}\dif x\Big)^{\frac{1}{q}}  \leq c\Big(\dashint_{\sigma_{1}Q}|D^{k}(\rho(u-\Pi_{\A}^{\sigma_{1}Q}u))|^{q}\dif x\Big)^{\frac{1}{q}} \\ & \;\;\;\;\;\;\;\;\;\;\;\;\;\;\;\;\;\;\;\;\;\;\;\;\;\;\;\;\;\;\;\;\;\;+c \Big(\dashint_{Q}|D^{k}(\Pi_{\A}^{\sigma_{1}Q}u-\Pi_{\A}^{Q}u)|^{q}\dif x \Big)^{\frac{1}{q}}\\ 
& \!\!\!\stackrel{\eqref{eq:KornAux}}{\leq} c\Big(\dashint_{\sigma_{1}Q}|\A(\rho(u-\Pi_{\A}^{\sigma_{1}Q}u))|^{q}\dif x\Big)^{\frac{1}{q}} + \Big(\dashint_{Q}|D^{k}(\Pi_{\A}^{\sigma_{1}Q}u-\Pi_{\A}^{Q}u)|^{q}\dif x \Big)^{\frac{1}{q}}\\ 
& \! \leq c\Big(\sum_{l=1}^{k}\ell(Q)^{-l}\Big(\dashint_{\sigma_{1}Q}|D^{k-l}(u-\Pi_{\A}^{\sigma_{1}Q}u)|^{q}\dif x\Big)^{\frac{1}{q}} \Big) + c\Big(\dashint_{\sigma_{1}Q}|\A u|^{q}\dif x\Big)^{\frac{1}{q}} \\ 
& \;\;\;\; + c\ell(Q)^{-k} \sum_{l=0}^{k}\ell(Q)^{l}\Big(\dashint_{Q}|D^{l}(\Pi_{\A}^{\sigma_{1}Q}u-\Pi_{\A}^{Q}u)|^{q}\dif x \Big)^{\frac{1}{q}}\\
& \!\!\!\!\!\!\!\!\!\!\stackrel{\eqref{eq:poincare-star}, \eqref{eq:inverseFull}_{2}}{\leq} c\Big(\dashint_{\sigma_{1}Q}|\A u|^{q}\dif x\Big)^{\frac{1}{q}}  + c\ell(Q)^{-k} \Big(\dashint_{Q}|\Pi_{\A}^{\sigma_{1}Q}u-\Pi_{\A}^{Q}u|^{q}\dif x \Big)^{\frac{1}{q}} \\
& \!\!\!\stackrel{\eqref{eq:PoincAuxKorn}}{\leq} c\Big(\dashint_{\sigma_{1}Q}|\A u|^{q}\dif x\Big)^{\frac{1}{q}},
\end{align*}
where again $c=c(\sigma_{1},q,\A)>0$. We define $\mathcal{N}:=D^{k}\ker(\A)$. Then, using that $D^{k}\Pi_{\A}^{Q}u\in\mathcal{N}$ for any open cube $Q$, 
\begin{align}\label{eq:sharpnonsharpest}
\mathcal{M}_{\mathrm{res},\Omega,\sigma_{1},q,\mathcal{N}}^{\sharp}(D^{k}u) \leq c \mathcal{M}_{\mathrm{res},\Omega,1,q}(\A u).
\end{align}
In the sequel, denote $\ball$ the central ball of $\Omega$, cf.~Definition~\ref{def:boman_chain}. We now invoke the best approximation property from Proposition~\ref{pro:best_approximation} with respect to the weight $w$ and $\ell=k$ in the first and the Fefferman-Stein-type inequality from Theorem~\ref{thm:fefferman-stein} with $f=D^{k}u$, $\mathcal{N}=D^{k}\ker(\A)$ in the second step to obtain 
\begin{align}\label{eq:almostKorn}
\begin{split}
\|D^{k}u-D^{k}\Pi_{\A}^{\ball}u\|_{\lebe_{w}^{p}(\Omega)} & \stackrel{\text{Proposition~\ref{pro:best_approximation}}}{\leq} c\inf_{\pi\in \mathcal{N}}\|f-\pi\|_{\lebe_{w}^{p}(\Omega)}\\ 
& \;\stackrel{\text{Theorem~\ref{thm:fefferman-stein}}}{\leq} c\|\mathcal{M}_{\mathrm{res},\Omega,\sigma_{1},q,\mathcal{N}}^{\sharp}(D^{k}u)\|_{\lebe_{w}^{p}(\Omega)} \\ 
& \;\;\;\,\,\stackrel{\eqref{eq:sharpnonsharpest}}{\leq} c \|\mathcal{M}_{\mathrm{res},\Omega,1,q}(\A u) \|_{\lebe_{w}^{p}(\Omega)} \\ 
& \;\;\;\;\;\;\leq c\|\mathcal{M}_{\mathrm{res},\Omega,1,1}(|\A u|^{q})\|_{\lebe_{w}^{p/q}(\Omega)}^{1/q}  \\ 
& \;\;\;\stackrel{1<q<p}{\leq} c\|\A u\|_{\lebe_{w}^{p}(\Omega)}, 
\end{split}
\end{align}
where we used that $\mathcal{M}_{\mathrm{res},\Omega,1,1}\colon\lebe_{w}^{p/q}(\Omega)\to\lebe_{w}^{p/q}(\Omega)$ is bounded provided $1<p/q<\infty$ and $w\in A_{p/q}$. Tracking the dependencies of the constants in \eqref{eq:almostKorn}, we obtain  $\eqref{eq:KornWeightedMain}_{1}$. To obtain $\eqref{eq:KornWeightedMain}_{2}$, the weighted inverse and stability estimates (\eqref{eq:inverseweighted} ff.) yield
\begin{align*}
\|D^{k}\Pi_{\A}^{\ball}u\|_{\lebe_{w}^{p}(\Omega)} \leq c\mathrm{diam}(\Omega)^{-k}\|\Pi_{\A}^{\ball}u\|_{\lebe_{w}^{p}(\Omega)}\leq c\mathrm{diam}(\Omega)^{-k}\|u\|_{\lebe_{w}^{p}(\Omega)},
\end{align*}
where $c=c(n,p,[w]_{A_{p}},\sigma_{1},\sigma_{2},\A)>0$. Combining this inequality with \eqref{eq:almostKorn}, \eqref{eq:KornWeightedMain} follows and the proof is complete. 
\end{proof}
\begin{proof}[Proof of Theorem~\ref{thm:Korn}]
For a $\mathbb{C}$-elliptic $k$-th order operator, Theorems~\ref{thm:KornMain} and~\ref{thm:poincare-john}  immediately imply Theorem~\ref{thm:Korn}, direction \ref{item:KornA}$\Rightarrow$\ref{item:KornB}, by taking $w\equiv 1$. For direction \ref{item:KornB}$\Rightarrow$\ref{item:KornA} of Theorem~\ref{thm:Korn} it suffices to note that Theorem~\ref{thm:mainLp} remains valid when only requiring~\ref{item:mainDLp} to hold for all \emph{connected}, open and bounded domains $\Omega\subset\R^{n}$; since every such domain is John, we thereby obtain \ref{item:KornB}$\Rightarrow$\ref{item:KornA} in Theorem~\ref{thm:Korn}.
\end{proof}
\begin{remark}[Korn for unbounded John domains] 
Theorem~\ref{thm:KornMain} also persists for unbounded John domains which are defined in analogy with Definition~\ref{def:john}. Indeed, by a result due to \textsc{V\"{a}is\"{a}l\"{a}} \cite[Thm.~4.6]{Vai94}, if $\Omega\subset\R^{n}$ is $\alpha$-John, then there exists a sequence of relatively compact $c\alpha$-John domains $\Omega_{j}\subset\Omega$ with $\Omega_{j}\nearrow\Omega$. Since any open and bounded John domain satisfies the emanating chain condition, we obtain by \eqref{eq:KornWeightedMain}
\begin{align*}
\|D^{k}u\|_{\lebe_{w}^{p}(\Omega_{j})}\leq C(\mathrm{diam}(\Omega_{j})^{-k}\|u\|_{\lebe_{w}^{p}(\Omega)}+\|\A u\|_{\lebe_{w}^{p}(\Omega)})\;\;\;\;\text{for all}\;u\in\sobo^{\A}\lebe_{w}^{p}(\Omega)
\end{align*}
for any $1<p<\infty$ and $w\in A_{p}$, with $C=C(c\alpha,p,[w]_{A_{p}},\A)>0$. Sending $j\to\infty$ in the preceding inequality then yields the Korn inequality
\begin{align*}
\|D^{k}u\|_{\lebe_{w}^{p}(\Omega)}\leq C\|\A u\|_{\lebe_{w}^{p}(\Omega)}\;\;\;\;\text{for all}\;u\in\sobo^{\A}\lebe_{w}^{p}(\Omega).
\end{align*} 
\end{remark}
\subsection{Korn-type inequalities for other space scales}\label{sec:Kornextend}
Theorem~\ref{thm:KornMain} easily allows for variants of Korn-type inequalities by means of extrapolation. For this, we recall a far-reaching variant of the \textsc{Rubio de Francia}-extrapolation theorem due to \textsc{Cruz-Uribe, Martell \& P\'{e}rez} \cite{CruMarPer11} that requires the following terminology: Given a Banach function space $(\mathbb{X},\|\cdot\|_{\mathbb{X}})$ on $\R^{n}$ with respect to $\mathscr{L}^{n}$, we call $(\mathbb{X},\|\cdot\|_{\mathbb{X}})$ \emph{rearrangement invariant} provided for each $f\in\mathbb{X}$ the norm $\|f\|_{\mathbb{X}}$ exclusively depends on its distribution function $d_{f}(\lambda):=\mathscr{L}^{n}(\{x\in\R^{n}\colon\;|f(x)|>\lambda\})$. Defining for $f\in\mathbb{X}$ its \emph{decreasing rearrangement} $f^{*}\colon [0,\infty)\to\R$ by $f^{*}(t):=\inf\{\lambda\geq 0\colon\;d_{f}(\lambda)\leq t\}$, there exists a unique rearrangement invariant Banach function space $(\overline{\mathbb{X}},\|\cdot\|_{\overline{\mathbb{X}}})$ on $[0,\infty)$ with respect to $\mathscr{L}^{1}$ such that $f\in\mathbb{X}$ precisely if $f^{*}\in\overline{\mathbb{X}}$. In this situation, defining for $0<t<\infty$ the dilation operator $D_{t}\colon\overline{\mathbb{X}}\ni g \mapsto g(\cdot/t)\in\overline{\mathbb{X}}$ and denoting $\|D_{t}\|_{\overline{\mathbb{X}}}$ its operator norm, the \emph{lower} or \emph{upper Boyd indices} of $\mathbb{X}$ are given by 
\begin{align}\label{eq:Boyd}
p_{\mathbb{X}}:=\lim_{t\to\infty}\frac{\log(t)}{\log(\|D_{t}\|_{\overline{\mathbb{X}}})}\;\;\text{and}\;\;q_{\mathbb{X}}:=\lim_{t\searrow 0}\frac{\log(t)}{\log(\|D_{t}\|_{\overline{\mathbb{X}}})},
\end{align}
respectively. 

On the other hand, a differentiable function $\varphi\colon\R_{\geq 0}\to\R_{\geq 0}$ with right-continuous, non-decreasing derivative $\varphi'$ such that $\varphi'(0)=0$ and $\varphi'(t)>0$ for $t>0$ is called an \emph{N-function}. If there exists $c_{1}>0$ such that $\varphi(2t)\leq c_{1}\varphi(t)$ holds for all $t>0$, we say that $\varphi$ satisfies the $\Delta_{2}$-condition; in this situation, the infimum over all such possible constants $c_{1}$ is denoted $\Delta_{2}(\varphi)$. If the convex conjugate of $\varphi$, $\widetilde{\varphi}(t):=\sup_{s>0}(st-\varphi(s))$, satisfies the $\Delta_{2}$-condition, we say that $\varphi$ satisfies the $\nabla_{2}$-condition and put $\nabla_{2}(\varphi):=\Delta_{2}(\widetilde{\varphi})$. 
\begin{lemma}[{\cite[Thm.~4.10,~4.15]{CruMarPer11}, \cite[Prop.~6.1]{DieRuzSch10}}]\label{lem:extrapolation}
Let $1<p_{0}<\infty$ and let $\mathcal{F}\subset\lebe_{\locc}^{1}(\R^{n})\times\lebe_{\locc}^{1}(\R^{n})$ be a family of pairs of non-negative functions not identically zero, such that for any $w\in A_{p_{0}}$ there exists $c=c(p_{0},[w]_{A_{p_{0}}})>0$ with
\begin{align}\label{eq:MucBound1}
\int_{\R^{n}}f(x)^{p_{0}}w(x)\dif x \leq c \int_{\R^{n}}g(x)^{p_{0}}w(x)\dif x \qquad\text{for all}\;(f,g)\in\mathcal{F}, 
\end{align}
the inequality tacitly implying finiteness of the left-hand side. Then the following hold: 
\begin{enumerate}
\item\label{item:extrapol1} If $(\mathbb{X},\|\cdot\|_{\mathbb{X}})$ is a rearrangement invariant function space whose lower and upper Boyd indices $p_{\mathbb{X}},q_{\mathbb{X}}$ satisfy $1<p_{\mathbb{X}}\leq q_{\mathbb{X}}<\infty$, then there holds
\begin{align}\label{eq:MucBound2}
\|f\|_{\mathbb{X}}\leq c\|g\|_{\mathbb{X}}\qquad\text{for all}\;(f,g)\in\mathcal{F}. 
\end{align}
\item\label{item:extrapol2} For any $N$-function $\varphi\colon\R_{\geq 0}\to\R_{\geq 0}$ with $\Delta_{2}(\varphi),\nabla_{2}(\varphi)<\infty$ there exists a constant $C=C(\Delta_{2}(\varphi),\nabla_{2}(\varphi),p_{0})>0$ such that there holds 
\begin{align}\label{eq:MucBound3}
\int_{\R^{n}}\varphi(f)\dif x \leq C\int_{\R^{n}}\varphi(cg)\dif x\qquad\text{for all}\;(f,g)\in\mathcal{F}. 
\end{align}
\end{enumerate}
\end{lemma}
By virtue of Theorem~\ref{thm:KornMain}, Lemma~\ref{lem:extrapolation} implies a whole family of Korn-type inequalities for $\mathbb{C}$-elliptic differential operators on John domains, out of which we point out two:
\begin{corollary}[Korn for Lorentz]\label{cor:Lorentz}
Let $\Omega\subset\R^{n}$ be an open and bounded domain that satisfies the emanating chain condition with constants $\sigma_{1}>1$ and $\sigma_{2}\geq 1$ and central ball $\ball$. Let $1<p<\infty$, $1\leq q < \infty$. Then for any $\mathbb{C}$-elliptic operator $\A$ of order $k$ there exists a constant $c=c(\sigma_{1},\sigma_{2},p,q,\A)>0$ such that for all $u\in\sobo^{\A}\lebe^{p,q}(\Omega)$ there holds 
\begin{align}\label{eq:LorentzKorn}
\begin{split}
&\|D^{k}(u-\Pi_{\A}^{\ball}u)\|_{\lebe^{p,q}(\Omega)}\leq c\|\A u\|_{\lebe^{p,q}(\Omega)}, \\ & \|D^{k}u\|_{\lebe^{p,q}(\Omega)}\leq c\mathrm{diam}(\Omega)^{-k}\|u\|_{\lebe^{p,q}(\Omega)}+c\|\A u\|_{\lebe^{p,q}(\Omega)}. 
\end{split}
\end{align}
\end{corollary}
\begin{proof} 
For $p_{0}:=\frac{1}{2}(p+1)$, put 
\begin{align*}
\mathcal{F}:=\{(\mathbbm{1}_{\Omega}\mathbbm{1}_{\{|D^{k}(u-\Pi_{\A}^{\ball}u)|\leq j\}}|D^{k}(u-\Pi_{\A}^{\ball}u)|,\mathbbm{1}_{\Omega}|\A u|)\colon\;u\in\sobo^{\A,p_{0}}(\Omega),j\in\mathbb{N}\}
\end{align*}
so that $\mathcal{F}$ meets the requirements of Lemma~\ref{lem:extrapolation} by virtue of Theorem~\ref{thm:KornMain}. The Lorentz space $\lebe^{p,q}(\R^{n})$ is rearrangement invariant and by our assumptions on $(p,q)$, $p_{\mathbb{X}}=q_{\mathbb{X}}=p\in (1,\infty)$ (cf.~\cite[p.~70]{CruMarPer11}). Therefore, \eqref{eq:MucBound2} and Fatou's lemma yield \eqref{eq:LorentzKorn} for all $u\in\sobo^{\A,p_{0}}(\Omega)$. Since $\sobo^{\A}\lebe^{p,q}(\Omega)\subset\sobo^{\A,p_{0}}(\Omega)$ as $1<p_{0}<p$ and $\Omega$ is bounded, this implies  $\eqref{eq:LorentzKorn}_{1}$ for all $u\in\sobo^{\A}\lebe^{p,q}(\Omega)$, and $\eqref{eq:LorentzKorn}_{2}$ follows by inverse estimates. 
\end{proof} 
\begin{remark}[$\lebe^{\infty}$-estimates]
If~ $\Omega\subset\R^{n}$ is as in Corollary~\ref{cor:Lorentz} and moreover is such that the  embedding $\sobo^{1}\lebe^{n,1}(\Omega;V)\hookrightarrow \hold(\overline{\Omega};V)$ holds (cf.~\cite{CianchiPick98,KKM99,SteinLorentz}), then Corollary~\ref{cor:Lorentz} immediately yields the stronger embedding $\sobo^{\A}\lebe^{n,1}(\Omega)\hookrightarrow\hold(\overline{\Omega};V)$ for any first order $\mathbb{C}$-elliptic differential operator $\A$. 
\end{remark}
In an analogous way as for Corollary~\ref{cor:Lorentz}, we obtain the following
\begin{corollary}[Korn for Orlicz]\label{cor:Orlicz}
Let $\Omega\subset\R^{n}$ and $\A$ be as in Corollary~\ref{cor:Lorentz} and let $\varphi\colon\R_{\geq 0}\to\R_{\geq 0}$ be an N-function that satisfies the $\Delta_{2}$- and $\nabla_{2}$-conditions. Then there exists $c=c(\sigma_{1},\sigma_{2},\Delta_{2}(\varphi),\nabla_{2}(\varphi),\A)>0$ such that for all $u\in\sobo^{\A}\lebe^{\varphi}(\Omega)$ there holds
\begin{align}\label{eq:OrliczKorn}
\begin{split}
&\int_{\Omega}\varphi(|D^{k}(u-\Pi_{\A}^{\ball}u)|)\dif x \leq c\int_{\Omega}\varphi(|\A u|)\dif x, \\ 
& \int_{\Omega}\varphi(|D^{k}u|)\dif x \leq c \mathrm{diam}(\Omega)^{n}\varphi\Big(\mathrm{diam}(\Omega)^{-k}\dashint_{\Omega}|u|\dif x\Big) + c\int_{\Omega}\varphi(|\A u|)\dif x.
\end{split}
\end{align}
\end{corollary}
\begin{remark}
In Theorem~\ref{thm:KornMain} and Corollaries~\ref{cor:Lorentz},~\ref{cor:Orlicz}, the main focus is on functions that do not vanish close to $\partial\Omega$, as otherwise the weaker ellipticity allows the use of global singular integral estimates (cf.~\cite{ConGme20,Kal93}). We also wish to point out that, whereas $\eqref{eq:KornWeightedMain}_{2},\eqref{eq:LorentzKorn}_{2}$ and $\eqref{eq:OrliczKorn}_{2}$ are in fact equivalent to the $\mathbb{C}$-ellipticity of $\A$, this is \emph{not} the case for inequalities  $\eqref{eq:KornWeightedMain}_{1},\eqref{eq:LorentzKorn}_{1}$ and $\eqref{eq:OrliczKorn}_{1}$, cf.~\textsc{Fuchs} \cite{Fuc11}. This can be seen by the operator $\varepsilon^{D}$ in $n=2$ dimensions, cf.~Example~\ref{ex:FuchsSeregin}, and the classical Cauchy-Pompeiu-formula. 
\end{remark}
\section{Appendix}\label{sec:appendix}
 
\subsection{$\A$-strict continuity of the trace operator on $\bv^{\A}$}\label{sec:BVAstrict} 
The approximation argument underlying the proof of  Lemma~\ref{lem:GaussGreen} to conclude the existence of a continuous linear trace operator $\trace_{\partial\Omega}\colon\sobo^{\A,p}(\Omega)\to{\besov}{_{p,p}^{k-1/p}}(\partial\Omega;V)$ cannot be applied to $\bv^{\A}(\Omega)$. This is due to the non-density of $\hold^{\infty}(\Omega;V)\cap\bv^{\A}(\Omega)$ in $\bv^{\A}(\Omega)$ for the norm topology. In the well-known case of $\bv(\Omega)$ (cf.~ \textsc{Ambrosio} et al. \cite{AmbFusPal00}), this issue is resolved by passing to the \emph{strict metric}. The suitable substitute here is the $\A$-strict metric, see Section~\ref{sec:repr-form}, but now we have to keep track of the intermediate derivatives as well. To this end, we require a multiplicative inequality as follows:  
\begin{lemma}[Intermediate derivatives]\label{lem:intermediate}
Let $k\geq 1$ and let $\A$ be a $k$-th order $\mathbb{C}$-elliptic differential operator of the form \eqref{eq:form}. Then there exists a constant $c=c(\A)>0$ such that for all $l\in\{1,...,k-1\}$ there holds 
\begin{align}\label{eq:interpol}
\|D^{l}u\|_{\lebe^{1}(\R^{n})}\leq c\|u\|_{\lebe^{1}(\R^{n})}^{1-\frac{l}{k}}\|\A u\|_{\lebe^{1}(\R^{n})}^{\frac{l}{k}}
\end{align}
for all $u\in\sobo^{\A,1}(\R^{n})$ with compact support.
\end{lemma}
\begin{proof}
The proof is an adaptation of an argument to be found, e.g., in \textsc{Maz'ya} \cite[Thm.~1.4.7]{Maz85}. Let $u\in\hold_{c}^{\infty}(\R^{n};V)\setminus\{0\}$ be supported in $\ball(x_{0},R_{0})$ and $l\in\{1,...,k-1\}$ be given.  Let $x\in\spt(u)$, so that because of $u\in\hold^{\infty}(\R^{n};V)$, $\|u\|_{\lebe^{1}(\ball(x,R))}>0$ for all $R>0$. Fix $r_{0}>0$. There are two options: If $\|u\|_{\lebe^{1}(\ball(x,r_{0}))}<r_{0}^{k}\|\A u\|_{\lebe^{1}(\ball(x,r_{0}))}$, then we have
\begin{align}\label{eq:semi1}
r_{0}^{-l}<\Big(\frac{\|\A u\|_{\lebe^{1}(\ball(x,r_{0}))}}{\|u\|_{\lebe^{1}(\ball(x,r_{0}))}}\Big)^{\frac{l}{k}},
\end{align}
and define $r_{x}=r_{0}$ in this case. Then, by virtue of the Poincar\'{e} inequality from Theorem~\ref{thm:poincare-john},
\begin{align}\label{eq:semi2}
\begin{split}
\|D^{l}u\|_{\lebe^{1}(\ball(x,r_{x}))}& \stackrel{\eqref{eq:poincare-star2}}{\leq} c r_{x}^{-l}\|u\|_{\lebe^{1}(\ball(x,r_{x}))} + cr_{x}^{k-l}\|\A u\|_{\lebe^{1}(\ball(x,r_{x}))} \\ 
& \!\!\!\!\!\!\!\!\stackrel{\eqref{eq:semi1}}{\leq} c\|u\|_{\lebe^{1}(\ball(x,r_{x}))}^{1-\frac{l}{k}}\|\A u\|_{\lebe^{1}(\ball(x,r_{x}))}^{\frac{l}{k}}+cr_{0}^{k-l}\|\A u\|_{\lebe^{1}(\ball(x,r_{x}))}.
\end{split}
\end{align}
If $\|u\|_{\lebe^{1}(\ball(x,r_{0}))}\geq r_{0}^{k}\|\A u\|_{\lebe^{1}(\ball(x,r_{0}))}$, we then find $r_{x}\geq r_{0}$ such that $\|u\|_{\lebe^{1}(\ball(x,r_{x}))}= r_{x}^{k}\|\A u\|_{\lebe^{1}(\ball(x,r_{x}))}$. This follows from $u$ being compactly supported and the fact that $\|u\|_{\lebe^{1}(\ball(x,R))}$ will eventually be constant for large values of $R$. Fix this $r_{x}$, so that 
\begin{align}\label{eq:semi3} 
\begin{split}
\|D^{l}u\|_{\lebe^{1}(\ball(x,r_{x}))} \stackrel{\eqref{eq:poincare-star2}}{\leq} c\|u\|_{\lebe^{1}(\ball(x,r_{x}))}^{1-\frac{l}{k}}\|\A u\|_{\lebe^{1}(\ball(x,r_{x}))}^{\frac{l}{k}}
\end{split}
\end{align}
Adding \eqref{eq:semi2} and \eqref{eq:semi3}, in each of the cases there holds 
\begin{align}\label{eq:interinter1}
\begin{split}
\|D^{l}u\|_{\lebe^{1}(\ball(x,r_{x}))} 
& \leq c(r_{0}^{k-l}\|\A u\|_{\lebe^{1}(\ball(x,r_{x}))}+\|u\|_{\lebe^{1}(\ball(x,r_{x}))}^{1-\frac{l}{k}}\|\A u\|_{\lebe^{1}(\ball(x,r_{x}))}^{\frac{l}{k}} ).
\end{split}
\end{align}
Now, $\mathscr{B}:=(\ball(x,r_{x}))_{x\in\spt(u)}$ is a  covering of $\spt(u)$ by open balls such that $r_{x}\leq C<\infty$ and the sequence of radii of any sequence of mutually disjoint balls in $\mathscr{B}$ tends to zero. By the Besicovitch-type covering lemma in the form as given in \textsc{Maz'ya} \cite[Chpt.~1.2.1, Thm.~1]{Maz85}, there exists a sequence of balls $\ball_{m}\in\mathscr{B}$ such that $\spt(u)\subset\bigcup_{m}\ball_{m}$, the balls $\tfrac{1}{3}\ball_{m}$ are mutually disjoint and $\bigcup_{B\in\mathscr{B}}B\subset \bigcup_{m}4B_{m}$. Moreover, there exists a constant $N=N(n)\in\mathbb{N}$ such that every point belongs to at most $N$ balls contained in $\mathscr{B}$. We then conclude by \eqref{eq:interinter1} that
\begin{align*}
\|D^{l}u\|_{\lebe^{1}(\spt(u))} & \leq \sum_{m} \|D^{l}u\|_{\lebe^{1}(\ball_{m})} \\ 
& \leq c\sum_{m}\Big(r_{0}^{k-l}\|\A u\|_{\lebe^{1}(\ball_{m})}+\|u\|_{\lebe^{1}(\ball_{m})}^{1-\frac{l}{k}}\|\A u\|_{\lebe^{1}(\ball_{m})}^{\frac{l}{k}}  \Big) \\ 
& \leq cr_{0}^{k-l}\|\A u\|_{\lebe^{1}(\R^{n})}+c\|u\|_{\lebe^{1}(\R^{n})}^{1-\frac{l}{k}}\|\A u\|_{\lebe^{1}(\R^{n})}^{\frac{l}{k}}, 
\end{align*}
where $c=c(n,\A)>0$. Now it suffices to send $r_{0}\searrow 0$ to conclude \eqref{eq:interpol} for $u\in\hold_{c}^{\infty}(\R^{n};V)$; the general case of compactly supported $u\in\sobo^{\A,1}(\R^{n})$ follows by smooth approximation. The proof is complete. 
\end{proof}
We now turn to the remaining part of the 
\begin{proof}[Proof of Theorem~\ref{rem:tracestrict}]
We start by noting that the existence of a norm-continuous trace operator $\trace_{\partial\Omega}\colon{\sobo}{^{\A,1}}(\Omega)\to{\besov}{_{1,1}^{k-1}}(\partial\Omega;V)$ follows by smooth approximation along the same lines as for the case $1<p<\infty$ (cf.~Theorem~\ref{thm:mainLp}, direction '\ref{item:mainALp}$\Rightarrow$\ref{item:mainBLp}$\Rightarrow$\ref{item:mainCLp}' and Lemma~\ref{lem:GaussGreen}, now using Theorem~\ref{thm:main}), the reason being the density of $\widetilde{\sobo}{^{\A,1}}(\Omega)\cap\hold^{\infty}(\Omega;V)$ in $\sobo^{\A,1}(\Omega)$ for the norm topology on $\sobo^{\A,1}(\Omega)$; also note that $\sobo^{\A,1}(\Omega)={\widetilde{\sobo}}{^{\A,1}}(\Omega)$ by Theorem~\ref{thm:poincare-john}.

By Theorem~\ref{thm:poincare-john}, $\bv^{\A}(\Omega)=\sobo^{k-1,1}(\Omega;V)\cap\bv^{\A}(\Omega)$. To define an $\A$-strictly continuous trace operator by means of smooth approximation, the localisation argument below requires $d(u,u_{j}):=\|u_{j}-u\|_{\sobo^{k-1,1}(\Omega)}+|\,|\A u_{j}|(\Omega)-|\A u|(\Omega)|\to 0$ whenever $d_{\A}(u,u_{j})=\|u_{j}-u\|_{\lebe^{1}(\Omega)}+||\A u_{j}|(\Omega)+|\A u|(\Omega)|\to 0$ as $j\to\infty$. 
The Poincar\'{e} inequality of Theorem~\ref{thm:poincare-john} does not directly allow for this conclusion.  Also note that, if we define a  trace operator as the continuous extension for the metric $d$ as above, then it is not clear that this trace operator will be continuous for $d_{\A}$. It is here where we require the multiplicative inequality of Lemma~\ref{lem:intermediate}.

Let $u\in\bv^{\A}(\Omega)=\sobo^{k-1,1}(\Omega;V)\cap\bv^{\A}(\Omega)$ and choose $(u_{j})\subset \hold^{\infty}(\Omega;V)\cap\bv^{\A}(\Omega)$ such that $d_{\A}(u_{j},u)\to 0$ as $j\to\infty$. We then have $L:=\sup_{j}\|u_{j}\|_{\sobo^{\A,1}(\Omega)}<\infty$. Since $\A$ is $\mathbb{C}$-elliptic, the Jones-type extension operator $\widetilde{E}$ from \cite[Sec.~4.1]{GR19} maps $\widetilde{E}\colon \widetilde{\sobo}{^{\A,1}}(\Omega)\to\widetilde{\sobo}{^{\A,1}}(\R^{n})$ boundedly; clearly, we may assume that all maps $\widetilde{E}u$ are compactly supported in a fixed ball $\ball(0,R)$. This operator moreover satisfies $\|\widetilde{E}u\|_{\lebe^{1}(\R^{n})}\leq c\|u\|_{\lebe^{1}(\Omega)}$. Thus, for all $\ell\in\{0,...,k-1\}$,
\begin{align}\label{eq:intermedCauchy}
\begin{split}
\|D^{\ell}(u_{j}-u_{m})\|_{\lebe^{1}(\Omega)} & \leq \|D^{\ell}(\widetilde{E}(u_{j}-u_{m}))\|_{\lebe^{1}(\R^{n})} \\ & \!\!\!\!\stackrel{\eqref{eq:interpol}}{\leq } c \|\widetilde{E}(u_{j}-u_{m})\|_{\lebe^{1}(\R^{n})}^{1-\frac{\ell}{k}}\|\A(\widetilde{E}(u_{j}-u_{m}))\|_{\lebe^{1}(\R^{n})}^{\frac{\ell}{k}}\\
& \leq c \|u_{j}-u_{m}\|_{\lebe^{1}(\Omega)}^{1-\frac{\ell}{k}}\|u_{j}-u_{m}\|_{\sobo^{\A,1}(\Omega)}^{\frac{\ell}{k}}\\ & \leq c L^{\frac{\ell}{k}}\|u_{j}-u_{m}\|_{\lebe^{1}(\Omega)}^{1-\frac{\ell}{k}}\to 0
\end{split}
\end{align}
as $j,m\to\infty$, and so with $d$ as defined above, $d(u_{j},u)\to\infty$. We pick a sequence $(\rho_{l})\subset\hold^{\infty}(\overline{\Omega};[0,1])$ such that $\rho_{l}=1$ in a $\tfrac{1}{l}$-neighbourhood of $\partial\Omega$ and $\rho_{l}=0$ in $\Omega\setminus \Omega'_{l}:=\{x\in\Omega\colon\dista(x,\partial\Omega)>\frac{2}{l}\}$.  We may then assume that $|\nabla^{\ell}\rho_{l}|\leq cl^{\ell}$ for some $c>0$ and all $\ell\in\{1,...,k\}$. For each $l\in\mathbb{N}$, we find a set $\Omega_{l}$ with $\spt(\rho_{l})\subset\Omega_{l}\subset\Omega'_{l}$  and of smooth boundary such that $|\A u|(\partial\Omega_{l})=0$. Denoting $\trace_{\partial\Omega}$ the trace operator on ${\sobo}{^{\A,1}}(\Omega)$ that is already our disposal, we then estimate 
\begin{align*}
\|\trace_{\partial\Omega}(u_{j})-\trace_{\partial\Omega}(u_{m})\|_{{\besov}{_{1,1}^{k-1}}(\partial\Omega)} & \leq c (\|\rho_{l}(u_{j}-u_{m})\|_{\lebe^{1}(\Omega)}+\|\A (\rho_{j}(u_{l}-u_{m}))\|_{\lebe^{1}(\Omega)}) \\ 
& \leq c\Big(\sum_{\substack{|\alpha|+|\beta|\leq k,\\ |\beta|\leq k-1}}\|\partial^{\alpha}\rho_{l}\partial^{\beta}(u_{j}-u_{m})\|_{\lebe^{1}(\Omega)}\Big)\\
& \;\;\;\;\;+ c\|\rho_{l}(\A u_{j}-\A u_{m})\|_{\lebe^{1}(\Omega)} =:\mathrm{I}_{l,j,m}+\mathrm{II}_{l,j,m}. 
\end{align*} 
For an arbitrary $l\in\mathbb{N}$, \eqref{eq:intermedCauchy} implies that $\mathrm{I}_{l,j,m}\to 0$ as $j,m\to\infty$. Therefore, sending $j,m\to\infty$ first and then letting $l\to\infty$, 
\begin{align*}
\lim_{j,m\to\infty} \|\trace_{\partial\Omega}(u_{j})-\trace_{\partial\Omega}(u_{m})\|_{{\besov}{_{1,1}^{k-1}}(\partial\Omega)} & \leq c|\A u|(\Omega_{l})  \to 0
\end{align*}
as $l\to\infty$; note that $|\A u|$ is a finite Radon measure and $\bigcap_{l}\Omega_{l}=\emptyset$. Hence $(\trace_{\partial\Omega}(u_{j}))$ is Cauchy in ${\besov}{_{1,1}^{k-1}}(\partial\Omega;V)$, too. Working from here, the claimed existence of a trace operator with the requisite properties is routine, and the proof thus is complete. 
\end{proof} 
\subsection{Oscillation characterisation of ${\dot{\besov}}{_{p,q}^{s}}$}\label{sec:oscchar}
 To the best of our knowledge, the oscillation chracterisation of Besov spaces is only available in the inhomogeneous situation, cf.~\textsc{Triebel} \cite[Chpt.~3.5]{Trie2}. For completeness, we briefly explain how to arrive at the part of Lemma~\ref{lem:Besovosc} that is required for Section~\ref{sec:proofhalfspace}. 

Let $M> \lfloor s \rfloor$ and choose $2N>M$. Let $k\in\hold_{c}^{\infty}(\R^{n})$ be such that $\spt(k)\subset\ball(0,1)$ and consider, for $f\in\mathscr{S}(\R^{n})$, the means 
\begin{align*}
k(t,f)(x):=\int_{\R^{n}}k(y)f(x+ty)\dif y,\qquad t>0,\;x\in\R^{n}. 
\end{align*}
Given some $k^{0}\in\hold_{c}^{\infty}(\R^{n})$ with $\spt(k^{0})\subset\ball(0,1)$ and $\widehat{k^{0}}(0)\neq 0$, we put $
k:=\Delta^{N}k^{0}$, where $\Delta^{N}$ stands for the $N$-fold application of the Laplacean. We now define $\varphi:=k^{\vee}$. It is not too difficult to verify that $\varphi$ satisfies nonsubstantial  modifications of \cite[(62),(78),(79)]{Trie88a} (also see \cite[Proof of Thm.~2.4.6]{Trie2}) and so, by a combination of \cite[Rem.~17, Rem.~18]{Trie88a},  
\begin{align}\label{eq:kernelrewrite0}
\|f\|_{{\dot\besov}{_{p,q}^{s}}(\R^{n})}^{\sim} := \Big(\int_{0}^{\infty}\|\mathscr{F}^{-1}[\varphi(t\cdot)\mathscr{F}f]\|_{\lebe^{p}(\R^{n})}^{q}\frac{\dif t}{t^{1+sq}}\Big)^{\frac{1}{q}}
\end{align}
is an equivalent norm on ${\dot\besov}{_{p,q}^{s}}(\R^{n})$. By definition of $\varphi$, we have 
\begin{align}\label{eq:kernelrewrite}
\|f\|_{{\dot\besov}{_{p,q}^{s}}(\R^{n})}^{\sim}=c(n)\Big(\int_{0}^{\infty}\|k(t,f)\|_{\lebe^{p}(\R^{n})}^{q}\frac{\dif t}{t^{1+sq}}\Big)^{\frac{1}{q}}.
\end{align}
Since $2N>M$, we have for any $\pi\in\mathscr{P}_{M}(\R^{n})$ that $\Delta^{N}\pi=0$ and so, integrating by parts and Jensen's inequality imply
\begin{align*}
|k(t,f)(x)|\leq\int_{\R^{n}}|k(y)(f-\pi)(x+ty)|\dif y \leq c(n,N,k^{0}) \Big(\dashint_{\ball(x,t)}|f-\pi|^{u}\dif y\Big)^{\frac{1}{u}}
\end{align*}
and infimising the very right hand side over $\pi\in\mathscr{P}_{M}(\R^{n})$ yields 
\begin{align}\label{eq:osckernel}
|k(t,f)(x)|\leq c(n) \osc_{u}^{M}f(x,t).
\end{align}
Combining \eqref{eq:osckernel} with \eqref{eq:kernelrewrite0} and \eqref{eq:kernelrewrite} consequently yields the existence of a constant $c=c(n,N,s,p,q,k^{0})>0$ such that 
\begin{align*}
\|f\|_{{\dot\besov}{_{p,q}^{s}}(\R^{n})}\leq c \Big(\int_{0}^{\infty}\|\osc_{u}^{M}f(\cdot,t)\|_{\lebe^{p}(\R^{n})}^{q}\frac{\dif t}{t^{1+sq}}\Big)^{\frac{1}{q}}\qquad\text{for all}\;f\in\mathscr{S}(\R^{n}).
\end{align*}
The statement for ${\dot\besov}{_{p,q}^{s}}(\R^{n};V)$ then follows by componentwise application.

\bibliographystyle{plain}
\bibliography{biblio}

\begin{thebibliography}{10}

\bibitem{AdoPip98}
V.~Adolfsson and J.~Pipher.
\newblock The inhomogeneous {D}irichlet problem for {$\Delta^{2}$} in
  {L}ipschitz domains.
\newblock {\em J.Funct. Anal.}, 159(1):137--190, 1998.

\bibitem{Aik12}
H.~Aikawa.
\newblock Potential analysis on nonsmooth domains---{M}artin boundary and
  boundary {H}arnack principle.
\newblock In {\em Complex analysis and potential theory}, volume~55 of {\em CRM
  Proc. Lecture Notes}, pages 235--253. Amer. Math. Soc., Providence, RI, 2012.

\bibitem{AmbFusPal00}
L.~Ambrosio, N.~Fusco, and D.~Pallara.
\newblock {\em Functions of bounded variation and free discontinuity problems}.
\newblock Oxford Mathematical Monographs. The Clarendon Press Oxford University
  Press, New York, 2000.

\bibitem{Aro54}
N.~Aronszajn.
\newblock On coercive integro-differential quadratic forms.
\newblock {\em Conference on Partial Differential Equations, University of
  Kansas, Technical Report}, (No. 14):94--106, 1954.

\bibitem{MR58:2349}
J.~Bergh and J.~L{\"o}fstr{\"o}m.
\newblock {\em Interpolation spaces. {A}n introduction}.
\newblock Springer-Verlag, Berlin, 1976.
\newblock Grundlehren der Mathematischen Wissenschaften, No. 223.

\bibitem{Boman82}
J.~Boman.
\newblock $\lebe^{p}$-estimates for very strongly elliptic systems.
\newblock {\em Technical Report 29, Technical Report 29, Department of
  Mathematics, University of Stockholm, Sweden, 1982}.

\bibitem{BouBre04}
J.~Bourgain and H.~Brezis.
\newblock New estimates for the {L}aplacian, the div-curl, and related {H}odge
  systems.
\newblock {\em C. R. Math. Acad. Sci. Paris}, 338(7):539--543, 2004.

\bibitem{BouBre07}
J.~Bourgain and H.~Brezis.
\newblock New estimates for elliptic equations and {H}odge type systems.
\newblock {\em J. Eur. Math. Soc. (JEMS)}, 9(2):277--315, 2007.

\bibitem{VS14b}
P.~Bousquet and J.~Van~Schaftingen.
\newblock Hardy–{S}obolev inequalities for vector fields and canceling linear
  differential operators.
\newblock {\em Indiana Univ. Math.}, 63(5):1419--1445, 2014.

\bibitem{BreDie12}
D.~Breit and L.~Diening.
\newblock Sharp conditions for {K}orn inequalities in {O}rlicz spaces.
\newblock {\em J. Math. Fluid Mech.}, 14:565--573, 2012.

\bibitem{BreDieGme20}
D.~Breit, L.~Diening, and F.~Gmeineder.
\newblock On the trace operator for functions of bounded {$\Bbb A$}-variation.
\newblock {\em Anal. PDE}, 13(2):559--594, 2020.

\bibitem{BreSco08}
S.C. Brenner and L.R. Scott.
\newblock {\em The mathematical theory of finite element methods}, volume~15 of
  {\em Texts in Applied Mathematics}.
\newblock Springer, New York, third edition, 2008.

\bibitem{BrezisPonce}
H.~Brezis and A.C. Ponce.
\newblock Kato's inequality up to the boundary.
\newblock {\em Commun. Contemp. Math.}, 210(6):1217--1241, 2008.

\bibitem{BucKosLu96}
S.~Buckley, P.~Koskela, and G.~Lu.
\newblock Boman equals {J}ohn.
\newblock {\em XVIth Rolf Nevanlinna Colloquium (Joensuu, 1995), de Gruyter,
  Berlin}, pages 91--99, 1996.

\bibitem{CZ56}
A.P. Calder\'on and A.~Zygmund.
\newblock On singular integrals.
\newblock {\em Amer. J. Math.}, 78:289--309, 1956.

\bibitem{Cianchi14}
A.~Cianchi.
\newblock Korn type inequalities in {O}rlicz spaces.
\newblock {\em J. Funct. Anal.}, 267:2313--2352, 2014.

\bibitem{CianchiPick98}
A.~Cianchi and L.~Pick.
\newblock Sobolev embeddings into {BMO}, {VMO}, and $\lebe^{\infty}$.
\newblock {\em Ark. Mat.}, 36:317--340, 1998.

\bibitem{CoiFef74}
R.~Coifman and C.~Fefferman.
\newblock Weighted norm inequalities for maximal functions and singular
  integrals.
\newblock {\em Studia Math.}, 51(3):241--250, 1974.

\bibitem{ConGme20}
S.~Conti and F.~Gmeineder.
\newblock $\mathscr{A}$-quasiconvexity and partial regularity.
\newblock {\em \href{https://arxiv. org/abs/2009.13820}{https://arxiv.
  org/abs/2009.13820}}.

\bibitem{CruMarPer11}
D.V. Cruz-Uribe, J.M. Martell, and C.~P\'{e}rez.
\newblock {\em Weights, extrapolation and the theory of {R}ubio de {F}rancia},
  volume 215 of {\em Operator Theory: Advances and Applications}.
\newblock Birkh\"{a}user/Springer Basel AG, Basel, 2011.

\bibitem{DieGme20}
L.~Diening and F.~Gmeineder.
\newblock Continuity points via riesz potentials for $\mathbb{C}$-elliptic
  operators.
\newblock {\em The Quarterly Journal of Mathematics}, 71(4):1201--1218, 2020.

\bibitem{DieHarHasRuz11}
L.~Diening, P.~Harjulehto, P.~H\"{a}st\"{o}, and M.~R{\r u}\v{z}i\v{c}ka.
\newblock {\em Lebesgue and {S}obolev spaces with variable exponents}, volume
  2017 of {\em Lecture Notes in Mathematics}.
\newblock Springer, Heidelberg, 2011.

\bibitem{DieRuzSch10}
L.~Diening, M.~R{\r u}\v{z}i\v{c}ka, and K.~Schumacher.
\newblock A decomposition technique for {J}ohn domains.
\newblock {\em Ann. Acad. Sci. Fenn. Math.}, 35(1):87--114, 2010.

\bibitem{Dor85}
J.R. Dorronsoro.
\newblock Mean oscillation and {B}esov spaces.
\newblock {\em Canad. Math. Bull.}, 28(3), 1985.

\bibitem{Duo00}
J.~Duoandikoetxea.
\newblock {\em Fourier Analysis \emph{(translated and revised by David
  Cruz-Uribe, SFO)}}, volume~29 of {\em Graduate Studies in Mathematics}.
\newblock American Mathematica Society, 2001.

\bibitem{Eva98}
L.C. Evans.
\newblock {\em Partial differential equations}, volume~19 of {\em Graduate
  Studies in Mathematics}.
\newblock American Mathematical Society, 1998.

\bibitem{Fuc11}
M.~Fuchs.
\newblock An estimate for the distance of a complex valued {S}obolev function
  defined on the unit disc to the class of holomorphic functions.
\newblock {\em J. Applied Analysis}, 17(1):131--135, 2011.

\bibitem{FuchsSeregin}
M.~Fuchs and G.~Seregin.
\newblock Variational methods for problems from plasticity theory and for
  generalized {N}ewtonian fluids.
\newblock {\em Ann. Univ. Sarav. Ser. Math.}, 10(1):iv+283, 1999.

\bibitem{Gal57}
E.~Gagliardo.
\newblock Caratterizzazione delle sulla frontiera relative ad alcune tracce
  classi di funzioni in $n$ variabili.
\newblock {\em Rend. Sem. Mat. Univ. Padova}, 27:284--305, 1957.

\bibitem{GR19}
F.~Gmeineder and B.~Raita.
\newblock Embeddings for $\mathbb{A}$-weakly differentiable functions on
  domains.
\newblock {\em J. Func. Anal.}, 277(12):108278, 2019.

\bibitem{GRVS}
F.~Gmeineder, B.~Raita, and J.~Van~Schaftingen.
\newblock On limiting trace inequalities for vectorial differential operators.
\newblock {\em To appear at Indiana Math. Univ. J.}

\bibitem{HeiKilMar}
J.~Heinonen, T.~Kilpel\"{a}inen, and O.~Martio.
\newblock {\em Nonlinear {P}otential {T}heory of {D}egenerate {E}lliptic
  {E}quations}.
\newblock Dover Books on Mathematics. Dover Publications, 2006.
\newblock Reprint of the 1993 edition published by Oxford University Press.

\bibitem{HofMitTay10}
S.~Hofmann, M.~Mitrea, and M.~Taylor.
\newblock Singular integrals and elliptic boundary problems on regular
  {S}emmes-{K}enig-{T}oro domains.
\newblock {\em Int. Math. Res. Not. IMRN}, (14):2567--2865, 2010.

\bibitem{Hoermander1}
L.~H\"{o}rmander.
\newblock Differentiability properties of solutions of systems of differential
  equations.
\newblock {\em Ark. Mat.}, 3:527--535, 1958.

\bibitem{Hurri88}
R.~Hurri.
\newblock {P}oincar\'{e} domains in $\mathbb{R}^{n}$.
\newblock {\em Ann. Acad. Sci. Fenn. Ser. A I}, 71:1--42, 1988.

\bibitem{Iwa82}
T.~Iwaniec.
\newblock On {$L^{p}$}-integrability in {PDE}s and quasiregular mappings for
  large exponents.
\newblock {\em Ann. Acad. Sci. Fenn. Ser. A I Math.}, 7(2):301--322, 1982.

\bibitem{IwaNol85}
T.~Iwaniec and C.~A. Nolder.
\newblock Hardy-{L}ittlewood inequality for quasiregular mappings in certain
  domains in {${\bf R}^n$}.
\newblock {\em Ann. Acad. Sci. Fenn. Ser. A I Math.}, 10:267--282, 1985.

\bibitem{JerisonKenig}
D.~Jerison and C.~Kenig.
\newblock Boundary behaviour of harmonic functions in non-tangentially
  accessible domains.
\newblock {\em Adv. Math.}, 47(1):80--147, 1982.

\bibitem{Jiang17}
R.~Jiang and A.~Kauranen.
\newblock Korn's inequality and {J}ohn domains.
\newblock {\em J. Calc. Var.}, 56:109, 2017.

\bibitem{John61}
F.~John.
\newblock Rotation and strain.
\newblock {\em Comm. Pure Appl. Math.}, 14:391--413, 1961.

\bibitem{Kal93}
A.~Ka{\l}amajska.
\newblock Coercive inequalities on weighted {S}obolev spaces.
\newblock {\em Colloq. Math.}, LXVI(2):309--318, 1993.

\bibitem{Kal94}
A.~Ka{\l}amajska.
\newblock Pointwise multiplicative inequalities and {N}irenberg type estimates
  in weighted {S}obolev spaces.
\newblock {\em Studia Math.}, 108(3):275--290, 1994.

\bibitem{KKM99}
J.~Kauhanen, P.~Koskela, and J.~Mal\'{y}.
\newblock On functions with derivatives in a {L}orentz space.
\newblock {\em Manuscripta Math.}, 100:87--101, 1999.

\bibitem{KSW17}
Stolyarov~D.M. Kazaniecki, K. and M.~Wojciechowski.
\newblock Anisotropic {O}rnstein non- inequalities.
\newblock {\em Analysis \& PDE}, 10(2), 2017.

\bibitem{KiKr16}
B.~Kirchheim and J.~Kristensen.
\newblock On rank one convex functions that are homogeneous of degree one.
\newblock {\em Arch. Ration. Mech. Anal.}, 221(1), 1985.

\bibitem{Korn09}
A.~Korn.
\newblock \"{U}ber einige {U}ngleichungen, welche in der {T}heorie der
  elastischen und elektrischen {S}chwingungen eine {R}olle spielen.
\newblock {\em Classe des {S}ciences {M}ath\'{e}matiques et {N}aturels}, (9.
  Novembre):705--724, 1909.

\bibitem{LeoniTice}
G.~Leoni and I.~Tice.
\newblock Traces for homogeneous {S}obolev spaces in infinite strip-like
  domains.
\newblock {\em J. Funct. Anal.}, 277(7):2288--2380, 2019.

\bibitem{MarSar78}
O.~Martio and J.~Sarvas.
\newblock Injectivity theorems in plane and space.
\newblock {\em Ann. Acad. Sci. Fenn. Ser. A I Mth.}, 4:383–401, 1978--79.

\bibitem{Maz85}
V.G. Maz'ya.
\newblock {\em Sobolev spaces}.
\newblock Springer Series in Soviet Mathematics. Springer-Verlag, Berlin, 1985.
\newblock Translated from the Russian by T. O. Shaposhnikova.

\bibitem{MazMitSha05}
V.G. Maz'ya, M.~Mitrea, and T.~Shaposhnikova.
\newblock The {D}irichlet problem in {L}ipschitz domains with boundary data in
  {B}esov spaces for higher order elliptic systems with rough coefficients.
\newblock {\em arXiv.math.AP/0505372}.

\bibitem{Miller82}
N.~Miller.
\newblock Weighted {S}obolev spaces and pseudodifferential operators with
  smooth symbols.
\newblock {\em Trans. Amer. Math. Soc.}, 269(1):91--109, 1982.

\bibitem{MirRus15}
P.~Mironescu and E.~Russ.
\newblock Traces of weighted {S}obolev spaces. {O}ld and new.
\newblock {\em Nonlinear Analysis TMA}, 119:354--381, 2015.

\bibitem{Ornstein}
D.~Ornstein.
\newblock A non-equality for differential operators in the {$L_{1}$}-norm.
\newblock {\em Arch. Rational Mech. Anal.}, 11:40--49, 1962.

\bibitem{Raita19}
B.~Raita.
\newblock Critical $\lebe^{p}$-differentiability of $\bv^{\A}$-maps and
  canceling operators.
\newblock {\em Transactions of the American Mathematical Society},
  372(10):7297--7326, 2019.

\bibitem{Resh70}
Yu.~G. Reshetnyak.
\newblock Estimates for certain differential operators with finite-dimensional
  kernel. ({R}ussian).
\newblock {\em Sibirsk. Mat. Z.}, 11:414--428, 1970.

\bibitem{Smith70}
K.~T. Smith.
\newblock Formulas to represent functions by their derivatives.
\newblock {\em Math. Ann.}, 188:53--77, 1970.

\bibitem{Spencer}
D.~C. Spencer.
\newblock Overdetermined systems of linear partial differential equations.
\newblock {\em Bull. Amer. Math. Soc.}, 75:179--239, 1969.

\bibitem{SteinLorentz}
E.M. Stein.
\newblock Editor's note: {T}he differentiability of functions in
  $\mathbb{R}^{n}$.
\newblock {\em Annals of Mathematics}, 113:383--385, 1981.

\bibitem{Ste93}
E.M. Stein.
\newblock {\em Harmonic analysis: real-variable methods, orthogonality, and
  oscillatory integrals}, volume~43 of {\em Princeton Mathematical Series}.
\newblock Princeton University Press, Princeton, NJ, 1993.
\newblock With the assistance of Timothy S. Murphy, Monographs in Harmonic
  Analysis, III.

\bibitem{Tar07}
L.~Tartar.
\newblock {\em An {I}ntroduction to {S}obolev {S}paces and {I}nterpolation
  {S}paces}, volume~3 of {\em Lecture Notes of the Unione Matematica Italiana}.
\newblock Springer Science \& Business Media, 2007.

\bibitem{Trie1}
H.~Triebel.
\newblock {\em Theory of function spaces}, volume~78 of {\em Monographs in
  Mathematics}.
\newblock Birkh\"auser Verlag, Basel, 1983.

\bibitem{Trie88a}
H.~Triebel.
\newblock Characterizations of {B}esov-{H}ardy-{S}obolev spaces: {A} unified
  approach.
\newblock {\em J. Approx. Theory}, 52:162--203, 1988.

\bibitem{Trie2}
H.~Triebel.
\newblock {\em Theory of function spaces. {II}}, volume~84 of {\em Monographs
  in Mathematics}.
\newblock Birkh\"auser Verlag, Basel, 1992.

\bibitem{Uspenskii}
S.V. Uspenski\u{\i}.
\newblock Imbedding theorems for classes with weights.
\newblock {\em Tr. Mat. Inst. Steklova}, 60:282--303, 1961.
\newblock Am. Math. Soc. Transl. 87 (1970) 121--145. English translation.

\bibitem{Vai94}
J.~V\"{a}is\"{a}l\"{a}.
\newblock Exhaustion of {J}ohn domains.
\newblock {\em Ann. Acad. Sci. Fenn., Series A. I. Mathematica}, 19:47--57,
  1994.

\bibitem{VS04}
J.~Van~Schaftingen.
\newblock Estimates for $\lebe^{1}$-vector fields.
\newblock {\em C. R. Math. Acad. Sci. Paris}, 339, 2004.

\bibitem{Van13}
J.~Van~Schaftingen.
\newblock Limiting {S}obolev inequalities for vector fields and canceling
  linear differential operators.
\newblock {\em J. Eur. Math. Soc. (JEMS)}, 15(3):877--921, 2013.

\bibitem{VS14a}
J.~Van~Schaftingen.
\newblock Limiting {B}ourgain–{B}rezis estimates for systems of linear
  differential equations: {T}heme and variations.
\newblock {\em J. Fixed Point Theory Appl.}, 15, 2014.

\end{thebibliography}

\end{document}